\documentclass[reqno,12pt,letterpaper]{amsart}
\usepackage[utf8]{inputenc}
\usepackage{amsmath}
\usepackage{amsfonts}
\usepackage{amssymb}
\usepackage{amsthm}
\usepackage[dvipsnames]{xcolor}
\usepackage{enumitem}
\usepackage{mathtools}
\usepackage{mathrsfs}
\usepackage{hyperref}
\usepackage{subcaption}
\usepackage{graphicx}

\hypersetup{colorlinks = true,
            citecolor=NavyBlue,
            linkcolor=Green}

\newtheorem{proposition}{Proposition}[section]
\newtheorem{theorem}{Theorem}
\newtheorem{lemma}[proposition]{Lemma}

\theoremstyle{definition}
\newtheorem{definition}[proposition]{Definition}
\theoremstyle{remark}
\newtheorem{remark}[proposition]{Remark}

\newcommand{\n}[1]{\left\|#1\right\|}
\newcommand{\set}[1]{\left\{#1\right\}}
\newcommand{\brac}[1]{\langle #1 \rangle}

\DeclareMathOperator{\im}{Im}
\DeclareMathOperator{\re}{Re}

\DeclareMathOperator{\sign}{sign}

\numberwithin{equation}{section}

\let\oldtocsection=\tocsection
\let\oldtocsubsection=\tocsubsection
\renewcommand{\tocsection}[2]{\hspace{0em}\oldtocsection{#1}{#2}}
\renewcommand{\tocsubsection}[2]{\hspace{1em}\oldtocsubsection{#1}{#2}}

\title[viscous internal waves]{Stationary internal waves in a two-dimensional aquarium at low viscosity}

\author{Malo J\'ez\'equel}
\email{malo.jezequel@math.cnrs.fr}
\address{CNRS, Univ. Brest, UMR6205, Laboratoire de Math{\'e}matiques de Bretagne Atlantique, France}

\author{Jian Wang}
\email{wangjian@ihes.fr}
\address{Institut des Hautes {\'E}tudes Scientifiques, Bures-sur-Yvette, France}

\begin{document}

\begin{abstract}
    We prove the uniform solvability of a stationary problem associated to internal waves equation with small viscosity in a two dimensional aquarium with real-analytic boundary, under a Morse--Smale dynamical assumption. This is achieved by using complex deformations of the aquarium, on which the inviscid stationary internal wave operator is invertible.
\end{abstract}

\maketitle

\thispagestyle{empty}

{
\hypersetup{linkcolor=NavyBlue}
\tableofcontents
}

\newpage

\section{Introduction}

\subsection{Settings and main results}

Let $\Omega$ be a bounded open subset in $\mathbb{R}^2$ with real-analytic boundary. To understand forced internal waves in $\Omega$ in the presence of viscosity, we consider the equation \cite[(1.20), (1.21)]{brouzet2016}
\begin{equation}\label{eq:evolution_equation}
    \begin{cases}
    \partial_t^2 \Delta u + \partial_{x_2}^2 u - \nu \partial_t \Delta^2 u = f(x) \cos{\lambda t}, \\
    u_{|\partial \Omega} = 0, \\
    (\mathrm{d}u)_{|\partial \Omega} = 0 \textup{ if } \nu > 0, \\
    u_{|t = 0} = 0.
    \end{cases}
\end{equation}
Here, $\nu > 0$ is a viscosity parameter and $\lambda \in (0,1)$ is the forcing frequency. The function $f$ is real-valued and smooth and $u$ is the unknown. In order to understand the long-time asymptotic of the solutions of this equation, it is natural to apply inverse Fourier transform in time\footnote{We apply the inverse Fourier transform instead of a Fourier transform to get an operator invertible for $\omega \in (0,1) + i (0, + \infty)$, see Theorem \ref{theorem:main}, and then study what happens when we cross $(0,1)$ from above, which is standard in scattering theory. It also has the advantage to make signs in our intermediate results coherent with those in \cite{dwz_internal_waves}.} and introduce the operator
\begin{equation*}
    P_{\omega,\nu} = - \omega^2 \partial_{x_1}^2 + (1 - \omega^2) \partial_{x_2}^2 + i \omega \nu \Delta^2.
\end{equation*}
Here, $\omega \in \mathbb{C}$ is the frequency parameter in time, and we want to understand the invertibility properties of $P_{\omega,\nu}$ when $\omega$ is close to $\lambda$ and $\nu > 0$ is small. When $\nu > 0$, we will supplement $P_{\omega,\nu}$ with both Dirichlet and Neumann boundary conditions. Setting $\nu = 0$, we retrieve the operator $P_{\omega,0}$ that is studied in \cite{dwz_internal_waves} for which we work only with Dirichlet boundary condition. Recall \cite[Theorem 2]{ralston73} that, if $\Omega$ is simply connected, $P_{\omega,0} : H_0^1(\Omega) \to H^{-1}(\Omega)$ is invertible for\footnote{In the following, we will discuss only what happens when $\omega$ is near a point $\lambda \in (0,1)$. Results for $\omega$ near points of $(-1,0)$ are deduced immediately using the invariance of $P_{\omega,\nu}$ under $\omega \mapsto - \omega$.} $\omega \in \mathbb{C} \setminus [-1,1]$.

Under some geometric and dynamical assumptions on $\Omega$ (see Definition~\ref{definition:morse_smale}), Dyatlov, Wang and Zworski \cite{dwz_internal_waves} described the asymptotic of the solutions to \eqref{eq:evolution_equation} in the case $\nu = 0$ (in that case, Neumann boundary condition is dropped). More precisely, let $\Omega$ be a bounded open subset of $\mathbb{R}^2$ with $C^\infty$ boundary and $\lambda \in (0,1)$ satisfy the Morse--Smale condition (Definition \ref{definition:morse_smale}). Then for $f \in C_c^\infty(\Omega;\mathbb R)$, according to \cite[Theorem 1.3]{dwz_internal_waves}, the solution $u$ to \eqref{eq:evolution_equation} with $\nu = 0$ has the decomposition
\begin{equation*}
    u(t) = \re\left( e^{- i \lambda t} u_+ \right) + r(t) + e(t),
\end{equation*}
where $r(t)$ is bounded in $H^1(\Omega)$ uniformly in $t \geq 0$, the term $e(t)$ goes to $0$ in $H^{\frac{1}{2}-}(\Omega)$ as $t$ goes to $+ \infty$ and $u_+ \in H^{\frac{1}{2}-}(\Omega)$.
The leading singular term $u_+$ is called the {\em internal wave attractor} and is of particular interest. It is shown that $u_+$ belongs to a space of Lagrangian distributions defined in terms of the dynamics of a chess billiard map. Furthermore, $u_+$ is a particular solution to the stationary equation $P_{\lambda,0}u_+=f$ given by the limiting distribution 
\[ u_+ = P_{\lambda+ i0, 0}^{-1}f\coloneqq \lim\limits_{\epsilon \to 0^+} P_{\lambda + i \epsilon,0}^{-1}f. \]
The existence of the limit in the space of distributions $\mathcal D'(\Omega)$ is guaranteed by the {\em limiting absorption principle} established in \cite[Theorem 1.4]{dwz_internal_waves}.

Results of \cite{dwz_internal_waves} provide microlocal descriptions of internal wave attractors in the inviscid case. A natural question is to include viscosity in the equation and consider its dissipation effect on the inviscid attractors. This leads us to consider solutions to the stationary equation $P_{\omega,\nu}u_{\omega,\nu}=f$ with both Dirichlet and Neumann boundary conditions. The main theorem of this paper proves the {\em uniform} solvability of this stationary internal wave equation with viscosity.

\begin{theorem}\label{theorem:main}
Suppose $\Omega$ is a bounded open subset of $\mathbb{R}^2$ with real-analytic boundary and $\lambda \in (0,1)$ satisfies the Morse--Smale condition (Definition \ref{definition:morse_smale}). Then there is $\delta > 0$ such that for every $\omega \in (\lambda - \delta, \lambda + \delta) + i (- \delta,+ \infty)$ and $\nu \in (0,\delta)$ the operator $P_{\omega,\nu} : H^4(\Omega) \cap H_0^2(\Omega) \to L^2(\Omega)$ is invertible.
\end{theorem}

Notice that the inviscid operator $P_{\lambda,0}: H^2(\Omega)\cap H_0^1(\Omega)\to L^2(\Omega)$ is {\em not} invertible for $\lambda\in (0,1)$.\footnote{Otherwise, there exists $\lambda\in (0,1)$ and $C>0$ such that $\|\Delta u\|_{L^2(\Omega)}\leq C\|P_{\lambda,0}u\|_{L^2(\Omega)}$ for all $u\in H^2(\Omega)\cap H_0^1(\Omega)$. Taking $u_n(x) = \psi(x) e^{i\langle n,x\rangle}$ with $\psi\in C_c^{\infty}(\Omega)$ non-zero, $n\in \mathbb Z^2$, $|n|\to +\infty$ and $\frac{n_2^2}{|n|^2}\to \lambda^2$ then gives a contradiction. }
Therefore one can be surprised that the range of $\omega$ in Theorem \ref{theorem:main} contains complex numbers with non-positive imaginary parts and does {\em not} depend on $\nu$ once the value of $\nu$ gets sufficiently small.

Theorem \ref{theorem:main} should be compared with \cite[Theorem 1]{galkowski_zworski_2022} by Galkowski and Zworski, where microlocal models of $P_{\omega,\nu}\Delta_{\Omega}^{-1}$ on torus are considered. Here $\Delta_{\Omega}^{-1}$ is the inverse of $\Delta: H^2(\Omega)\cap H_0^1(\Omega)\to L^2(\Omega)$. In \cite{galkowski_zworski_2022} the inviscid operators can have embedded eigenvalues (see \cite[Appendix B.3]{galkowski_zworski_2022}), hence instead of showing the uniform invertibility, they showed eigenvalues of the operator with viscosity converges to resonances for the inviscid operator (see \cite{jw_orr_sommerfeld} for a similar statement in the context of Orr--Sommerfeld equation). In our case, \cite[Theorem 1.4]{dwz_internal_waves} implies the absence of embedded eigenvalues near $\lambda\in (0,1)$ satisfying the Morse--Smale condition. Later in Theorem \ref{theorem:existence_deformation}, we show the absence of ``resonances'' in complex neighbourhoods of such $\lambda$. Consequently, in our case, the operator $P_{\omega,\nu}$ with viscosity is uniformly invertible in a complex neighbourhood of $\lambda$.

The geophysical and astrophysical interests of Theorem \ref{theorem:main} come in two ways. First, in physics literature, in order to understand the long time evolution of solutions to \eqref{eq:evolution_equation}, stationary solutions of the form $u(t,x)=\re(u_{\lambda,\nu}(x)e^{-i\lambda t})$ are searched for, ignoring the initial condition. See the discussion by Ogilvie in \cite[\S 2]{Ogilvie_2005}. The profile $u_{\lambda,\nu}$ then must satisfy the stationary equation $P_{\lambda,\nu}u_{\lambda,\nu}=f$ with Dirichlet and Neumann boundary conditions. In view of \cite[Theorem 1.3 and 1.4]{dwz_internal_waves}, the solution $u_{\lambda,\nu}$ can be regarded as internal or inertial wave attractors with viscosity. A major purpose of \cite{Ogilvie_2005} is to construct formal asymptotics of $u_{\lambda,\nu}$ for a toy-model in the vanishing viscosity limit. Theorem \ref{theorem:main} (and its proof) justifies the existence and uniqueness of thus defined internal or inertial wave attractors at low viscosity with uniform estimates in complex deformed spaces, under the Morse--Smale condition. Second, for $\nu>0$, the complex values of $\omega$ such that $P_{\omega,\nu}$ is not invertible can be regarded as ``eigenvalues for internal waves''. A related notion of ``eigenvalues for inertial waves'' can be formulated in a very similar way, see for instance the work by Rieutord, Valdettaro and Georgeot \cite[(3.13)]{RIEUTORD_VALDETTARO_GEORGEOT_2002}. The distribution of such eigenvalues on the complex plane is rather complicated and of great physical interest. Theorem \ref{theorem:main} shows, uniformly as $\nu\to 0^+$, the absence of eigenvalues in complex neighbourhoods of $\lambda\in (0,1)$ satisfying the Morse--Smale condition.

As already mentioned above, Theorem \ref{theorem:main} is a consequence of the absence of ``resonances'' for the inviscid operator. In other words, one needs to show the existence of holomorphic continuations of the resolvent $P_{\omega,0}^{-1}$, $\im\omega>0$, across $(0,1)$ near $\lambda$.
In order to provide such a continuation, we will rely on a complex deformation method (see the end of \S \ref{subsection:context} for references) in order to turn $P_{\lambda,0}$ into an elliptic operator. We search for a totally real surface with boundary $\widetilde{\Omega} \subseteq \mathbb{C}^2$ (typically a small deformation of $\overline{\Omega}$) such that $P_{\lambda,0} : H^2(\widetilde{\Omega})\cap H_0^1(\widetilde{\Omega}) \to L^2(\widetilde{\Omega})$ is invertible. How $P_{\lambda,0}$ defines a differential operator on $\widetilde{\Omega}$ is explained in \S \ref{section:differential_operators}. However, even if the action of $P_{\lambda,0}$ on $\widetilde{\Omega}$ is always well-defined, we still need to make sense of the boundary conditions we are working with. In order to do so, recall that, since $\partial \Omega$ is a real-analytic submanifold of dimension $1$ in $\mathbb{R}^2$, it admits a complexification $(\partial \Omega)_{\mathbb{C}}$, that may be seen as a connected complex submanifold of dimension one in $\mathbb{C}^2$ containing $\partial \Omega$. In order to make sense of boundary conditions on $\widetilde{\Omega}$, we impose that $\partial \widetilde{\Omega} \subseteq (\partial \Omega)_{\mathbb C}$. Hence, if $u$ is a holomorphic function on a connected open subset of $\mathbb{C}^2$ that contains $\Omega$ and $\widetilde{\Omega}$ then u satisfies Dirichlet boundary condition on $\Omega$ if and only if it satisfies the condition on $\widetilde{\Omega}$ (this fact is crucial for an argument that appears in the proof of Lemma \ref{lemma:positive_imaginary_part_invertibility}, and that is reused in the proofs of Proposition \ref{proposition:inviscid_invertibility} and Lemma \ref{lemma:invariance_invertibility}). The proof of Theorem \ref{theorem:main} is based on the construction of a surface satisfying the conditions above:

\begin{theorem}\label{theorem:existence_deformation}
Suppose $\Omega$ is a bounded open subset of $\mathbb{R}^2$ with real-analytic boundary and $\lambda \in (0,1)$ satisfies the Morse--Smale condition (Definition \ref{definition:morse_smale}). Then there is $\delta > 0$ and a totally real real-analytic surface with boundary $\widetilde{\Omega} \subseteq \mathbb{C}^2$ such that:
\begin{enumerate}[label=(\roman*)]
    \item the boundary of $\widetilde{\Omega}$ is contained within the complexification $(\partial \Omega)_{\mathbb C}$ of the boundary of $\Omega$;\label{item:boundary_condition}
    \item for every $\omega \in (\lambda - \delta,\lambda + \delta) + i (-\delta,+ \infty)$, the operator $P_{\omega,0} : H^2(\widetilde{\Omega}) \cap H_0^1(\widetilde{\Omega}) \to L^2(\widetilde{\Omega})$ is invertible.\label{item:invertibility_deformation}
\end{enumerate}
\end{theorem}

\begin{remark}
Since $\partial \Omega$ is contained in $\mathbb{R}^2$, we may assume that $(\partial \Omega)_{\mathbb{C}}$ is invariant under complex conjugacy. Hence, if we replace $\widetilde{\Omega}$ by its image by the complex conjugacy, then~\ref{item:boundary_condition} from Theorem \ref{theorem:existence_deformation} still holds, but in \ref{item:invertibility_deformation} the set $(\lambda - \delta,\lambda + \delta) + i (-\delta,+ \infty)$ is replaced by  $(\lambda - \delta,\lambda + \delta) + i (- \infty, \delta)$. See Remark \ref{remark:complex_conjugacy}.
\end{remark}

\begin{remark}\label{remark:long_time_evolution}
Unfortunately, we were unable to prove any concrete results on the long-time asymptotics of the solutions of \eqref{eq:evolution_equation} using Theorem \ref{theorem:main}. Working with a toy-model for \eqref{eq:evolution_equation} on a closed manifold, Frantz and Wrochna \cite{frantz_wrochna_2025} are able to describe the solutions at a time $t \sim \nu^{- \frac{1}{3} - }$. However, there are two difficulties that forbid us to apply the methods of Frantz and Wrochna in our context.

The first one is the particular role played by $\omega = 0$ in our setting, due to the presence of a factor $\omega$ in the viscosity term in $P_{\omega,\nu}$. Hence, $P_{0,\nu}$ is not elliptic, even when $\nu > 0$, so a priori the sets of $\omega$ for which $P_{\omega,\nu}$ is not invertible could accumulate at zero (for a fixed $\nu > 0$). This feature is impossible for the class of perturbations considered in \cite{frantz_wrochna_2025} (the spectral parameter does not appear in the viscosity term there), which makes possible to describe accurately the solutions of their evolution problems using integration of the resolvent on a well-chosen contour. Notice however that this difficulty does not appear when working with the equation \eqref{eq:sixth_order} discussed below.

The second issue is due to the loss of selfadjointness. To study the inviscid problem, one can use the functional calculus associated to the self-adjoint operator $\partial_{x_2}^2 \Delta_{\Omega}^{-1}$ on $H^{-1}(\Omega)$, see \cite{dwz_internal_waves}. Here, $\Delta_{\Omega}^{-1}$ denote the inverse of the Dirichlet Laplacian on $\Omega$. Introducing $\nu > 0$, the operator $\partial_{x_2}^2 \Delta_{\Omega}^{-1}$ is replaced by a non-self-adjoint operator. This issue is dealt with in \cite{frantz_wrochna_2025} by using the functional calculus associated to the operator $P$ that plays the role of $\partial_{x_2}^2 \Delta_{\Omega}^{-1}$ in their context. They replace the forcing term $f$ by $\varphi(P)f$ where $\varphi$ is a smooth function supported near the forcing frequency. It is important in their argument that, starting with $f$ smooth ($C^\infty$), the function $\varphi(P)f$ is still smooth. This is true in their context because $\varphi(P)$ is a pseudodifferential operator of order $0$. However, to apply a similar argument in our context, we would need $\varphi(P)f$ to be real-analytic. It is unlikely to be true in general, unless $\varphi$ is real-analytic, but then it cannot have a small support.

Let us mention that the loss of selfadjointness is central in our method: the operator $\partial_{x_2}^2 \Delta_{\Omega}^{-1}$ is a priori not self-adjoint on the surface $\widetilde{\Omega}$ from Theorem \ref{theorem:existence_deformation}. However, the complex deformation method greatly simplifies all other aspects of the problem, see Remark \ref{remark:bc_deformation}.
\end{remark}

\subsection{Context}\label{subsection:context}

With $\partial_{x_2}^2$ replaced by $\partial_{x_1}^2$, equation \eqref{eq:evolution_equation} describes the time evolution of two dimensional internal waves, see the derivation in \cite[(1.20), (1.21)]{brouzet2016} and \cite[\S 2]{Ogilvie_2005}.
A similar equation can be used to approximate the time evolution of inertial waves in rotating fluids. To see this we recall the following linear approximation for three dimensional inertial waves used by Rieutord and Valdettaro \cite[(2.1)]{Rieutord_Valdettaro_2018}
\[\begin{cases}
    & \frac{\partial \mathbf u}{\partial t} + \mathbf k \times \mathbf u = -\nabla p+E\Delta\mathbf u, \\
    & \nabla \cdot \mathbf u=0
\end{cases}\]
where $\mathbf k$ is the axis of rotation, $p$ is the pressure and $E$ is Ekman number which is proportional to the viscosity. Following Ralston \cite{ralston73} we assume $\mathbf k=(0,1,0)$, $\mathbf u$ is independent of $x_3$ and $\mathbf u = (\partial_{x_2}\phi, \ -\partial_{x_1}\phi, u_3)$ for some stream function $\phi$. Then the system for $\mathbf u$ can be rewritten in terms of $\phi$ as 
\begin{equation}\label{eq:sixth_order}
(\partial_t-E\Delta)^2\Delta \phi + \partial_{x_2}^2\phi=0. 
\end{equation}
The left-hand-side of this equation only differs from that of \eqref{eq:evolution_equation} by an $\mathcal O(E^2)$ term. Our method also gives an analogue of Theorem \ref{theorem:main} for the equation \eqref{eq:sixth_order}, see Remark~\ref{remark:sixth_order}.

An important astrophysical question for inertial waves is to describe the distribution of eigenvalues at low Ekman numbers, as the eigenfunctions correspond to time-stationary solutions to the evolution equation. This is a central topic in works of Rieutord et al \cite{RIEUTORD_GEORGEOT_VALDETTARO_2001, RIEUTORD_VALDETTARO_GEORGEOT_2002, Rieutord_Valdettaro_2018, He_Favier_Rieutord_Le_Dizès_2023} in the context of spherical shells. It is found in numerical studies for spherical shells that for low Ekman numbers, eigenvalues for inertial waves can accumulate at $\lambda$ such that the corresponding chess billiard for $\lambda\in (0,1)$ has attractors with vanishing Lyapunov exponents. See \cite[\S 3.2.2]{moser2024} for an explicit construction of such chess billiard attractors in a two dimensional convex domain. On the contrary, if $\lambda\in (0,1)$ determines attractors with non-vanishing Lyapunov exponents, then in a complex neighborhood of $\lambda$, there is no inertial wave eigenvalues in the vanishing Ekman number limit. See \cite[Figure 3]{Rieutord_Valdettaro_2018} for a numerical illustration for both accumulation and absence of inertial eigenvalues. Our paper gives rigorous mathematical justification to the later case, in the context of two dimensional convex domains.

Two dimensional {\em inviscid} internal and inertial waves in domains {\em with} boundaries have attracted many research interests in diverse settings \cite{dwz_internal_waves, Colin_Li_2024, Li_2025, lww2024, lww2025}, where spectral theory for zeroth order internal wave operators, in particular limiting absorption principles, are established. For ideal microlocal models {\em without} boundary, which were first introduced by Colin de Verdi\`ere and Saint-Raymond \cite{cdv_sr_2020}, Galkowski and Zworski \cite{galkowski_zworski_2022} studied the vanishing viscosity limits of the eigenvalues. The work \cite{galkowski_zworski_2022} of Galkowski and Zworski results in defining scattering resonances for such microlocal models and showing the eigenvalues converge to resonances as the viscosity vanishes. For these microlocal models, Frantz and Wrochna \cite{frantz_wrochna_2025} studied the dissipation effect of viscosity in the long time evolution.

The problem considered in the present paper, in which both domain boundaries and viscosity are included, is more physically relevant. It is also more difficult mathematically: the perturbation of a differential operator by a higher order operator in the presence of a boundary may produce ``boundary layers'' (see \cite{vishik_ljusternik}) that forbid to get uniform a priori estimates in the vanishing viscosity limit, except maybe in spaces of relatively low regularity. The complex deformation method used in this paper is similar to \cite[Appendix B]{galkowski_zworski_2022} by Galkowski and Zworski, \cite{poincare_series_2025} by Dang, Guedes Bonthonneau, L\'eautaud and Rivi\`ere, and \cite{jw_orr_sommerfeld} by authors of the current paper. The complex deformation method also shares many similarities with the complex scaling method from scattering theory (see e.g. \cite{dyatlov_zworski_book}). Notice that \cite{galkowski_zworski_2022} and \cite{poincare_series_2025} deal with manifolds without boundaries; in \cite{jw_orr_sommerfeld}, the boundary consists of only two points and the complex deformations are designed to be away from the boundary. In the present paper, the domain boundary is a topological circle, and to reflect the chess billiard dynamics, complex deformations have to be performed on the whole domain, including the boundary.

\subsection{Strategy of proof}

The first idea to prove Theorem \ref{theorem:existence_deformation} is to choose $\widetilde{\Omega}$ as a small deformation of $\Omega$ such that $P_{\omega,0}$ is elliptic on $\widetilde{\Omega}$ when $\omega \in (\lambda- \delta,\lambda + \delta) + i (- \delta, + \infty)$. This is not too hard\footnote{One can just take the image of $\Omega$ by the map $(x_1,x_2) \mapsto (e^{- i \theta} x_1,x_2)$ for some small $\theta > 0$.}, the difficulty is to fulfill condition \ref{item:boundary_condition}. This is where we use the assumptions on $\lambda$ and $\Omega$. We start by constructing a deformation of $\partial \Omega$ within $(\partial \Omega)_{\mathbb{C}}$. This construction is motivated by dynamical considerations exposed in \S \ref{subsection:escape_function}. We will not present them here, as they require the definition of the chess billiards dynamics that we recall in \S \ref{subsection:generalities}. Once we designed a deformation of the boundary, we extend it to a deformation of the whole domain by interpolating between points of the boundary in \S \ref{subsection:new_domain}. 

Once we know that $P_{\omega,0}$ is elliptic on $\widetilde{\Omega}$ for $\omega \in (\lambda- \delta,\lambda + \delta) + i(-\delta, + \infty)$, we need to prove that it is invertible (with Dirichlet boundary condition). To do so, we use that $\widetilde{\Omega}$ is a deformation of $\Omega$. Hence, we have a smooth family\footnote{More precisely, we produce a family of deformations $(\overline{\Omega}_\tau)_{\tau \in [0,\tau_0)}$ and the prove that we can take $\widetilde{\Omega} = \overline{\Omega}_\tau$ for $\tau > 0$ small.} $(N_s)_{s \in [0,1]}$ of surfaces with boundary such that $N_0 = \overline{\Omega}$ and $N_1 = \widetilde{\Omega}$. Moreover, our construction gives that, if $\omega \in (\lambda - \delta, \lambda + \delta) + i (0, + \infty)$, then $P_{\omega,0}$ is elliptic on $N_s$ for every $s \in [0,1]$. Using elliptic regularity results in the real-analytic category (we work with \cite[Chapitre 8, Théorème 1.2]{lions_magenes_3}, but see also \cite{morrey_nirenberg_57}), we can see that if $u$ is in the kernel of $P_{\omega,0}$ acting on $\widetilde{\Omega} = N_1$, then $u$ has a holomorphic extension to a neighbourhood of $\bigcup_{s \in [0,1]} N_s$, and $u_{|N_0}$ is also in the kernel of $P_{\omega,0}$ (by the analytic continuation principle). An integration by parts argument proves then that $u_{|\Omega} = 0$ and thus $u = 0$. For this argument to work, we need $u$ to vanish on the boundary of $\Omega$, which follows from the analytic continuation principle, the fact that $u$ vanishes on the boundary of $\widetilde{\Omega}$ and the fact that the boundaries of $\Omega$ and $\widetilde{\Omega}$ are contained within the connected complex curve $(\partial \Omega)_{\mathbb{C}}$. 

The argument above proves the invertibility of $P_{\omega,0}$ on $\widetilde{\Omega}$ (with Dirichlet boundary condition) when $\omega \in (\lambda - \delta, \lambda + \delta) + i (0, + \infty)$. In order to conclude the proof of Theorem \ref{theorem:existence_deformation}, we want to prove that $P_{\lambda,0}$ is also invertible on $\widetilde{\Omega}$. In practice, we prove that it is invertible on $N_s$ for $s > 0$ small enough and then replace $\widetilde{\Omega}$ by $N_s$. To do so, we apply reduction to the boundary to $P_{\lambda,0}$ acting on $N_s$ (which is elliptic when $s > 0$). The boundary operator that appears degenerate as $s$ goes to $0$, but this degeneracy is similar to what happens in \cite{dwz_internal_waves} when their parameter ``$\epsilon = \im \omega$'' goes to $0$. Hence, we can rely on their analysis to prove invertibility, the key statement being \cite[Lemma 7.2]{dwz_internal_waves}. 

Once Theorem \ref{theorem:existence_deformation} is proven, we want to use the surface $\widetilde{\Omega}$ to prove Theorem~\ref{theorem:main}. Notice that when $\nu > 0$, the operator $P_{\omega,\nu}$ is elliptic for every $\omega \in \mathbb{C} \setminus \set{0}$. Hence, using the same argument as above based on real-analytic elliptic regularity and the analytic continuation principle, we find that, provided $\widetilde{\Omega}$ is a small enough deformation of $\overline{\Omega}$, the invertibility of $P_{\omega,\nu}$ (with Dirichlet and Neumann boundary conditions) on $\Omega$ and $\widetilde{\Omega}$ are equivalent. Consequently, we want to study $P_{\omega,\nu}$ on $\widetilde{\Omega}$ as $\nu$ goes to $0$, with the hope that the invertibility of $P_{\omega,0}$ on $\widetilde{\Omega}$ (proven in Theorem \ref{theorem:existence_deformation}) will imply the invertibility of $P_{\omega,\nu}$ for $\nu > 0$ small. 

The difficulty is that $P_{\omega,\nu}$ is a perturbation of $P_{\omega,0}$ of higher order. There is a literature dedicated to the study of such perturbations in the context of boundary value problems. A major reference is \cite{vishik_ljusternik}. We will rather rely on \cite{frank_coercive_singular} which contains a result that we can quote directly. The exposition from \cite{volevich_small_parameter} may be a good introduction to this topic. If perturbations of differential operators by higher order operators is already a non-trivial topic in the boundaryless case, it is more complicated in the presence of a boundary. For instance, in our case we have two boundary conditions for $P_{\omega,\nu}$ when $\nu$ goes to $0$, but only one at $\nu = 0$. The disappearance of the highest order boundary condition may be explained by the presence of boundary layers (functions in the kernel of $P_{\omega,\nu}$ that are concentrated near the boundary), see \cite{vishik_ljusternik}. This is why we can only get uniform a priori estimates for $P_{\omega,\nu}$ as $\nu$ goes to $0$ in spaces of relatively low regularity, see Proposition \ref{proposition:a_priori_estimate}. However, these estimates are sufficient to prove Theorem \ref{theorem:main}, using an elliptic regularity result proven in Appendix \ref{appendix:elliptic_regularity}.

\begin{remark}\label{remark:bc_deformation}
Let us mention that the complex deformation method not only make $P_{\lambda,0}$ elliptic but also improves the way boundary conditions for $P_{\lambda,\nu}$ behave as $\nu$ goes to $0$. Indeed, if one looks at the polynomial appearing in the Shapiro--Lopatinskii condition for $P_{\lambda,\nu}$ with Dirichlet and Neumann boundary conditions at a point of the boundary $\widetilde{\Omega}$ as $\nu$ goes to $0$, then it has two roots of order of magnitude $1$ and two roots of order of magnitude $\nu^{-\frac{1}{2}}$. The understanding of these roots is crucial in the analysis from \cite{frank_coercive_singular} (it is related to the presence of boundary layers of size $\nu^{\frac{1}{2}}$). If one works in $\Omega$ instead of $\widetilde{\Omega}$, one will see a similar phenomenon \emph{except at characteristic points of the boundary}. In that case, we have one root of size $1$ and three roots of size $\nu^{- \frac{1}{3}}$ (which could result in larger boundary layers). Notice that this phenomenon can only appear in the presence of a boundary, hence it is absent from \cite{frantz_wrochna_2025}. We expect this remark to be the source of major difficulties when trying to understand the viscosity limit $\nu \to 0$ directly on $\Omega$ (which would be natural to get a result in the $C^\infty$ category).
\end{remark}

\subsection{Structure of the paper}

We start by discussing the dynamical and geometric properties of $\Omega$ in relation with the equation \eqref{eq:evolution_equation}. In particular, we recall the definition of the chess billiards dynamics and state the properties of $\Omega$ required for Theorems~\ref{theorem:main} and~\ref{theorem:existence_deformation} in \S \ref{subsection:generalities}. Then, we construct the surface $\widetilde{\Omega}$ from Theorem \ref{theorem:existence_deformation} as a deformation of $\Omega$. We start by constructing a deformation of $\partial \Omega$ in \S \ref{subsection:escape_function} and then extend it as a deformation of the whole domain in \S \ref{subsection:new_domain}.

In \S \ref{section:differential_operators}, we recall the definition of the action of a differential operator with constant coefficient on $\mathbb{R}^2$ on a totally real surface in $\mathbb{C}^2$, and we prove that for $\im \omega \geq 0$ and $\re \omega$ close to $\lambda$ (in the setting of Theorem \ref{theorem:existence_deformation}) the operator $P_{\omega,0}$ is elliptic on $\widetilde{\Omega}$.

In \S \ref{section:inviscid_invertibility}, we end the proof of Theorem \ref{theorem:existence_deformation} by showing that $P_{\omega,0}$ is invertible in the range of $\omega$ from this theorem. This is the most technical part of the paper, and we rely crucially on a difficult invertibility result on $\Omega$ from \cite{dwz_internal_waves} (see Lemma 7.2 of this reference).

The proof of Theorem \ref{theorem:main} is presented in \S \ref{section:viscosity_limit}, more precisely in \S \ref{subsection:proof_main}. It is based on the invertibility of $P_{\omega,0}$ proven in \S \ref{section:inviscid_invertibility} and a priori estimates for $P_{\omega,\nu} : H_0^2(\widetilde{\Omega}) \to H^{-2}(\Omega)$ that are uniform when $\nu > 0$ goes to $0$. These estimates are deduced from the work of Frank \cite{frank_coercive_singular} in \S \ref{subsection:a_priori_estimates}.

The paper has one appendix, that gives the proof of an elliptic regularity results for boundary value problems that we were not able to locate in the literature.

Notice that $\Omega \subseteq \mathbb{R}^2$ and $\lambda \in (0,1)$ will be fixed in \S \ref{subsection:generalities} for the whole paper.

\vspace{8pt}
\noindent{\bf Acknowledgements.} The authors would like to thank Maciej Zworski, Semyon Dyatlov, Long Jin and Michel Rieutord for many helpful discussions. MJ benefits from the support of the French government ``Investissements d’Avenir'' program
integrated to France 2030, bearing the following reference ANR-11-LABX-0020-01. JW is supported by Simons Foundation through a postdoctoral position at Institut des Hautes Études Scientifiques.

\section{Geometric and dynamical considerations}

We start this section by recalling in \S \ref{subsection:generalities} some geometric and dynamical notions from \cite{dwz_internal_waves}. We recall in particular the definitions of $\lambda$-simplicity and of the Morse--Smale condition that are needed for Theorems~\ref{theorem:main} and~\ref{theorem:existence_deformation}. The rest of the section is dedicated to the construction of the deformation $\widetilde{\Omega}$ from Theorem \ref{theorem:existence_deformation}, which is produced as an element of a family $(\overline{\Omega}_\tau)_{\tau \in [0,\tau_0)}$ of deformations of $\overline{\Omega}$ (i.e. $\overline{\Omega} = \overline{\Omega}_0$ and $\widetilde{\Omega} = \overline{\Omega}_\tau$ for some small $\tau > 0$). The deformation of the boundary of $\Omega$ plays a crucial role, and we will design it first in \S \ref{subsection:escape_function}. We deduce a deformation of the whole $\overline{\Omega}$ from the deformation of $\partial \Omega$ in \S \ref{subsection:new_domain}.

The main results of this section are Propositions \ref{proposition:existence_escape_function} and \ref{proposition:deformation_domain}. They are enough to construct the domain $\widetilde{\Omega}$ from Theorem \ref{theorem:existence_deformation}, however item \ref{item:invertibility_deformation} of this theorem will only be proved in \S \ref{section:inviscid_invertibility}.

\subsection{Generalities}\label{subsection:generalities}

From now on, we fix a bounded open set with real-analytic boundary $\Omega \subseteq \mathbb{R}^2$. Let us fix also a real-analytic orientation preserving diffeomorphism $z : \mathbb{S}^1 \to \partial \Omega$, where our model of $\mathbb{S}^1$ is $\mathbb{R}/ \mathbb{Z}$. In order to state the $\lambda$-simplicity and Morse--Smale conditions, let us introduce for every $\omega \in (0,1) + i \mathbb{R}$ the linear forms $\ell_{\omega}^{\pm}$ on $\mathbb{R}^2$ by\footnote{In the range of $\omega$ that we consider, $1 - \omega^2$ always has a positive real part, which implies that the square root is well-defined (using the branch of the square root on $\mathbb{R}_+^* + i \mathbb{R}$ which is positive on $\mathbb{R}_+^*$).}
\begin{equation*}
     \ell_{\omega}^{\pm}(x_1,x_2) \coloneqq \pm \frac{x_1}{\omega} + \frac{x_2}{\sqrt{1-\omega^2}}. 
\end{equation*}
The first hypothesis that we will require on $\Omega$ and $\lambda$ is the following.

\begin{definition}[{\cite[Definition 1.1]{dwz_internal_waves}}]\label{definition:lambda_simple}
Let $\lambda \in (0,1)$. We say that $\Omega$ is $\lambda$-simple if the restrictions of $\ell_{\lambda}^+$ and $\ell_{\lambda}^-$ to $\partial \Omega$ have exactly two critical points (each) and that these critical points are non-degenerate.
\end{definition}

\begin{remark}
Notice that if there is a $\lambda \in (0,1)$ such that $\Omega$ is $\lambda$-simple, then $\partial \Omega$ is diffeomorphic to a circle and thus $\Omega$ is simply connected.
\end{remark}

From now on, let us fix $\lambda \in (0,1)$ and assume that $\Omega$ is $\lambda$-simple. Recall from \cite[(1.3)]{dwz_internal_waves} that there are two orientation-reversing involutions $\gamma_\lambda^{\pm}$ of $\partial \Omega$ such that 
\begin{equation}\label{eq:definition_involution}
    \ell_\lambda^{\pm}(\gamma_\lambda^\pm(x)) = \ell_\lambda^\pm(x)
\end{equation}
for every $x \in \partial \Omega$. Notice that \eqref{eq:definition_involution} uniquely defines $\gamma_\lambda^\pm$ once we impose that it is smooth and different from the identity: for every $x \in \partial \Omega$ we have $(x + \ker \ell_\lambda^\pm) \cap \partial \Omega = \set{x,\gamma_\lambda^\pm(x)}$. The smoothness of $\gamma_\lambda^\pm$ at characteristic points can be seen in Morse coordinates (for the restriction of $\ell_\lambda^\pm$) to the boundary, see \cite[Section 2.1]{dwz_internal_waves}. In our context, one can even prove that $(\lambda,x) \mapsto \gamma_\lambda^\pm(x)$ is real-analytic, but we will not need this fact.

In order to introduce the Morse--Smale condition we will work with, let us define the orientation-preserving diffeomorphism $b = \gamma_\lambda^+ \circ \gamma_\lambda^-$ of $\mathbb{S}^1$, called the chess billiard map.

\begin{figure}[t]
    \centering
    \includegraphics[width=.8\textwidth]{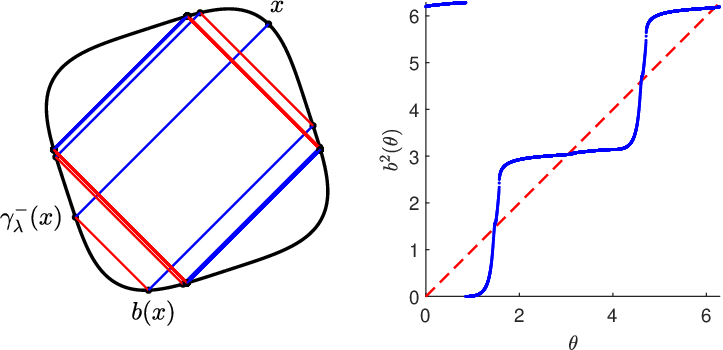}
    \caption{A numerically computed chess billiard trajectory and a chess billiard map. Here we take $\Omega$ to be $\{x\in \mathbb R^2 : x_1^4+x_2^4<1\}$ rotated by $\frac{\pi}{10}$ counter-clock-wisely and $\lambda=\frac{1}{\sqrt{2}}$. Left: A trajectory of the chess billiard starting from a point on $\partial\Omega$. Right: Graph of the map $b^2$ under the parametrization $\theta\mapsto (s_1|\cos(\theta+\frac{\pi}{10})|^{\frac12}, s_2 | \sin(\theta+\frac{\pi}{10}) |^{\frac12}  )$ with $s_1=\sign(\cos(\theta+\frac{\pi}{10}))$, $s_2=\sign(\sin(\theta+\frac{\pi}{10}))$ and $\theta\in \mathbb R/2\pi\mathbb Z$. The numerical results suggest $\Omega$ and $\lambda$ satisfy the Morse--Smale condition. }
    \label{fig:loop}
\end{figure}

\begin{definition}[{\cite[Definition 1.2]{dwz_internal_waves}}]\label{definition:morse_smale}
We say that $\lambda$ satisfies the Morse--Smale condition if:
\begin{enumerate}[label=(\roman*)]
    \item $\Omega$ is $\lambda$-simple;
    \item $b$ has at least a periodic point;
    \item for every $x \in \mathbb{S}^1$ and $n \geq 1$, if $b^n(x) = x$ then\footnote{Since $b^n(x) = x$, the derivative of $b^n$ at $x$ is an endomorphism of the $1$-dimensional real vector space $T_x \partial \Omega$, and thus identifies with a real number. Moreover, $(b^n)'(x) > 0$ as $b^n$ is orientation-preserving.} $(b^n)'(x) \neq 1$.
\end{enumerate}
\end{definition}

From now on, we assume that $\lambda$ satisfies the Morse--Smale condition. For extensive discussion of the Morse--Smale condition and the dynamics of $b$, we refer to \cite[\S 2]{dwz_internal_waves}. For further discussion of the dynamics of chess billiard maps, one may look at \cite{dynamics_billiard}. Examples of domains satisfying the Morse--Smale condition are given in \cite[\S 2.4-2.5]{dwz_internal_waves} as ``rounding'' of certain domains with corners. Notice that if a domain with $C^\infty$ boundary satisfies the Morse--Smale condition, then we can approximate its boundary by a real-analytic curve to get an example with real-analytic boundary. See also Figure \ref{fig:loop}.

We will need the following lemma from \cite{dwz_internal_waves}.

\begin{lemma}[{\cite[Lemma 2.1]{dwz_internal_waves}}]\label{lemma:sign_stuff}
For every $x \in \partial \Omega$, let $\mu^\pm(x)$ denote the sign of\footnote{Notice that if we identify $\partial \Omega$ with $\mathbb{S}^1 = \mathbb{R}/ \mathbb{Z}$ using the parametrization $z$, then $\ell_\lambda^\pm(z'(z^{-1}(x)))$ is just the derivative of $\ell_\lambda^\pm$ at $x$. In particular, $\mu^\pm(x)$ does not depend on the choice of the parametrization.} $\ell_\lambda^\pm(z'(z^{-1}(x)))$. For every $x \in \partial \Omega$, we have
\begin{equation}\label{eq:sign_difference}
    \sign(\ell_\lambda^{\mp}(\gamma_\lambda^{\pm}(x) - x)) = \pm \mu^\pm(x),
\end{equation}
\begin{equation}\label{eq:sign_change}
    \mu^\pm(\gamma_\lambda^{\pm}(x)) = - \mu^\pm(x).
\end{equation}
Moreover, for every $x \in \mathbb{R}^2$, we have
\begin{equation}\label{eq:derivative_lomega}
    \frac{\partial}{\partial \omega}\left( \ell_\omega^{\pm}(x) \right)_{|\omega = \lambda} = \frac{2 \lambda^2 - 1}{2\lambda(1 - \lambda^2)} \ell_\lambda^{\pm}(x) + \frac{1}{2 \lambda (1 - \lambda^2)} \ell_\lambda^{\mp}(x).
\end{equation}
\end{lemma}

\subsection{Deformation of the boundary}\label{subsection:escape_function}

As explained above, the first step in the proof of Theorem \ref{theorem:existence_deformation} is to construct a deformation of the \emph{boundary} of $\Omega$. This is done in the following lemma:

\begin{proposition}\label{proposition:existence_escape_function}
There is a real-analytic vector field $X$ on $\partial \Omega$ such that the vector fields $Y_+$ and $Y_-$ on $\partial \Omega$ defined by
\begin{equation*}
    Y_\pm(x) = X(x) - \mathrm{d}\gamma_\lambda^\pm(\gamma_\lambda^\pm(x)) \cdot X(\gamma_\lambda^\pm (x)),
\end{equation*}
for $x \in \partial \Omega$, do not vanish, and are respectively negatively and positively oriented.
\end{proposition}

Let us explain how we deduce a deformation of $\partial \Omega$ in $\mathbb{C}^2$ from Proposition \ref{proposition:existence_escape_function}. Let us denote by $(\Phi_\tau^X)_{\tau \in \mathbb{R}}$ the flow of the vector field $X$. Since $X$ is real-analytic, the function $(\tau,x) \mapsto \Phi_\tau^X(x)$ is real-analytic from $\mathbb{R} \times \partial \Omega$ to $\partial \Omega \subseteq \mathbb{C}^2$ (see e.g. \cite[Theorem~4.1]{teschl_book}) Hence, this map has a holomorphic extension from a neighbourhood of $\mathbb{R} \times \partial \Omega$ in $\mathbb{C} \times (\partial \Omega)_{\mathbb{C}}$ to $\mathbb{C}^2$. Hence, for $\tau_0 > 0$ small, we may define a map $\Xi : (-\tau_0,\tau_0) \times \partial \Omega \to \mathbb{C}^2$ by\footnote{It is not essential to integrate the vector field $X$ here, as we only care about the first derivative in $\tau$ of $\Xi$. Hence, we could also define $\Xi$ by the formula:
\begin{equation*}
    \Xi(\tau,x) = z(z^{-1}(x) + i \tau \mathrm{d}z^{-1}(x) \cdot X(x)),
\end{equation*}
where we recall that $z : \mathbb{S}^1 \to \partial \Omega$ is a real-analytic orientation-preserving diffeomorphism and we see $\mathbb{S}^1$ as a subset of $\mathbb{C}/ \mathbb{Z}$.
}
\begin{equation}\label{eq:definition_Xi}
    \Xi(\tau,x) = \Phi_{i\tau}^X(x) \textup{ for } (\tau,x) \in (-\tau_0,\tau_0) \times \partial \Omega.
\end{equation}

Let us give a heuristic argument for the choice of this deformation (a down-to-earth reason is that it makes Proposition \ref{proposition:deformation_domain} possible, from which we can then deduce Theorems \ref{theorem:main} and \ref{theorem:existence_deformation}). A standard method to study a boundary problem is the \emph{reduction to the boundary} (see e.g. \cite[Chapter XX]{hormander3}). This is the method which is used in \cite{dwz_internal_waves}, and it will also play a role in \S \ref{section:inviscid_invertibility}. The idea is to use a fundamental solution to replace a boundary problem on $\Omega$ for an elliptic linear partial differential operator $P$ by a (system of) pseudodifferential equations on $\partial \Omega$. The new equation is a priori simpler because $\partial \Omega$ is a closed manifold. However, if we try to apply this method to the operator $P_{\lambda,0}$, which is not elliptic, we end up with an equation on $\partial \Omega$ of the form $\mathcal{C}u = f$ where $\mathcal{C}$ is not a pseudodifferential operator. Indeed, the Schwartz kernel of $\mathcal{C}$ has singularities on the graphs of $\gamma_\lambda^+$ and $\gamma_\lambda^-$ (while the kernel of a pseudodifferential operator only has singularities on the diagonal), see \cite[Proposition 4.15]{dwz_internal_waves}. A natural idea to bypass this difficulty is to replace $\partial \Omega$ by a complex deformation $\partial \widetilde{\Omega}$ such that $\partial \widetilde{\Omega} \times \partial \widetilde{\Omega}$ does not intersect the graphs of (the holomorphic extensions of) $\gamma_\lambda^+$ and $\gamma_\lambda^-$. It is a simple exercise using Taylor's formula to check that for $\tau$ small but non-zero the deformation $\Xi(\tau,\partial \Omega)$ has this property.

\begin{figure}[t]
    \centering
    \includegraphics[width=.8\textwidth]{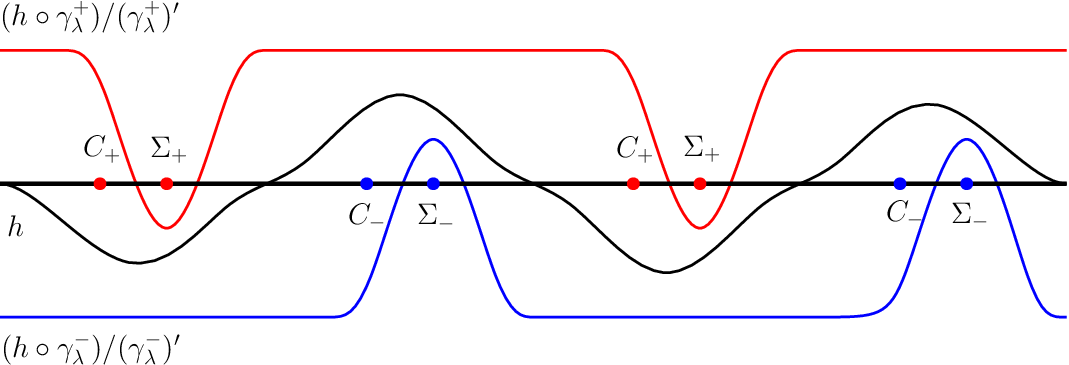}
    \caption{An illustration of the function $h$ trivializing $X=h(\theta)\frac{\partial}{\partial \theta}$ and functions trivializing $(\gamma_\lambda^\pm)^*X$, where we identify $\partial \Omega$ with $\mathbb{S}^1 = \mathbb{R}/ \mathbb{Z}$ using the parametrization $z$. Here $C_{\pm}\coloneqq \{\gamma_\lambda^\pm(\theta)=\theta\}$ are sets of characteristic points and $\Sigma_{\pm}$ are sets of periodic points.  }
    \label{fig:escape}
\end{figure}

\begin{remark}
A previous version of this work made a more extensive use of the reduction to the boundary method. We were only deforming the boundary, no deformation of the whole domain was used. This other method is also interesting, in particular has it requires to define the action of a real-analytic pseudodifferential operator, originally defined on $\partial \Omega$, on a complex deformation of $\partial \Omega$. We finally chose the present method because it is somehow more intrinsic (we work with the equation itself rather than its reduction to the boundary) and, most of all, because it is way less technical.
\end{remark}

Let us denote by $n$ the common minimal period of the periodic points of $b$ (see \cite[Lemma 2.4]{dwz_internal_waves}) and $\Sigma_{+}$ (resp. $\Sigma_-$) denote the set of periodic points $x$ for $b$ such that $(b^{n})'(x) > 1$ (resp. $(b^n)'(x) < 1$). We will need the following lemma.

\begin{lemma}\label{lemma:explicit_dynamics}
Let $U$ and $V$ be neighbourhoods in $\partial \Omega$ of $\Sigma_+$ and $\Sigma_-$ respectively. Let $K_+$ and $K_-$ be compact subsets respectively of $\partial \Omega \setminus \Sigma_-$ and $\partial \Omega \setminus \Sigma_+$. There is an integer $m_0$ such that for every $m \geq m_0$ we have
\begin{equation*}
b^{-m}(K_+) \subseteq U \textup{ and } b^{m}(K_-) \subseteq V.
\end{equation*}
\end{lemma}

\begin{proof}
Let us prove for instance that $b^m(K_-) \subseteq V$ for $m$ large enough. Since $\Sigma_-$ is a union attracting periodic orbits, we may find a neighbourhood $W \subseteq V$ of $\Sigma_-$ such that $b(W) \subseteq W$. Hence, we only need to prove that for every $x \in \mathbb{S}^1 \setminus \Sigma_+$ there is $m \geq 0$ such that $b^m(x) \in W$ (we can then conclude by a compactness argument).

Let $x \in \partial \Omega \setminus \Sigma_+$. Let $\pi : \mathbb{R} \to \partial \Omega$ be the (smooth) universal covering for $\partial \Omega$. Let $\tilde{b}: \mathbb{R} \to \mathbb{R}$ be a lift of $b$ (i.e. $\tilde{b}$ is smooth and we have $\pi \circ \tilde{b} = b \circ \pi$) and $\tilde{x} \in \mathbb{R}$ be a lift of $x$. Let $x_0,x_1$ be the two points in $\mathbb{R}$ such that $x_0 < \tilde{x} < x_1$, the projection of $x_0$ and $x_1$ are periodic points for $b$ and there are no other points in $(x_0,x_1)$ whose projections are periodic points. Such points exist because periodic points of $b$ are isolated, due to the Morse--Smale condition. There is then an integer $k$ such that\footnote{Since the projections of $x_0$ and $x_1$ are periodic points, we know that $\tilde{b}^n(x_0) - x_0$ and $\tilde{b}^n(x_1) - x_1$ are integers, and they must coincide, the opposite would contradict the fact that $b$ is injective.} $\tilde{b}^n(x_0) - x_0 = \tilde{b}^n(x_1) - x_1 = k$. Consider the sequence $(y_m)_{m \geq 0}$ defined by $y_m = \tilde{b}^{mn}(\tilde{x}) - mk$. Since $\tilde{b}$ is increasing, this is a sequence of points in $(x_0,x_1)$. Moreover, since there are no point in $(x_0,x_1)$ that projects to a periodic point of $b$, we find that $\tilde{b}^n(y) - k < y$ for every $y \in (x_0,x_1)$ or $\tilde{b}^n(y) - k > y$ for every $y \in (x_0,x_1)$. Consequently, the sequence $(y_m)_{m \geq 0}$ is monotonic, and thus has a limit. It implies that the sequence $(b^{mn}(x))_{m \geq 0}$ has a limit, which must be a periodic point for $b$, and since it cannot be an element of $\Sigma_+$, we find that for $m$ large enough $b^{mn}(x) \in W$.
\end{proof}
 
In the following lemma, we construct a first candidate for the vector field $X$ from Proposition \ref{proposition:existence_escape_function}, that we will upgrade later. In the statement of this lemma we use the notation
\begin{equation*}
    g^* Y(x) = \mathrm{d}g(x)^{-1} \cdot Y(g(x)) \textup{ for } x \in \partial \Omega,
\end{equation*}
where $g : \partial \Omega \to \partial \Omega$ is a diffeomorphism and $Y \in \Gamma(T \partial \Omega)$. Notice that if $g_1$ and $g_2$ are two diffeomorphisms on $\partial \Omega$ then the chain rule implies that $(g_1 \circ g_2)^* Y = g_2^*(g_1^* Y)$. Using the fact that $\gamma_\lambda^\pm$ are involutions, we find that the vector fields $Y_\pm$ in Proposition~\ref{proposition:existence_escape_function} are defined by $Y_\pm = X - (\gamma_\lambda^\pm)^* X$.

\begin{lemma}\label{lemma:first_vector_field}
There is $X_0 \in \Gamma(T \partial \Omega)$ and an integer $N \geq 1$ such that the vector field $X_0 - (b^{N})^*X_0$ does not vanish and is positively oriented.
\end{lemma}

\begin{proof}
In this proof, we use the direct parametrization $z : \mathbb{S}^1 \to \partial \Omega$ to identify $\partial \Omega$ with $\mathbb{S}^1$. By a slight abuse of notation, we will use the same notation for objects defined on $\mathbb{S}^1$ and on $\partial \Omega$, e.g. seeing $b$ as a map from $\mathbb{S}^1$ to itself (in particular the derivative of $b$ at any point identifies with a positive real number).

Let $U_+$ and $U_-$ be neighbourhoods of $\Sigma_+$ and $\Sigma_-$ respectively such that
\begin{equation*}
\sup_{\theta \in U_+} (b^{-n})'(\theta) < 1 \textup{ and } \sup_{\theta \in U_-} (b^n)'(\theta) < 1.
\end{equation*}
We may assume in addition that $U_+$ and $U_-$ are respectively backward and forward invariant by $b$. Let then $V_+$ and $V_-$ be open neighbourhoods of $\Sigma_+$ and $\Sigma_-$, relatively compacts in $U_+$ and $U_-$ respectively. According to Lemma \ref{lemma:explicit_dynamics}, we can find a large $N \geq 1$ such that $b^N(\mathbb{S}^1 \setminus V_+) \subseteq V_-$ and $b^{-N}(\mathbb{S}^1 \setminus V_-) \subseteq V_+$. We may assume in addition that $N$ is a multiple of $n$, i.e. $N = pn$ for some integer $p \geq 1$. Let then $h_0$ be a smooth function, supported in $U_+ \cup U_-$ and such that
\begin{itemize}
\item $|h_0| \leq 1$;
\item $h_0 \geq 0$ on $U_+$ and $h_0 \leq 0$ on $U_-$;
\item $h_0 \equiv 1$ on $V_+$ and $h_0 \equiv -1$ on $V_-$.
\end{itemize}
We define then $X_0 = h_0 \frac{\partial}{\partial \theta}$ and notice that
\begin{equation*}
    X_0 - (b^N)^*(X_0) = \left( h_0 - \frac{h_0 \circ b^N}{(b^N)'}\right) \frac{\partial}{\partial \theta}.
\end{equation*}
Hence, we want to prove that the function $h_0 - h_0 \circ b^N /(b^N)'$ is strictly positive.

Let us distinguish several cases. If $\theta \notin V_+$ then $b^N(\theta) \in V_-$ and thus $h_0(b^N(\theta)) = -1$. If $\theta \notin U_-$, then $h_0(\theta) \geq 0$ and thus $ h_0(\theta) - h_0(b^N(\theta))/(b^N)'(\theta) \geq (b^N)'(\theta)^{-1} > 0$. If $\theta \in U_-$, then $h_0(\theta) - h_0(b^N(\theta))/(b^N)'(\theta) \geq -1 + (b^N)'(\theta)^{-1}$. But since $\theta \in U_-$, the set $U_-$ is forward invariant by $b^n$, the derivative of $b^n$ is strictly less than $1$ everywhere on $U_-$ and $N$ is a multiple of $n$, we find that $(b^N)'(\theta) < 1$, and thus $h_0(\theta) - h_0(b^N(\theta))/(b^N)'(\theta) > 0$.

It remains to deal with the case $\theta \in V_+$. In this case, we have $h_0(\theta) = 1$. If $b^N(\theta) \notin U_+$, then $h_0(b^N(\theta))$ is non-positive and thus $h_0(\theta) - h_0(b^N(\theta))/(b^N)'(\theta) \geq 1 > 0$. If $b^N(\theta) \in U_+$, then we find as in the previous case (using in addition that $(b^N)'(\theta) = (b^{-N})'(b^N(\theta))^{-1}$) that $(b^N)'(\theta) > 1$, and thus $h_0(\theta) - h_0(b^N(\theta))/(b^N)'(\theta) \geq 1 - (b^N)'(\theta)^{-1} > 0$.
\end{proof}

By averaging the vector field $X_0$ from Lemma \ref{lemma:first_vector_field}, we prove Proposition \ref{proposition:existence_escape_function}.

\begin{proof}[Proof of Proposition \ref{proposition:existence_escape_function}]
Let $X_1 = \sum_{k = 0}^{N-1} (b^k)^* X_0$ and $X = 2 X_1 + (\gamma_\lambda^+)^* X_1 + (\gamma_\lambda^-)^* X_1$, where $X_0$ is the vector field from Lemma \ref{lemma:first_vector_field}. Notice that $X_1 - b^* X_1 = X_0 - (b^N)^* X_0$ is directly oriented. Now, let us compute
\begin{equation*}
\begin{split}
X - (\gamma_\lambda^+)^* X & = 2 X_1 + (\gamma_\lambda^+)^* X_1 + (\gamma_\lambda^-)^* X_1 - 2 (\gamma_\lambda^+)^* X_1 - X_1 - (b^{-1})^* X_1   \\
                   & = X_1 - (\gamma_\lambda^+)^* X_1 - (b^{-1})^* X_1 + (\gamma_\lambda^-)^* X_1 \\
                   & = (\gamma_\lambda^-)^* (X_1 - b^* X_1 ) - (b^{-1})^* (X_1 - b^* X_1).
\end{split}
\end{equation*}
Since $X_1 - b^* X_1$ is positively oriented, $b$ is orientation preserving and $\gamma_\lambda^-$ is orientation reversing, we get that $Y_+ = X - (\gamma_\lambda^+)^* X$ is negatively oriented. Similarly, we compute
\begin{equation*}
\begin{split}
X - (\gamma_\lambda^-)^* X & = X_1 + (\gamma_\lambda^+)^* X_1 - (\gamma_\lambda^-)^* X_1 - b^* X_1  \\
                   & = (X_1 - b^* X_1) - (\gamma_\lambda^-)^* (X_1 - b^* X_1),
\end{split}
\end{equation*}
and as above, we find that $Y_- = X - (\gamma_\lambda^-)^* X$ is positively oriented.

With our definition of $X$, it is a priori only $C^\infty$. However, since the properties that we want for $X$ are open (in the topology of uniform convergence on the circle), we may approximate $X$ by a real-analytic vector field that still satisfies these properties.
\end{proof}

\subsection{Complex deformation of the domain}\label{subsection:new_domain}

We want now to construct a family $(\Omega_\tau)_{\tau \in [0,\tau_0)}$ of deformation of the domain $\Omega$ within $\mathbb{C}^2$ that extends the deformation of the boundary that we designed in the previous subsection (see \eqref{eq:definition_Xi}). In order to state the properties of the deformation of the domain that we will construct, let us introduce the differential operators
\begin{equation}\label{eq:basis_R2}
    L_{\lambda}^{\pm} = \frac{1}{2}\left( \pm \lambda \partial_{x_1} + \sqrt{1 - \lambda^2} \partial_{x_2} \right).
\end{equation}
These are constant coefficients differential operators of order $1$, and thus they identify with constant vector fields on $\mathbb{R}^2$. We will identify them with the vectors $\frac12(\pm \lambda, \sqrt{1 - \lambda^2})$. With this convention, for every $x \in \mathbb{R}^2$, we have
\begin{equation*}
    x = \ell_{\lambda}^+(x) L_\lambda^+ + \ell_{\lambda}^-(x) L_\lambda^-.
\end{equation*}

The family $(\Omega_\tau)_{\tau \in [0,\tau_0)}$ will be given by $\Omega_\tau = \Xi(\tau,\Omega)$ where the extension of $\Xi$ to $\Omega$ is given by the following result.

\begin{proposition}\label{proposition:deformation_domain}
There is $\tau_0 > 0$ such that the function $\Xi$ defined in \eqref{eq:definition_Xi} has a real-analytic extension to $(-\tau_0,\tau_0) \times \overline{\Omega}$ with the following properties:
\begin{enumerate}[label=(\roman*)]
    \item $\Xi(0,x) = x$ for every $x \in \overline{\Omega}$;\label{item:no_deformation_at_zero}
    \item for every $x \in \overline{\Omega}$ we have\footnote{The operator $L_\lambda^\pm$ is applied to $\Xi$ with the $\tau$ variable frozen.}:
    \begin{equation*}
        \ell_{\lambda}^{\pm}\left(\frac{\partial L_\lambda^{\mp} \Xi}{\partial \tau}(0,x)\right) \in i \mathbb{R}_+^*.
    \end{equation*}\label{item:transversality_deformation}
\end{enumerate}
\end{proposition}

\begin{remark}\label{remark:transversality}
Item \ref{item:transversality_deformation} in Proposition \ref{proposition:deformation_domain} may be understood as a transversality condition. Indeed, it implies that for $\tau > 0$ small enough the surface with boundary $\Omega_\tau = \Xi(\tau,\Omega)$ is transversal to the two direction of propagation associated to the hyperbolic operators $P_{\lambda,0}$: the \emph{complex} lines $\ker \ell_\lambda^+$ and $\ker \ell_\lambda^-$ (which are directed by $L_\lambda^-$ and $L_\lambda^+$ respectively). It will eventually imply the ellipticity of $P_{\lambda,0}$ on $\Omega_\tau$ for $\tau> 0$ small (see Lemma \ref{lemma:inviscid_ellipticity}). 

To understand loosely the relation between \ref{item:transversality_deformation} and the ellipticity of $P_{\lambda,0}$ on $\Omega_\tau$, recall from \cite[(4.3)]{dwz_internal_waves} that $P_{\lambda,0} = 4 L_{\lambda}^+ L_\lambda^-$. Hence, $P_{\lambda,0}$ is a wave operator with propagation directions given by $L_{\lambda}^+$ and $L_\lambda^-$. The complex deformation has this effect of removing these directions, making the operator elliptic.

We will deduce from \ref{item:transversality_deformation} that $\ell_\lambda^+$ and $\ell_\lambda^-$ restricted to $\Omega_\tau$ induces diffeomorphism from $\Omega_\tau$ to their images (which are subsets of $\mathbb{C}$) in Lemma \ref{lemma:diffeomorphisms_deformation}.
\end{remark}

\begin{remark}\label{remark:control_boundary}
It follows from \eqref{eq:definition_Xi} that if $x \in \partial \Omega$ and $\tau$ is small then $\Xi(\tau,x) \in (\partial \Omega)_{\mathbb{C}}$. Indeed, if $\varphi : U \to \mathbb{R}$ is a real-analytic submersion defined on a neighbourhood of $x$ in $\mathbb{R}^2$ and such that $U \cap \partial \Omega = \set{\varphi = 0}$ then $\varphi$ has a holomorphic extension $\tilde{\varphi} : W \to \mathbb{C}$ on some neighbourhoods $W$ of $x$ in $\mathbb{C}^2$. Moreover, up to taking $W$ smaller, we have $(\partial \Omega)_{\mathbb{C}} \cap W = \set{\tilde{\varphi} = 0}$. For a small real $\tau$, we have $\Phi_\tau^X(x) \in V \cap \partial \Omega$ and thus $\tilde{\varphi}(\Phi_\tau^X(x)) = 0$. It follows then from the analytic continuation principle that for $\tau$ small we also have $\tilde{\varphi}(\Phi_{i\tau}^X(x)) = 0$, that is $\Xi(\tau,x) \in (\partial \Omega)_{\mathbb{C}}$.
\end{remark}

\begin{remark}
Notice that in Proposition \ref{proposition:deformation_domain} we are able to produce a real-analytic deformation of $\Omega$. This regularity property is sometimes convenient, see for instance Remark \ref{remark:control_boundary} and the beginning of the proof of Lemma \ref{lemma:linear_forms_deformed}. We also use the fact that $\overline{\Omega}_\tau = \Xi(\tau,\overline{\Omega})$ is a real-analytic surface to apply real-analytic elliptic regularity results such as \cite[Chapitre 8, Théorème 1.2]{lions_magenes_3} or \cite{morrey_nirenberg_57}. However, we suspect that we could work without the real-analytic property of $\Xi$ by using the structure of hypo-analytic manifold of $\overline{\Omega}_\tau$.
\end{remark}

\begin{remark}\label{remark:complex_conjugacy}
We could also consider the surface $\overline{\Omega}_\tau$ for small \emph{negative} $\tau$'s. It would allow us to continue the inverse $P_{\omega,0}^{-1}$ from below the real-axis. That is, in Theorem \ref{theorem:existence_deformation}, item \ref{item:invertibility_deformation}, the subset $(\lambda- \delta,\lambda + \delta) + i (- \delta, + \infty)$ would be replaced by $(\lambda- \delta,\lambda + \delta) + i (-\infty, \delta)$. One way to see that is to notice that, with our construction, we have $\Omega_{- \tau} = \iota(\Omega_\tau)$ where $\iota : \mathbb{C}^2 \to \mathbb{C}^2$ denote the complex conjugacy (coordinate by coordinate). Hence, we have an anti-linear isomorphism between $H_0^1(\overline{\Omega}_\tau)$ (resp. $H^{-1}(\overline{\Omega}_\tau)$) and $H_0^1(\overline{\Omega}_{-\tau})$ (resp. $H^{-1}(\overline{\Omega}_{- \tau})$) given by $u \mapsto \bar{u} \circ \iota$. This isomorphism conjugates $P_{\omega,\nu}$ acting on $\Omega_{\tau}$ and $P_{\bar{\omega}, - \nu}$ acting on $\Omega_{- \tau}$ (see \S \ref{section:differential_operators} for the definition of the action of a constant differential operator on a totally real surface in $\mathbb{C}^2$). 

Notice that going from $\Omega_{\tau}$ to $\Omega_{- \tau}$ switches the sign of $\nu$. Hence, we cannot use $\Omega_{- \tau}$ in the proof of Theorem \ref{theorem:main}. The point that would fail is the proof of the ellipticity estimate Lemma \ref{lemma:singular_ellpiticity}.
\end{remark}

\begin{proof}[Proof of Proposition \ref{proposition:deformation_domain}.]
We will define the extension of $\Xi$ by the formula
\begin{equation}\label{eq:formula_to_extend_Xi}
    \Xi(\tau,x) = g_+(\tau,x) L_\lambda^+ + g_-(\tau,x) L_\lambda^-
\end{equation}
where $g_+$ and $g_-$ are complex functions to be defined.

We will construct for instance $g_+$, the function $g_-$ is obtained similarly, switching the relevant signs in the construction below. For $x \in \overline{\Omega}$, it follows from the $\lambda$-simplicity condition that the intersection of the line $\set{y \in \mathbb{R}^2 : \ell_\lambda^+(y) = \ell_{\lambda}^+(x)}$ with $\partial \Omega$ is made of two points $\pi_+^{\textup{down}}(x)$ and $\pi_+^{\textup{up}}(x)$ (that may coincide) such that $\ell_\lambda^-(\pi_+^{\textup{down}}(x)) \leq \ell_{\lambda}^-(\pi_+^{\textup{up}}(x))$. Notice that the points $\pi_+^{\textup{up}}(x)$ and $\pi_+^{\textup{down}}(x)$ coincide exactly when $x$ is a point of $\partial \Omega$ such that $\set{y \in \mathbb{R}^2 : \ell_\lambda^+(y) = \ell_{\lambda}^+(x)}$ is tangent to $\partial \Omega$. We define $g_+$ by the formula
\begin{equation*}
    g_+(\tau,x) = \begin{cases}
     \ell_{\lambda}^+(\Xi(\tau,\pi_+^{\textup{down}}(x)))  + \ell_\lambda^-(x - \pi_+^{\textup{down}}(x)) \frac{\ell_{\lambda}^+(\Xi(\tau,\pi_+^{\textup{up}}(x)) - \Xi(\tau,\pi_+^{\textup{down}}(x)))}{\ell_\lambda^-(\pi_+^{\textup{up}}(x)-\pi_+^{\textup{down}}(x))}, \\  \, \qquad \qquad \qquad \qquad \qquad \qquad \qquad  \textup{ if } \pi_+^{\textup{up}}(x) \neq \pi_+^{\textup{down}}(x), \\
                   \ell_\lambda^+(\Xi(\tau,x)), \qquad \quad \textup{ if } \pi_+^{\textup{up}}(x) = \pi_+^{\textup{down}}(x).
    \end{cases}
\end{equation*}
The function $g_-$ is defined similarly, inverting the index and exponents ``$+$'' and ``$-$'' (see Figure \ref{fig:interp} for the definition of $\pi_-^{\textup{up}/\textup{down}}$). The motivation behind the definition of $g_+$ is the following: we extend the definition on the boundary (that is imposed) by interpolating in the direction given by $L_\lambda^-$.

\begin{figure}[t]
    \centering
    \includegraphics[width=.5\textwidth]{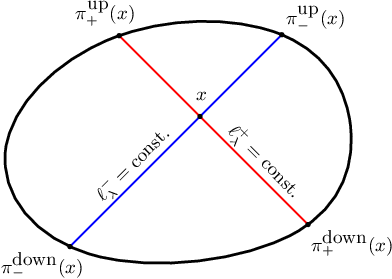}
    \caption{Maps $\pi_{\pm}^{\textup{up/down}}$ used in the extension of $\Xi$. }
    \label{fig:interp}
\end{figure}

Let us prove that the function $g_+$ is real-analytic on $(-\tau_0,\tau_0) \times \overline{\Omega}$ for a small $\tau_0 > 0$. Let $x_0 \in \overline{\Omega}$. If $x_0$ is not a point on $\partial \Omega$ at which the kernel of $\ell_\lambda^+$ is tangent to $\partial \Omega$, it follows from the implicit function theorem that $\pi_+^{\textup{up}}$ and $\pi_+^{\textup{down}}$ define real-analytic functions on a neighbourhood of $x_0$, and thus that $g_+$ is real-analytic on a neighbourhood of $(-\tau_0,\tau_0) \times \set{x_0}$.

Let us now consider the case of a point $x_0 \in \partial \Omega$ such that $T_{x_0} \partial \Omega = \ker \ell_{\lambda}^+$. Since $\ker \ell_{\lambda}^+$ is the span of $L_\lambda^-$, we find that there is a neighbourhood $U$ of $x_0$ such that $\ell_\lambda^-$ induces a diffeomorphism from $U \cap \partial \Omega$ to its image $I$. In particular, there is a real-analytic function $f : I \to \mathbb{R}$ such that $\ell_\lambda^+(x) = f(\ell_\lambda^-(x))$ for every $x \in \partial \Omega \cap U$. It implies in particular that
\begin{equation*}
    \partial \Omega \cap U = \set{ x_0 + u L_{\lambda}^- + f(u) L_\lambda^+ : u \in I}.
\end{equation*}
By the $\lambda$-simplicity assumption, $f$ has a non-degenerate critical point at $\ell_\lambda^-(x_0)$, and thus there is $\alpha \in \mathbb{R}^*$ and a real-analytic function $w : I \to \mathbb{R}$ such that $w(\ell_\lambda^-(x_0)) = 0$ and $w'(\ell_\lambda^-(x_0)) = 1$ such that $f(u) = \ell_\lambda^+(x_0) +  \alpha w(u)^2$ (by the one-dimensional Morse lemma). Notice that the sign of $\alpha$ depends on whether $\ell_\lambda^+(x_0)$ is the global minimum or maximum for $\ell_\lambda^+$ on $\overline{\Omega}$. Hence, for $x \in \overline{\Omega}$ close to $x_0$, the number $\ell_\lambda^+(x-x_0)/\alpha$ is non-negative and close to $0$. Hence, for such an $x$ we have
\begin{equation*}
    \pi_{+}^{\textup{up}/\textup{down}}(x) = x_0 + w^{-1}\left(\pm \sqrt{\frac{\ell_\lambda^+(x-x_0)}{\alpha}}\right) L_\lambda^- + \ell_\lambda^+(x) L_\lambda^+.
\end{equation*}
Indeed, the formula in the right-hand side defines point on $\partial \Omega$ with the same image by $\ell_\lambda^+$ as $x$.

Let us then introduce the real-analytic function
\begin{equation*}
    k(\tau,u) = \ell_\lambda^+(\Xi(\tau, x_0 + w^{-1}(u) L_\lambda^- + (\alpha u^2 + \ell_\lambda^+(x_0)) L_\lambda^+))
\end{equation*}
and notice that for $x \neq x_0$ close to $x_0$ we have
\begin{equation*}
    g_+(\tau,x) = \ell_\lambda^-(x) a_1\left(\tau,\frac{\ell_\lambda^+(x-x_0)}{\alpha}\right) + a_2\left(\tau,\frac{\ell_\lambda^+(x-x_0)}{\alpha}\right),
\end{equation*}
where
\begin{equation*}
    a_1(\tau,u) = \frac{k (\tau,\sqrt{u}) - k(\tau,-\sqrt{u}) }{w^{-1}( \sqrt{u}) - w^{-1}(- \sqrt{u})},
\end{equation*}
and
\begin{equation*}
    a_2(\tau,u) = \frac{k(t,-\sqrt{u}) w^{-1}(\sqrt{u}) - k(\tau,\sqrt{u})w^{-1}(- \sqrt{u})}{w^{-1}(\sqrt{u}) - w^{-1}(-\sqrt{u})}.
\end{equation*}
Thus, we only need to prove that $a_1$ and $a_2$ are real-analytic on a neighbourhood of $(0,0)$.

Let us expand $k$ in power series in its second argument in the expression $k(\tau,\sqrt{u}) - k(\tau,- \sqrt{u})$, and notice that the term with even powers of $\sqrt{u}$ of $k(\tau,\sqrt{u})$ and $k(\tau,-\sqrt{u})$ cancel each other. Hence, we are left with the product of $\sqrt{u}$ and an analytic function of $\tau$ and $u$. Similarly, $w^{-1}(\sqrt{u}) - w^{-1}(-\sqrt{u})$ is the product of $\sqrt{u}$ and an analytic function of $u$. Hence, in the definition of $a_1$, the $\sqrt{u}$'s in the numerator and the denominator cancel with each other, and thus $a_1$ is a quotient of analytic functions. The denominator does not vanish when $u$ is small since $(w^{-1})'(0) = 1$, and it follows that $a_1$ is real-analytic.

Let us now prove that $a_2$ is analytic. Let us write explicitly the expansion in power series:
\begin{equation*}
    k(\tau,z) = \sum_{k \geq 0} b_k(\tau) z^k \textup{ and } w^{-1}(z) = \sum_{k \geq 0} c_k z^k
\end{equation*}
and then notice that
\begin{equation*}
\begin{split}
    & k(\tau,-\sqrt{u}) w^{-1}(\sqrt{u}) - k(\tau,\sqrt{u})w^{-1}(- \sqrt{u})  \\ & \qquad \qquad \qquad \qquad \qquad = \sum_{k_1,k_2 \geq 0} b_{k_1}(\tau) c_{k_2}((-1)^{k_1} - (-1)^{k_2}) \sqrt{u}^{k_1 + k_2}.
\end{split}
\end{equation*}
When $k_1 + k_2$ is even, then $(-1)^{k_1} = (-1)^{k_2}$. Hence, $$k(\tau,-\sqrt{u}) w^{-1}(\sqrt{u}) - k(\tau,\sqrt{u})w^{-1}(- \sqrt{u})$$ is the product of $\sqrt{u}$ and an analytic function of $\tau$ and $u$. It follows that $a_2$ is analytic (we deal with the denominator as for $a_1$).

It follows that $g_+$ is analytic on a small neighbourhood of $x_0$ in $\overline{\Omega}$ from which $\set{x_0}$ is removed, with an analytic extension to the whole neighbourhood. It remains to check that the value of the extension at $x_0$ coincides with the value of $g_+(\tau,x_0)$ from our definition. It follows by continuity, letting $x$ go to $x_0$ along $\partial \Omega$.

We proved that $g_+$ is analytic, and it follows that the formula \eqref{eq:formula_to_extend_Xi} defines a real-analytic function. Furthermore, since $g_+(\tau,x) = \ell_{\lambda}^+(\Xi(\tau,x))$ when $x \in \partial \Omega$, we find that \eqref{eq:formula_to_extend_Xi} indeed defines an extension for $\Xi$. Since $\ell_\lambda^+(\Xi(0,x)) = \ell_\lambda^+(x)$ for $x \in \partial \Omega$, we find that $g_+(0,x) = \ell_{\lambda}^+(x)$ for $x \in \overline{\Omega}$, and \ref{item:no_deformation_at_zero} holds (in all this paragraph, it is implicit that we have symmetric results for $g_-$).

It remains to prove \ref{item:transversality_deformation}. As above we will only deal with the ``$+$'' case. We start by noticing that $\ell_\lambda^+ (L_\lambda^- \Xi) = L_\lambda^- g_+$ so that for $x \in \Omega$ we have
\begin{equation}\label{eq:to_get_transversality}
\begin{split}
    \ell_{\lambda}^+ \left(L_\lambda^- \frac{\partial \Xi}{\partial \tau}(0,x) \right) & = \frac{\partial L_\lambda^- g_+}{\partial \tau}(0,x) \\
    & = i \frac{\ell_\lambda^+(X(\pi_+^{\textup{up}}(x)) - X(\pi_+^{\textup{down}}(x)))}{\ell_\lambda^-(\pi_+^{\textup{up}}(x) - \pi_+^{\textup{down}}(x))}.
\end{split}
\end{equation}
Differentiating the relation $\ell_\lambda^+(y) = \ell_\lambda^+(\gamma_\lambda^+(y))$ with respect to $X$, we find that
\begin{equation*}
    \ell_\lambda^+(\mathrm{d}\gamma_\lambda^+(y) \cdot X(y)) = \ell_\lambda^+(X(y))
\end{equation*}
for every $y \in \partial \Omega$. Notice also that $\pi_+^{\textup{down}}(x) = \gamma_\lambda^+(\pi_+^{\textup{up}}(x))$. Hence, \eqref{eq:to_get_transversality} becomes
\begin{equation}\label{eq:simplified_transversality}
\begin{split}
    \ell_{\lambda}^+ \left(L_\lambda^- \frac{\partial \Xi}{\partial t}(0,x) \right) & = i \frac{\ell_\lambda^+(X(\pi_+^{\textup{up}}(x)) - \mathrm{d}\gamma_\lambda^+(\pi_+^{\textup{down}}(x)) \cdot X(\pi_+^{\textup{down}}(x)))}{\ell_\lambda^-(\pi_+^{\textup{up}}(x) - \pi_+^{\textup{down}}(x))} \\
    & = i \frac{\ell_\lambda^+(Y_+(\pi_+^{\textup{up}}(x)))}{\ell_\lambda^-(\pi_+^{\textup{up}}(x) - \pi_+^{\textup{down}}(x))}.
\end{split}
\end{equation}
 Recalling Lemma \ref{lemma:sign_stuff} and Proposition \ref{proposition:existence_escape_function}, we find that the numerator and the denominator in \eqref{eq:simplified_transversality} both have the sign $- \mu^+(\pi_+^{\textup{up}}(x))$. It follows that $\ell_{\lambda}^+(L_\lambda^- \frac{\partial \Xi}{\partial t}(0,x)) \in i \mathbb{R}_+^*$ for every $x \in \overline{\Omega}$ such that $\mu^+(\pi_+^{\textup{up}}(x)) \neq 0$.
 
 There are exactly two points in $\overline{\Omega}$ such that $\mu^+(\pi_+^{\textup{up}}(x)) = 0$. If $x$ is such a point, then $x \in \partial \Omega$ and $x + \ker \ell_\lambda^+$ is tangent to $\partial \Omega$, thus by continuity
 \begin{equation*}
     \ell_{\lambda}^+\left(L_\lambda^- \frac{\partial \Xi}{\partial t}(0,x)\right) = i \lim\limits_{\substack{y \to x \\ y \in \partial \Omega}} \frac{\ell_\lambda^+(Y_+(y))}{\ell_\lambda^-(y - \gamma_\lambda^+(y))}.
 \end{equation*}
 It follows from the $\lambda$-simplicity assumption that the numerator and the denominator in this limit both vanishes at order $1$ at $x$, and thus that the limit is non-zero (and positive by continuity), which ends the proof of the proposition. Notice that we also get that $L_\lambda^+ g_-$ takes values in $i \mathbb{R}_+^*$ because $Y_-$ (from Proposition \ref{proposition:existence_escape_function}) is positively oriented while $Y_+$ is negatively oriented and there is an additional minus sign in the ``$-$'' case in Lemma \ref{lemma:sign_stuff}. These two signs cancel with each other and we get that the imaginary parts of $L_\lambda^+ g_-$ and $L_\lambda^- g_+$ have the same sign.
\end{proof}

\section{Action of constant coefficient differential operators on the deformed domain}\label{section:differential_operators}

As announced above, we will work on the surface with boundary $\overline{\Omega}_\tau = \Xi(\tau,\overline{\Omega})$ for some small $\tau > 0$, where $\Xi$ is defined in \eqref{eq:definition_Xi} and Proposition \ref{proposition:deformation_domain}. It will be convenient in \S \ref{section:inviscid_invertibility} to extend $\Xi$ to a real-analytic function on $M \times (-\tau_0,\tau_0)$ where $\tau_0 > 0$ is small and $M$ is an open neighbourhood of $\overline{\Omega}$ in $\mathbb{R}^2$. The actual choice of $M$ will be made in Lemma \ref{lemma:linear_forms_deformed}. Since $\Xi(0,\cdot)$ is the identity, we see that if $M$ and $\tau$ are small enough, then $\Xi(\tau,\cdot)$ induces a real-analytic embedding of a neighbourhood of $\overline{M}$ in $\mathbb{C}^2$. We denote by $M_\tau$ the image of $M$ by this embedding and by $\Omega_\tau$ the image of $\Omega$. Provided $\tau$ is small enough, $M_\tau$ is totally real, that is for every $x \in M_\tau$ we have $\mathbb{C}^2 = T_x M_{\tau} \oplus i T_x M_\tau$.

Our goal is to study the operator $P_{\omega,\nu}$ acting on $\Omega_\tau$. Let us remind how one can define this action. Let $A$ be a differential operator with constant coefficients on $\mathbb{R}^2$ of order $m \in \mathbb{R}$. Writing
\begin{equation*}
    A = \sum_{p+ q \leq m} a_{p,q} \frac{\partial^{p+q}}{\partial x_1^p \partial x_2^q}
\end{equation*}
we can define a differential operator $\widetilde{A}$ on $\mathbb{C}^2$ by the formula
\begin{equation*}
    \widetilde{A} = \sum_{p+ q \leq m} a_{p,q} \frac{\partial^{p+q}}{\partial z_1^p \partial z_2^q}
\end{equation*}
where for $j = 1,2$ the holomorphic derivative $\partial/\partial z_j$ is defined in the coordinates $(x_1+i y_1,x_2 + i y_2)$ on $\mathbb{C}^2$ by the formula
\begin{equation*}
    \frac{\partial}{\partial z_j} = \frac{1}{2}\left(\frac{\partial}{\partial x_j} - i \frac{\partial}{\partial y_j}\right).
\end{equation*}

Now, if $u \in C_c^\infty(M_\tau)$, since $M_\tau$ is totally real, there is $\tilde{u} \in C_c^\infty(\mathbb{C}^2)$ such that $\tilde{u}_{|M_\tau} = u$ and $\overline{\partial} \tilde{u}$ vanishes to infinite order on $M_\tau$. We say that $\tilde{u}$ is an almost analytic extension for $u$ (see e.g. \cite[Theorem 3.6]{zworski_book} for the existence of $\tilde{u}$). We can then define the action of $A$ on $u$ by 
\begin{equation*}
    A u \coloneqq (\widetilde{A}\tilde{u})_{|M_\tau}.
\end{equation*}
It is simple\footnote{Actually, it is possible to define the action of $A$ on $M_\tau$ without working with almost analytic extensions. Since $M_\tau$ is totally real, if $u \in C^\infty(M_\tau)$, then there is a unique jet at $M_\tau$ of function in $\mathbb{C}^2$, say $U$, that satisfies the Cauchy--Riemann equations and whose restriction to $M_\tau$ is the jet of $u$. We can then apply $\widetilde{A}$ to $U$ and evaluate the result on $M_\tau$. Notice that when $u$ is compactly supported $\tilde{u}$ is just obtained by Borel summation of $U$.} to check that the result does not depend on the choice of $\tilde{u}$. Notice that if $M_\tau$ is contained within $\mathbb{R}^2$ (which happens when $\tau = 0$) then this definition coincides with the original definition of $A$. 

It will be convenient to work on $M \subseteq \mathbb{R}^2$ rather than $M_\tau$. We will use the parametrization $\Xi(\tau,\cdot)$ to write operators on $M_\tau$ in coordinates for $M$. We let $A^{(\tau)}$ denote the operator defined on $M$ by 
\begin{equation*}
    A^{(\tau)} u =  A( u \circ \Xi(\tau,\cdot)^{-1}) \circ \Xi(\tau,\cdot)
\end{equation*}
for $u \in C^\infty(M)$. Hence, $A^{(\tau)}$ is just $A$ acting on $M_\tau$ but in the coordinates given by $\Xi(\tau,\cdot)$. Let $a$ be the homogeneous principal symbol of $A$:
\begin{equation*}
    a(\xi) = i^m \sum_{p + q = m} a_{p,q} \xi_1^p \xi_2^q. 
\end{equation*}
Then the homogeneous principal symbol $a^{(\tau)}$ of $A^{(\tau)}$ is given by the formula\footnote{In this formula $D_x \Xi(\tau,x)$ denote the derivative of the map $y \mapsto \Xi(\tau,y)$ at $y = x$, that we see as $2 \times 2$ invertible matrix with complex coefficients. To check the formula, one can notice that our definition of the action of a differential operator on $M_\tau$ commutes with multiplication of differential operator. Hence, one only needs to prove it for operator of order $0$ and $1$, which is just an application of the chain rule and of the definition.}
\begin{equation}\label{eq:transformation_symbol}
    a^{(\tau)}(x,\xi) = a({}^t D_x \Xi(\tau,x)^{-1} \xi) \textup{ for } (x,\xi) \in M \times \mathbb{R}^2.
\end{equation}

Of particular interest to us is the case $A = P_{\omega,0}$. Notice that if $\omega \in (0,1)$, then the symbol $p_{\omega,0}$ of $P_{\omega,0}$ is given by\footnote{Here, for $\xi = (\xi_1,\xi_2) \in \mathbb{C}^2$ and $x =(x_1,x_2)$ in $\mathbb{C}^2$, we write $\xi(x) = \xi_1 x_1 + \xi_2 x_2$, which is coherent with the interpretation of $\xi$ as a covector.}
\begin{equation*}
    p_{\omega,0}(\xi) = -4 \xi(L_\omega^+) \xi(L_\omega^-),
\end{equation*}
where $L_\omega^\pm$ are defined by \eqref{eq:basis_R2} with $\lambda$ replaced by $\omega$. It follows from Cramer's formula that
\begin{equation*}
\begin{split}
    & D_x \Xi(\tau,x)^{-1} L_{\omega}^{\pm}  =\frac{1}{\det(D_x \Xi(\tau,x))} \left( \ell_{\omega}^{\mp}(D_x \Xi(\tau,x) L_{\omega}^{\mp}) L_{\omega}^{\pm} - \ell_{\omega}^{\mp}(D_x \Xi(\tau,x) L_\omega^\pm)L_\omega^\mp \right).
\end{split}
\end{equation*}
Consequently, \eqref{eq:transformation_symbol} implies that
\begin{equation}\label{eq:formula_deformed_symbol}
\begin{split}
    & \frac{\det(D_x \Xi(\tau,x))^2}{4} p_{\omega,0}^{(\tau)}(x,\xi)  \\ & = - \left( \ell_\omega^-(L_\omega^- \Xi(\tau,x)) \ell_\omega^+(L_\omega^+ \Xi(\tau,x)) + \ell_\omega^-(L_\omega^+ \Xi(\tau,x)) \ell_\omega^+(L_\omega^- \Xi(\tau,x)) \right)\xi(L_\omega^+) \xi(L_\omega^-) \\
    & \qquad \qquad \qquad \qquad + \ell_\omega^-(L_\omega^- \Xi(\tau,x)) \ell_\omega^+(L_\omega^- \Xi(\tau,x))  \xi(L_\omega^+)^2 \\
    & \qquad \qquad \qquad \qquad + \ell_\omega^+(L_\omega^+ \Xi(\tau,x)) \ell_\omega^-(L_\omega^+ \Xi(\tau,x))  \xi(L_\omega^-)^2, \\
\end{split}
\end{equation}
for $\omega \in (0,1)$.

In order to prove Theorems \ref{theorem:main} and \ref{theorem:existence_deformation}, we need to establish a precise ellipticity property for the symbol $p_{\omega,0}^{(\tau)}$.

\begin{lemma}\label{lemma:inviscid_ellipticity}
There are $\tau_0 > 0, \delta > 0$ and a constant $C_0 > 0$ such that for every $x \in \overline{\Omega}$, $\xi \in \mathbb{R}^2$, $\tau \in[0,\tau_0)$ and $\omega \in (\lambda - \delta, \lambda + \delta) + i [0,+ \infty)$ we have:
\begin{equation*}
    \im p_{\omega,0}^{(\tau)}(x,\xi) +  C_0 \tau |\re p_{\omega,0}^{(\tau)}(x,\xi)| \geq C_0^{-1}(\tau + \im \omega)|\xi|^2.
\end{equation*}
\end{lemma}

\begin{remark}
Lemma \ref{lemma:inviscid_ellipticity} is an ellipticity estimate for $p_{\omega,0}^{(\tau)}$ (when $\tau > 0$ or $\im \omega > 0$) with the additional precision that the range of this symbol is contained within 
\begin{equation*}
    \set{r e^{i \theta} : r \in \mathbb{R}_+, \theta \in [- C_1 \tau, \pi + C_1 \tau]}
\end{equation*}
for some $C_1 > 0$.
\end{remark}

\begin{proof}[Proof of Lemma \ref{lemma:inviscid_ellipticity}]
Write $\omega = \omega_R + i \omega_I$ the decomposition of $\omega$ into real and imaginary parts. Let us consider first the symbol $p_{\omega_R,0}^{(\tau)}$. It follows from Proposition \ref{proposition:deformation_domain} that if $\delta$ is small enough and $\omega_R \in (\lambda- \delta,\lambda + \delta)$ then 
\begin{equation*}
    \im \ell_{\omega_R}^{\pm}(L_{\omega_R}^\mp \Xi(0,x)) > 0.
\end{equation*}
Hence, formula \eqref{eq:formula_deformed_symbol} implies that there is a constant $C > 0$ such that, provided $\tau$ and $\delta$ are small enough, we have
\begin{equation*}
    C^{-1}|\xi(L_{\omega_R}^+) \xi(L_{\omega_R}^-)| - C \tau |\xi|^2 \leq |\re p_{\omega_R,0}^{(\tau)}(x,\xi)| \leq C|\xi|^2
\end{equation*}
and
\begin{equation*}
    \im p_{\omega_R,0}^{(\tau)}(x,\xi) \geq C^{-1} \tau |\xi|^2 - C \tau |\xi(L_{\omega_R}^+) \xi(L_{\omega_R}^-)|.
\end{equation*}
In order to take into account the imaginary part of $\omega$, notice that
\begin{equation*}
    p_{\omega,0}(\xi) - p_{\omega_R,0}(\xi) = \omega_I(-\omega_I + 2 i \omega_R)|\xi|^2. 
\end{equation*}
Hence, we have\footnote{In the following formula, we write $|\eta|^2 = \eta_1^2 + \eta_2^2$ for $\eta = (\eta_1,\eta_2) \in \mathbb{C}^2$.}
\begin{equation*}
    p_{\omega,0}^{(\tau)}(x,\xi) - p_{\omega_R,0}^{(\tau)}(x,\xi) = \omega_I(-\omega_I + 2 i \omega_R)|{}^t D_x \Xi(\tau,x)^{-1} \xi|^2.
\end{equation*}
Since $D_x \Xi(0,x)$ is the identity, we find that for some $ C > 0$, we have
\begin{equation*}
    C^{-1} \omega_I^2 |\xi|^2 - C \tau \omega_I |\xi|^2 \leq |\re(p_{\omega,0}^{(\tau)}(x,\xi) - p_{\omega_R,0}^{(\tau)}(x,\xi))| \leq C \omega_I^2 |\xi|^2 + C \tau \omega_I |\xi|^2 
\end{equation*}
and
\begin{equation*}
    \im(p_{\omega,0}^{(\tau)}(x,\xi) - p_{\omega_R,0}^{(\tau)}(x,\xi)) \geq C^{-1} \omega_I |\xi|^2 - C \omega_I^2 \tau |\xi|^2.
\end{equation*}
Putting the estimates above together, and maybe taking a larger $C$, we get
\begin{equation*}
    \im p_{\omega,0}^{(\tau)}(x,\xi) \geq C^{-1}(\tau + \omega_I)|\xi|^2 - C \tau |\xi(L_{\omega_R}^+) \xi(L_{\omega_R}^-)| - C \omega_I^2 \tau |\xi|^2.
\end{equation*}
We have two estimates for the real part of the symbol (coming from the two versions of the second triangle inequality)
\begin{equation*}
    |\re(p_{\omega,0}^{(\tau)}(x,\xi))| \geq C^{-1}|\xi(L_{\omega_R}^+) \xi(L_{\omega_R}^-)| - C \tau |\xi|^2 - C \omega_I^2|\xi|^2 - C \tau \omega_I|\xi|^2
\end{equation*}
and
\begin{equation*}
    |\re(p_{\omega,0}^{(\tau)}(x,\xi))| \geq C^{-1} \omega_I^2 |\xi|^2 - C \tau \omega_I|\xi|^2 - C |\xi|^2.
\end{equation*}

Hence, maybe taking $C \geq 1$ even larger (in particular to duplicate some positive terms), we find that for $C_0 \geq 1$, we have
\begin{equation}\label{eq:ellipticity_first}
\begin{split}
    & \im p_{\omega,0}^{(\tau)}(x,\xi) +  C_0 \tau |\re p_{\omega,0}^{(\tau)}(x,\xi)| \\ & \qquad \qquad \qquad \qquad \geq C^{-1}(\tau + \omega_I)|\xi|^2 + (C_0 C^{-1} - C)\tau |\xi(L_{\omega_R}^+) \xi(L_{\omega_R}^-)| \\ & \qquad \qquad \qquad \qquad \qquad \qquad + (C^{-1} - C_0 C \omega_I \tau) \omega_I |\xi|^2 + (C^{-1} - C_0 C \tau)\tau|\xi|^2 \\& \qquad \qquad \qquad \qquad \qquad \qquad +(C^{-1} -C_0 C \tau^2) \omega_I |\xi|^2
\end{split}
\end{equation}
and
\begin{equation}\label{eq:ellipticity_second}
\begin{split}
    & \im p_{\omega,0}^{(\tau)}(x,\xi) +  C_0 \tau |\re p_{\omega,0}^{(\tau)}(x,\xi)| \\ & \qquad \qquad \qquad \qquad \geq C^{-1}(\tau + \omega_I)|\xi|^2 + (C_0 C^{-1} - C) \omega_I^2 \tau |\xi|^2 \\ & \qquad \qquad \qquad \qquad \qquad \qquad +(C^{-1} - C_0 C \tau^2)\omega_I|\xi|^2 + C_0(C^{-1} \omega_I^2 - C) \tau |\xi|^2.
\end{split}
\end{equation}
We can now choose $C_0$ and $\tau_0$. With $C$ the constant from \eqref{eq:ellipticity_first} and \eqref{eq:ellipticity_second}, we set $C_0 = C^2$, and then we impose $\tau_0 \leq C^{-5}$. From $C_0 = C^2$, we get rid of all the terms with a factor $C_0 C^{-1} - C$ in the right hand side of \eqref{eq:ellipticity_first} and \eqref{eq:ellipticity_second}. For $\tau \in [0,\tau_0)$, we have $\tau \leq (C_0 C^2)^{-1}$, and thus all the terms with a factor $C^{-1} - C_0 C \tau$ or $C^{-1} - C_0 C \tau^2$ are non-negative and can be removed. Apart from the first term (which is the one that we want), the only thing left in the right hand side of \eqref{eq:ellipticity_first} is $(C^{-1} - C_0 C \omega_I \tau) \omega_I |\xi|^2$ and in \eqref{eq:ellipticity_second} is $C_0(C^{-1} \omega_I^2 - C) \tau |\xi|^2$. If $\omega_I \tau \leq (C_0 C^2)^{-1}$, then the extra term in \eqref{eq:ellipticity_first} is non-negative, and we get the result. If $\omega_I \tau \geq (C_0 C^2)^{-1}$, then, using $\tau \leq \tau_0 \leq C^{-5}$, we have
\begin{equation*}
    C^{-1} \omega_I^2 - C \geq C^{-9} \tau^{-2} - C \geq C - C = 0.
\end{equation*}
Hence, in that case, the last factor in \eqref{eq:ellipticity_second} is non-negative, and the result follows.
\end{proof}

\section{Invertibility of the inviscid operator on the deformed domain}\label{section:inviscid_invertibility}

In this section, we end the proof of Theorem \ref{theorem:existence_deformation} by proving:

\begin{proposition}\label{proposition:inviscid_invertibility}
There are $\tau_0 > 0$ and $\delta > 0$ such that for every $\tau \in (0,\tau_0)$ and $\omega \in (\lambda- \delta, \lambda + \delta) + i [0,+\infty)$ the operator $P_{\omega,0} : H^2(\overline{\Omega}_\tau) \cap H_0^1(\overline{\Omega}_\tau) \to L^2(\overline{\Omega}_\tau)$ is invertible.
\end{proposition}

As explained in \S \ref{section:differential_operators}, we will rather prove that the operator $P_{\omega,0}^{(\tau)} : H^2(\Omega) \cap H_0^1(\Omega) \to L^2(\Omega)$ is invertible, which is equivalent\footnote{For $s \in \mathbb{R}$, the space $H^s(\overline{\Omega}_\tau)$ may be defined by embedding $\overline{\Omega}_\tau$ in a closed surface $N$ (e.g. a double of $\overline{\Omega}_\tau$) and then defining $H^s(\overline{\Omega}_\tau)$ as the space of the restrictions to $\overline{\Omega}_\tau$ of the elements of $H^s(N)$. We endow $H^s(\overline{\Omega}_\tau)$ with a Hilbert space structure by identifying it with the orthogonal of the subspace of $H^s(N)$ made of the distributions supported in $N \setminus \Omega_\tau$. We can then define $H^s_0(N)$ as the closure of $C_c^\infty(\Omega_\tau)$ in $H^s(\overline{\Omega}_\tau)$. The diffeomorphism $\Xi(\tau,\cdot) : \overline{\Omega} \to \overline{\Omega}_\tau$ induces natural isomorphisms between $H^s(\Omega)$ (resp. $H^s_0(\Omega)$) and $H^s(\overline{\Omega}_\tau)$ (resp. $H_0^s(\overline{\Omega}_\tau)$), as topological vector spaces.}. 

We start by proving invertibility of $P_{\omega,0}^{(\tau)}$ when $\im \omega > 0$ in \S \ref{subsection:upper_invertibility}. To do so, we use the fact that, in that case, $P_{\omega,0}^{(\tau)}$ is elliptic for $\tau \geq 0$ small (including the case $\tau = 0$). Hence, any function in the kernel of $P_{\omega,0}^{(\tau)}$ is analytic (using e.g. \cite[Chapitre 8, Théorème 1.2]{lions_magenes_3}). We will use this fact to show that the (absence of) invertibility of $P_{\omega,0}^{(\tau)}$ does not depend on $\tau \geq 0$ small. An integration by parts argument prove that $P_{\omega,0} : H^2(\Omega) \cap H_0^1(\Omega) \to L^2(\Omega)$ is invertible.

The rest of this section is dedicated to the proof of the invertibility of $P_{\lambda,0}^{(\tau)} : H^2(\Omega) \cap H_0^1(\Omega) \to L^2(\Omega)$ for $\tau > 0$ small. The strategy is similar as in the case $\im \omega > 0$, but, since $P_{\lambda,0}$ is not elliptic, it is harder to relate $P_{\lambda,0}^{(\tau)}$ when $\tau > 0$ with $P_{\lambda,0}$. The invertibility statement at $\tau = 0$ that we will use is \cite[Lemma 7.2]{dwz_internal_waves}. To be able to apply it, we will rely on the reduction to the boundary method from \cite{dwz_internal_waves}. In \S \ref{subsection:fundamental_solution}, we define an operator on $\Omega_\tau$ that plays the role of the operator of convolution by the fundamental solution $E_{\lambda + i 0}$ for $P_{\lambda,0}$ defined in \cite[\S 4.3.2]{dwz_internal_waves}. In \S \ref{subsection:reduction_boundary}, we use this operator to replace the equation $P_{\lambda,0}^{(\tau)} u = 0$ by an equation on $\partial \Omega$. In \S \ref{subsection:boundary_layer}, we study the behavior as $\tau$ goes to $0$ of the restricted layer potentials associated to the reduction to the boundary of the problem on $\Omega_\tau$. This behavior is very similar to what we observe on $\Omega$ as $\im \omega$ goes to $0$, a case that is detailed in \cite[\S 4.6]{dwz_internal_waves}. Thanks to the similarities between these two situations, we will be able in \S \ref{subsection:proof_invertibility} to apply the methods from \cite{dwz_internal_waves} to end the proof of Proposition \ref{proposition:inviscid_invertibility}.

\subsection{Invertibility when \texorpdfstring{$\im \omega > 0$}{imomega}.}\label{subsection:upper_invertibility}

Let us start with a result that reduces Proposition~\ref{proposition:inviscid_invertibility} to a question of injectivity.

\begin{lemma}\label{lemma:inviscid_index}
There are $\tau_0 > 0$ and $\delta > 0$ such that for every $\tau \in [0,\tau_0)$ and $\omega \in (\lambda- \delta, \lambda + \delta) + i [0,+\infty)$, if $\tau > 0$ or $\im \omega > 0$, the operator $P_{\omega,0}^{(\tau)}$ is properly elliptic, and in particular $P_{\omega,0}^{(\tau)} : H^2(\Omega) \cap H_0^1(\Omega) \to L^2(\Omega)$ is Fredholm of index zero.
\end{lemma}

\begin{proof}
Let $\tau_0 > 0$ and $\delta > 0$ be as in Lemma \ref{lemma:inviscid_ellipticity}. Let $\tau \in [0,\tau_0)$ and $\omega \in (\lambda- \delta, \lambda + \delta) + i [0,+\infty)$, and assume that $\tau > 0$ or $\im \omega > 0$. For $\varpi \geq 0$, let us consider the operator $Q_{\varpi} = P_{\omega,0}^{(\tau)} - i \varpi \Delta$, where $\Delta$ denote the Laplace operator on $\Omega$. It follows from Lemma \ref{lemma:inviscid_ellipticity} that $Q_\varpi$ is elliptic for every $\varpi \geq 0$. Notice that for $\varpi$ large enough the operator $Q_\varpi$ is strongly elliptic (\cite[Chapitre 2, Définition 1.3]{lions_magenes_1}). In particular, it is properly elliptic (\cite[Chapitre 2, Définition 1.2]{lions_magenes_1}). However, being properly elliptic is preserved under deformation in the class of elliptic operators. Hence, $Q_\varpi$ is properly elliptic for every $\varpi \geq 0$. In particular $P_{\omega,0}^{(\tau)} = Q_0$ is properly elliptic, which implies that $P_{\omega,0}^{(\tau)} : H^2(\Omega) \cap H_0^1(\Omega) \to L^2(\Omega)$ is Fredholm of index zero, see \cite[Chapitre 2, (8.4)]{lions_magenes_1}.
\end{proof}

We prove now Proposition \ref{proposition:inviscid_invertibility} in the case $\im \omega > 0$ by reducing to the case $\tau = 0$.

\begin{lemma}\label{lemma:positive_imaginary_part_invertibility}
There are $\tau_0 > 0$ and $\delta > 0$ such that for every $\tau \in [0,\tau_0)$ and $\omega \in (\lambda- \delta,\lambda + \delta) + i (0,+\infty)$ the operator $P_{\omega,0}^{(\tau)} : H^2(\Omega) \cap H_0^1(\Omega)  \to L^2(\Omega)$ is invertible.
\end{lemma}

\begin{proof}
Let $\tau_0$ and $\delta > 0$ be as in Lemma \ref{lemma:inviscid_index}. Let $\omega \in (\lambda- \delta,\lambda + \delta) + i (0,+\infty)$. For every $\tau \in [0,\tau_0)$, the dimension of the kernel of the operator $P_{\omega,0}^{(\tau)} : H^2(\Omega) \cap H_0^1(\Omega) \to L^2(\Omega)$ is finite. Let us prove that the function $\tau \mapsto \dim \ker P_{\omega,0}^{(\tau)}$ is locally constant (and thus constant) on $[0,\tau_0)$.

Pick $\tau_1 \in [0,\tau_0)$. Let $u_1,\dots,u_n$ be a basis for $\ker P_{\omega,0}^{(\tau_1)}$. It follows from \cite[Chapitre 8, Théorème 1.2]{lions_magenes_3} that $u_1,\dots,u_n$ are real-analytic (recalling from Lemma \ref{lemma:inviscid_index} that $P_{\omega,0}^{(\tau_1)}$ is properly elliptic and thus that the Dirichlet boundary condition covers $P_{\omega,0}^{(\tau_1)}$ on $\partial \Omega$). For $j = 1,\dots,n$, let $v_j$ be the real-analytic function on $\overline{\Omega}_{\tau_1}$ defined by $v_j(\Xi(\tau,x)) = u_j(x)$. These functions  have holomorphic extensions, that we still denote by $v_1,\dots,v_n$, to a connected neighbourhood $U$ of $\overline{\Omega}_{\tau_1}$. For $\tau \in [0,\tau_0)$ close enough to $\tau_1$, the set $\bigcup_{\tau' \in [\tau,\tau_1]} \overline{\Omega}_{\tau'}$ is contained within $U$ and we can define a map
\begin{equation*}
    \begin{array}{ccccc}
        \Phi & : & \ker P_{\omega,0}^{(\tau_1)} & \to & \ker P_{\omega,0}^{(\tau)}  \\
             &   & \sum_{j = 1}^n a_j u_j & \mapsto & \left(x \mapsto \sum_{j = 1}^n a_j v_j(\Xi(\tau,x))\right).
    \end{array}
\end{equation*}
It follows from the analytic continuation principle that $\Phi$ maps indeed $\ker P_{\omega,0}^{(\tau_1)}$ to $\ker P_{\omega,0}^{(\tau)}$ (in particular, Dirichlet boundary condition is satisfied by functions in the range of $\Phi$ because $\partial \Omega_{\tau_1}$ and $\partial \Omega_{\tau}$ are contained in $(\partial \Omega)_{\mathbb{C}}$). The analytic continuation principle also implies that $\Phi$ is injective. Hence, $\dim \ker P_{\omega,0}^{(\tau_1)} \leq \dim \ker P_{\omega,0}^{(\tau)}$ for $\tau$ close to $\tau_1$.

Now, let $f_1,\dots,f_n$ be a linearly independent family in $L^2(\Omega)$ that spans a supplementary subspace for the range of $P_{\omega,0}^{(\tau_1)}$. There are $n$ elements (the dimension of the kernel of $P_{\omega,0}^{(\tau_1)}$) in this family because $P_{\omega,0}^{(\tau_1)}$ is Fredholm of index zero. For $\tau$ close to $\tau_1$, consider the operator
\begin{equation*}
    \begin{array}{ccccc}
         \mathcal{P}(\tau) & : & H^2(\Omega) \cap H^1_0(\Omega) \times \mathbb{C}^n & \to & L^2(\Omega) \times \mathbb{C}^n  \\
          & & (u,x_1,\dots,x_n) & \mapsto & \left(P_{\omega,0}^{(\tau)} u + \sum_{j = 1}^n x_j f_j, (\langle u, u_j \rangle)_{1 \leq j \leq n} \right),
    \end{array}
\end{equation*}
where the scalar products are in $L^2(\Omega)$. Notice that $\mathcal{P}(\tau_1)$ is invertible. Hence for $\tau$ close to $\tau_1$ the operator $\mathcal{P}(\tau)$ is invertible and $(\ker P_{\omega,0}^{(\tau)}) \times \set{0}$ is contained in $\mathcal{P}(\tau)^{-1}(\set{0} \times \mathbb{C}^n)$. It follows that $\dim \ker P_{\omega,0}^{(\tau)} \leq n = \dim \ker P_{\omega,0}^{(\tau_1)}$, proving that $\dim \ker P_{\omega,0}^{(\tau)} = \dim \ker P_{\omega,0}^{(\tau_1)}$.

We proved that the function $\tau \mapsto \dim \ker P_{\omega,0}^{(\tau)}$ is constant on $[0,\tau_0)$. Hence, thanks to Lemma \ref{lemma:inviscid_index}, we only need to prove that $P_{\omega,0} : H^2(\Omega) \cap H_0^1(\Omega) \to L^2(\Omega)$ is injective. Let $u \in H^2(\Omega) \cap H_0^1(\Omega)$ be such that $P_{\omega,0} u = 0$. An integration by part proves that
\begin{equation*}
    0 = \im \left( \int_\Omega P_{\omega,0} u \bar{u} \,\mathrm{d}x \right) = 2 \re \omega \im \omega \int_\Omega |\nabla u|^2 \,\mathrm{d}x.
\end{equation*}
By taking $\delta$ small enough, we ensure that $\re \omega > 0$ and by assumption we have $\im \omega > 0$, and thus Poincaré's inequality implies that $u =0$.
\end{proof}

\subsection{Fundamental solution on the deformed domain}\label{subsection:fundamental_solution}

The rest of this section is dedicated to the proof of the invertibility of $P_{\lambda,0}^{(\tau)} : H^2(\Omega) \cap H_0^1(\Omega) \to L^2(\Omega)$ for $\tau > 0$ small enough. As explained in \S \ref{section:differential_operators}, it is convenient to introduce a small neighbourhood $M$ of $\overline{\Omega}$ in $\mathbb{R}^2$ and define for $\tau \geq 0$ small the surface $M_\tau = \Xi(\tau,M) \subseteq \mathbb{C}^2$.

To prove the invertibility of $P_{\lambda,0}^{(\tau)}$, we will rely on the reduction to the boundary method from \cite{dwz_internal_waves}, our goal being to apply the invertibility statement \cite[Lemma 7.2]{dwz_internal_waves}. To do so, we need to explain how the reduction to the boundary method applies to $P_{\lambda,0}^{(\tau)}$. A crucial tool in the application of this method in \cite{dwz_internal_waves} is the fundamental solution $E_{\lambda + i 0}$ for $P_{\lambda,0}$ defined by 
\begin{equation*}
    E_{\lambda + i 0}(x) = \frac{i}{4 \pi \lambda \sqrt{1 - \lambda^2}} \log( A_\lambda(x) + i 0)
\end{equation*}
where
\begin{equation*}
    A_\lambda(x) = \ell_\lambda^+(x) \ell_\lambda^{-}(x) \textup{ and } \log(y + i 0) = \begin{cases}
    \log y & \textup{ if } y \in \mathbb{R}_+^*, \\
    \log(-y) + i \pi & \textup{ if } y \in \mathbb{R}_-^*.
    \end{cases}
\end{equation*}
Hence, $E_{\lambda + i 0}$ is a locally integrable function (defined almost everywhere). It is proven in \cite[\S 4.3.2]{dwz_internal_waves} that $E_{\lambda + i 0}$ is a fundamental solution for $P_{\lambda,0}$. 

We need an operator on $M_\tau$ that will play the role of the operator of convolution by $E_{\lambda + i 0}$. Hence, for $\tau > 0$ small we introduce the operator $E_\lambda^{(\tau)} : C_c^\infty(M) \to C^\infty(M)$ given by\footnote{Recall that for $\tau$ close to $0$ and $y \in M$, we use the notation $D_x \Xi(\tau,y)$ for the derivative at $y$ of the map $\Xi(\tau,\cdot)$. We see this derivative as a $2 \times 2$ matrix with complex coefficients.}
\begin{equation}\label{eq:fundamental_solution}
    E_\lambda^{(\tau)}u(x) = \frac{i}{4 \pi \lambda \sqrt{1 - \lambda^2}}\int_{M} \log(A_\lambda(\Xi(\tau,x) - \Xi(\tau,y))) \det(D_x \Xi(\tau,y)) u(y) \,\mathrm{d}y
\end{equation}
for $u \in C_c^\infty(M)$ and $x \in M$. In this formula, we use the holomorphic logarithm on $\mathbb{C} \setminus i \mathbb{R}_-$ which is real-valued on $\mathbb{R}_+$. It is not clear a priori that the integrand in \eqref{eq:fundamental_solution} makes sense, nor that $E_\lambda^{(\tau)}$ maps $C_c^\infty(M)$ into $C^\infty(M)$. We prove it in Lemma \ref{lemma:well_defined_fundamental_solution} after some preparation. 

Before starting to study the operator $E_{\lambda}^{(\tau)}$, let us motivate its definition. We introduced $E_{\lambda}^{(\tau)}$ as an operator on $M$, but it is more natural to think of it as an operator on $M_\tau$. Let $u \in C_c^\infty(M)$ and $v \in C_c^\infty(M_\tau)$ and $w \in C^\infty(M_\tau)$ be such that $u = v \circ \Xi(\tau,\cdot)$ and $E_{\lambda}^{(\tau)} u = w \circ \Xi(\tau,\cdot)$, then for $x \in M_\tau$ we have (by the change of variable formula)
\begin{equation}\label{eq:intrinsic_fundamental_solution}
    w(x) = \frac{i}{4 \pi \lambda \sqrt{1 - \lambda^2}} \int_{M_\tau} \log(A_\lambda(x-y)) v(y) \,\mathrm{d}y.
\end{equation}
In this formula, the $2$-form $\mathrm{d}y$ denotes (the restriction to $M_\tau$ of) $(\mathrm{d}x_1 + i \mathrm{d}y_1) \wedge (\mathrm{d}x_2 + i \mathrm{d}y_2)$, in coordinates $(x_1 + i y_1, x_2+ i y_2)$ on $\mathbb{C}^2$. Hence, $E_{\lambda}^{(\tau)}$ is just defined by the same formula as the operator of convolution by $E_{\lambda + i 0}$, with $M$ replaced by $M_\tau$.

Let us study the operator $E_{\lambda}^{(\tau)}$. We start with two merely technical results. It is in the first of these results, Lemma \ref{lemma:linear_forms_deformed}, that we fix our choice of neighbourhood $M$ for $\overline{\Omega}$. 

\begin{lemma}\label{lemma:linear_forms_deformed}
We may choose $M$ such that the following holds. There are $\tau_0 > 0$, a constant $C > 0$ and smooth functions $v_{\pm} : \overline{M} \times \overline{M} \to \mathbb{R}_+^*, w_\pm : \overline{M} \times \overline{M} \to \mathbb{C}$ and $H_\pm : [0,\tau_0] \times \overline{M} \times \overline{M} \to \mathbb{C}$ such that for $(\tau,x,y) \in [0,\tau_0] \times\overline{M} \times\overline{M}$ we have
\begin{equation}\label{eq:deformation_linear_form}
\begin{split}
    \ell_{\lambda}^{\pm} (\Xi(\tau,x) - \Xi(\tau,y)) = \ell_{\lambda}^{\pm} (x-y) & + i \tau v_{\pm}(x,y) \ell_\lambda^{\mp}(x-y) \\ &+ \tau w_\pm(x,y) \ell_{\lambda}^\pm(x-y) + \tau^2 H_\pm(\tau,x,y)
\end{split}
\end{equation}
and $|H_\pm(\tau,x,y)| \leq C|x-y|$.
\end{lemma}

\begin{proof}
For $x \in \partial \Omega$, let us denote by $\mathbf{n}(x)$ the inward pointing normal of $\partial \Omega$ at $x$. For some small $r_0$, let us define
\begin{equation*}
    N = \Omega \cup \set{ x  + r \mathbf{n}(x) : x \in \partial \Omega, r \in (-r_0,r_0)}.
\end{equation*}
Hence, $N$ is a domain with smooth boundary, and $\partial N$ is a small deformation of $\partial \Omega$, so that $N$ satisfies the $\lambda$-simplicity condition (by taking $r_0$ small enough). We also assume that $r_0$ is small enough so that $\Xi$ has a real-analytic extension to $\overline{N} \times [0,\tau_0]$ such that for every $x$ in a neighbourhood of $\overline{N}$ we have
\begin{equation}\label{eq:extended_transversality}
    \ell_\lambda^{\pm}\left( \frac{\partial L_\lambda^{\mp} \Xi}{\partial \tau}(0,x) \right) \in i \mathbb{R}_+^*.
\end{equation}
This is a consequence of Proposition \ref{proposition:deformation_domain}, and we apply the analytic continuation principle to guarantee the absence of real part in \eqref{eq:extended_transversality}.

Finally, let $M$ be any open neighbourhood of $\overline{\Omega}$ such that $\overline{M} \subseteq N$ and let us prove \eqref{eq:deformation_linear_form}. To lighten notations, we will prove it only in the ``$+$'' case. Let $x,y$ be two points in $M$. We start with the case in which $\ell_\lambda^{-}(x-y)$ or $\ell_\lambda^+(x-y)$ is small. We have
\begin{equation*}
    x - y = \ell_\lambda^+(x-y) L_\lambda^+ + \ell_\lambda^-(x-y) L_\lambda^-.
\end{equation*}
Since $\overline{M}$ is a compact subset of $N$, if $\ell_\lambda^{-}(x-y)$ or $\ell_\lambda^+(x-y)$ is small enough the point $y + \ell_\lambda^+(x-y) L_\lambda^+$ belongs to $N$. By $\lambda$-simplicity of $N$, we deduce that the segments from $y$ to $y + \ell_\lambda^+(x-y) L_\lambda^+$ and from $y + \ell_\lambda^+(x-y) L_\lambda^+$ to $x$ are contained within $N$. Hence, the fundamental theorem of calculus yields
\begin{equation*}
\begin{split}
    & \Xi(\tau,x) - \Xi(\tau,y) \\ & \qquad \qquad = \ell_\lambda^+(x-y) \int_0^{1} L_\lambda^+ \Xi(\tau,y + s \ell_\lambda^+(x-y) L_{\lambda}^+) \,\mathrm{d}s \\ & \qquad  \qquad \qquad \quad + \ell_\lambda^-(x-y) \int_0^{1} L_\lambda^- \Xi (\tau,y + \ell_{\lambda}^+(x-y) L_{\lambda}^+ + s \ell_\lambda^-(x-y) L_\lambda^-)\,\mathrm{d}s.
\end{split}
\end{equation*}
It follows that
\begin{equation*}
\begin{split}
    & \frac{\partial}{\partial \tau}\left( \ell_\lambda^+(\Xi(\tau,x) - \Xi(\tau,y)) \right)_{|\tau = 0} \\
    & \qquad \qquad = \ell_\lambda^+(x-y) \int_0^{1} \ell_{\lambda}^+\left( \frac{\partial L_\lambda^+ \Xi}{\partial \tau}(0,y + s \ell_\lambda^+(x-y) L_{\lambda}^+) \right) \,\mathrm{d}s \\ & \qquad \qquad \quad + \ell_\lambda^-(x-y) \int_0^{1} \ell_\lambda^+ \left(\frac{\partial L_\lambda^- \Xi}{\partial \tau} (0,y + \ell_{\lambda}^+(x-y) L_{\lambda}^+ + s \ell_\lambda^-(x-y) L_\lambda^- )\right) \,\mathrm{d}s.
\end{split}
\end{equation*}
The property \eqref{eq:deformation_linear_form} follows then from Taylor's formula, with the control of the range of $v_+$ obtained from \eqref{eq:extended_transversality}.

It remains to prove \eqref{eq:deformation_linear_form} when $\ell_\lambda^+(x-y)$ and $\ell_\lambda^-(x-y)$ are away from zero. From Taylor's formula, we have
\begin{equation*}
    \ell_\lambda^{+}(\Xi(\tau,x) - \Xi(\tau,y)) = \ell_\lambda^+(x-y) + \tau h_+(x,y) + \tau^2 H_+(\tau,x,y)
\end{equation*}
for some smooth functions $h_+$ and $H_{+}$. Since $x$ and $y$ are away from each other, the upper bound on $H_+$ is an empty condition. We can then set
\begin{equation*}
    v_+(x,y) = 1 \textup{ and } w_+(x,y) = \frac{h_+(x,y) - i \ell_\lambda^-(x-y)}{\ell_\lambda^+(x-y)}
\end{equation*}
and \eqref{eq:deformation_linear_form} follows. Notice that we proved \eqref{eq:deformation_linear_form} only locally in $(x,y) \in \overline{M} \times \overline{M}$. The local expressions can be glued using a partition of unity.
\end{proof}

Let us now apply Lemma \ref{lemma:linear_forms_deformed} to the quantity $\log(A_\lambda(\Xi(\tau,x) - \Xi(\tau,y)))$ that appears in the defining formula \eqref{eq:fundamental_solution} for $E_{\lambda}^{(\tau)}$.

\begin{lemma}\label{lemma:definition_log}
There are $\tau_0 > 0$ and smooth functions $w : \overline{M} \times \overline{M} \to \mathbb{C}$ and $g : [0,\tau_0] \times \overline{M} \times \overline{M} \to \mathbb{C}$ such that:
\begin{enumerate}[label = (\roman*)]
    \item for every $(\tau,x,y) \in [0,\tau_0] \times M \times M$ we have \label{item:quadratic_deformed}
    \begin{equation}\label{eq:factorization}
        A_\lambda(\Xi(\tau,x) - \Xi(\tau,y)) = (1 + \tau w(x,y))(A_\lambda(x-y) + i \tau g(\tau,x,y)),
    \end{equation}
    \item there is a constant $C$ such that for every $(\tau,x,y) \in [0,\tau_0] \times M \times M$ with $x \neq y$ we have \label{item:localization_g}
    \begin{equation*}
        \frac{g(\tau,x,y)}{|x-y|^2} \in [C^{-1},C] + i[-C\tau,C\tau].
    \end{equation*}
\end{enumerate}
In particular, up to making $\tau_0$ smaller, for every $\tau \in (0,\tau_0)$ there is a constant $\alpha_0 \in [0,\pi/2)$ such that for every $x,y \in M$ with $x \neq y$ we have
\begin{equation*}
    A_\lambda(\Xi(\tau,x) - \Xi(\tau,y)) \in \set{ r e^{i \theta} : r \in \mathbb{R}_+^*, \theta \in [- \alpha_0,\pi + \alpha_0)}.
\end{equation*}
\end{lemma}

\begin{proof}
Applying Lemma \ref{lemma:linear_forms_deformed} and expanding the definition of $A_\lambda$, we find that the formula \eqref{eq:factorization} is satisfied if we set $w(x,y) = w_+(x,y) + w_-(x,y)$ and
\begin{equation*}
\begin{split}
    & (1 + \tau w(x,y)) g(\tau,x,y) \\ & \qquad \qquad = v_-(x,y) \ell_\lambda^+(x-y)^2 + v_+(x,y) \ell_\lambda^-(x-y)^2 \\ & \qquad \qquad \qquad -i \tau H_+(\tau,x,y) \ell_{\lambda}^{-}(x-y) -i \tau H_-(\tau,x,y) \ell_\lambda^+(x-y) \\ & \qquad \qquad \qquad - i \tau (iv_+(x,y) \ell_\lambda^- (x-y) + w_+(x,y)\ell_\lambda^+(x-y) + \tau H_+(\tau,x,y))\\ & \qquad \qquad \qquad \qquad \times (i v_-(x,y) \ell_\lambda^+(x-y) + w_-(x,y) \ell_\lambda^-(x-y) + \tau H_-(\tau,x,y)).
\end{split}
\end{equation*}
The second formula can be used to define $g(\tau,x,y)$ when $\tau$ is small enough. Notice also that
\begin{equation*}
    g(\tau,x,y) = v_-(x,y) \ell_\lambda^+(x-y)^2 + v_+(x,y) \ell_\lambda^-(x-y)^2 + \mathcal{O}(\tau|x-y|^2)
\end{equation*}
and \ref{item:localization_g} follows then from the fact that $v_+$ and $v_-$ takes positive values.

From \ref{item:quadratic_deformed} and \ref{item:localization_g}, we deduce that there is $C > 1$ such that for $x,y \in M$ and $\tau > 0$ small we have
\begin{equation*}
    |\re A_\lambda(\Xi(\tau,x) - \Xi(\tau,y))| \geq C^{-1} |A_\lambda(x-y)| - C \tau|x-y|^2
\end{equation*}
and
\begin{equation*}
    \im A_\lambda(\Xi(\tau,x) - \Xi(\tau,y)) \geq C^{-1}\tau|x-y|^2 - C \tau |A_\lambda(x-y)|.
\end{equation*}
Hence, we have for $\tau \in (0, (2C^4)^{-1})$:
\begin{equation}\label{eq:to_define_fundamental_solution}
    \im A_\lambda(\Xi(\tau,x) - \Xi(\tau,y)) + C^2 \tau |\re A_\lambda(\Xi(\tau,x) - \Xi(\tau,y))| \geq \frac{\tau}{2C}|x-y|^2.
\end{equation}
The second part of the result follows.
\end{proof}

We are now ready to prove that $E_{\lambda}^{(\tau)}$ indeed defines an operator from $C_c^\infty(M)$ to $C^\infty(M)$.

\begin{lemma}\label{lemma:well_defined_fundamental_solution}
There is $\tau_0 > 0$ such that for every $\tau \in (0,\tau_0)$ the function
\begin{equation}\label{eq:kernel_fundamental}
    (x,y) \mapsto \log(A_\lambda(\Xi(\tau,x) - \Xi(\tau,y))) \det(D_x \Xi(\tau,y))
\end{equation}
is integrable on $M \times M$. Hence, $E_\lambda^{(\tau)}$ is well-defined. Moreover, $E_\lambda^{(\tau)}$ maps continuously $C_c^\infty(M)$ into $C^\infty(M)$ and has a continuous extension from $\mathcal{E}'(M)$ to $\mathcal{D}'(M)$.
\end{lemma}

\begin{proof}
Let $\tau_0$ be as in Lemma \ref{lemma:definition_log} and let $\tau \in (0,\tau_0)$. The log is well-defined thanks to the last part of Lemma \ref{lemma:definition_log} (we are using the holomorphic determination of the logarithm on $\mathbb{C} \setminus i \mathbb{R}_-$ that is real on $\mathbb{R}_+^*$). Recalling the estimate \eqref{eq:to_define_fundamental_solution} from the proof of this lemma, we find that there is a constant $C > 0$ (that depends on $\tau$) such that for $x,y \in M$ we have
\begin{equation*}
    \left| \log(A_\lambda(\Xi(\tau,x) - \Xi(\tau,y))) \right| \leq C (1 + |\log|x-y||).
\end{equation*}
The integrability of \eqref{eq:kernel_fundamental} follows, and thus the definition of $E_\lambda^{(\tau)}$ as an operator from $C_c^\infty(M)$ to $L^1(M)$.

Let us now prove that $E_\lambda^{(\tau)}$ maps $C_c^\infty(M)$ into $C^\infty(M)$ continuously. Take a function $u \in C_c^\infty(M)$ and $x_0 \in M$. We will prove that $E_\lambda^{(\tau)} u$ is smooth near $x_0$. To do so, pick a function $\chi : M \to [0,1]$ supported near $x_0$ and split
\begin{equation*}
\begin{split}
    &- 4 i \pi \lambda \sqrt{1 - \lambda^2} E_\lambda^{(\tau)} u (x) \\ & \qquad \qquad = \underbrace{\int_M \log(A_\lambda(\Xi(\tau,x) - \Xi(\tau,y))) \det(D_x \Xi(\tau,y)) \chi(y) u(y) \,\mathrm{d}y}_{v(x) \coloneqq} \\ & \qquad \qquad \qquad + \underbrace{\int_M \log(A_\lambda(\Xi(\tau,x) - \Xi(\tau,y))) \det(D_x \Xi(\tau,y))(1 - \chi(y)) u(y) \,\mathrm{d}y}_{w(x) \coloneqq}.
\end{split}
\end{equation*}
It is clear that $w$ is smooth near $x_0$ (because the kernel is smooth there). To study $v$ we use Taylor's formula to get 
\begin{equation*}
    A_\lambda(\Xi(\tau,x) - \Xi(\tau,y)) = B(\tau,x,y) \cdot (x-y)
\end{equation*}
where $B(\tau,x,y)$ is a quadratic form $\mathbb{R}^2 \to \mathbb{C}$ with coefficients that depend smoothly on $\tau,x$ and $y$. Then, we change variables to get for some small $\delta > 0$:
\begin{equation*}
\begin{split}
    v(x) & = \int_{B(0,\delta)} \log(B(\tau,x,x+z) \cdot z) F(x,z) u(x+z) \,\mathrm{d}z \\
         & = \int_{B(0,\delta)} \log|z| F(x,z) u(x+z) \,\mathrm{d}z \\ & \qquad \qquad + \frac{1}{2}\int_{B(0,\delta)} \log\left(B(\tau,x,x+z) \cdot \frac{z}{|z|} \right) F(x,z) u(x+z) \,\mathrm{d}z
\end{split}
\end{equation*}
where $F$ is a smooth function. We can then differentiate under the integral in each term to get that $v$ is smooth. 

We can prove as above that the formal adjoint of $E_{\lambda}^{(\tau)}$ is bounded $C_c^\infty(M) \to C^\infty(M)$, and it follows that $E_\lambda^{(\tau)}$ has a continuous extension $\mathcal{E}'(M) \to \mathcal{D}'(M)$.
\end{proof}

Let us state another consequence of Lemma \ref{lemma:definition_log} that will be useful later. We also justify a claim from Remark \ref{remark:transversality}.

\begin{lemma}\label{lemma:diffeomorphisms_deformation}
There is $\tau_0 > 0$ such that for every $\tau \in (0,\tau_0)$ the maps $\ell_\lambda^+$ and $\ell_\lambda^-$ induce $C^\infty$ diffeomorphisms from $M_\tau$ to open subsets of $\mathbb{C}$. Moreover, these diffeomorphisms are respectively orientation preserving and reversing (the orientation of $M_\tau$ is deduced from the orientation of $M$).
\end{lemma}

\begin{proof}
Let us consider for $\tau > 0$ small the map $F_\tau^{\pm} : M \to \mathbb{C}$ defined by $F_\tau^{\pm}(x) = \ell_\lambda^\pm(\Xi(\tau,x))$. For $x \in M$, the matrices of $DF_\tau^+(x)$ and $DF_\tau^-(x)$ are of the form
\begin{equation*}
    \begin{bmatrix}
    1 + \mathcal{O}(\tau)& \mathcal{O}(\tau) \\
    \mathcal{O}(\tau) & -i \tau \ell_\lambda^+(\frac{\partial L_\lambda^- \Xi}{\partial \tau}(0,x)) + \mathcal{O}(\tau^2)
    \end{bmatrix}
\end{equation*}
and
\begin{equation*}
    \begin{bmatrix}
    \mathcal{O}(\tau) & 1 + \mathcal{O}(\tau) \\
    -i \tau \ell_\lambda^-(\frac{\partial L_\lambda^+ \Xi}{\partial \tau}(0,x)) + \mathcal{O}(\tau^2) & \mathcal{O}(\tau)
    \end{bmatrix}.
\end{equation*}
Here, we used $(L_\lambda^+, L_\lambda^-)$ as a basis for $\mathbb{R}^2$ and $(1,i)$ as a basis for $\mathbb{C}$. Hence, for $\tau$ small enough we have $\pm \det(DF_{\tau}^{\pm}(x)) > 0$. Consequently, for $\tau$ small enough, $F_\tau^\pm$ induces a local diffeomorphism from $M$ to $\mathbb{C}$. In particular, $F_\tau^{\pm}(M)$ is open. It follows from Lemma \ref{lemma:definition_log} that the restriction of $F_{\tau}^\pm$ to $M$ is injective, and the result follows.
\end{proof}

We end this subsection by proving the main property of $E_{\lambda}^{(\tau)}$: it is a left inverse for $P_{\lambda,0}^{(\tau)}$ acting on compactly supported distributions.

\begin{proposition}\label{proposition:fundamental_solution}
There is $\tau_0 > 0$ such that for every $\tau \in (0,\tau_0)$ and every $u \in \mathcal{E}'(M)$ we have
\begin{equation*}
    E_\lambda^{(\tau)}(P_{\lambda,0}^{(\tau)}u) = u.
\end{equation*}
\end{proposition}

\begin{proof}
We only need to prove the result for $u \in C_c^\infty(M)$ (the result follows then by a density argument). We will use the description \eqref{eq:intrinsic_fundamental_solution} of $E_\lambda^{(\tau)}$ to adapt the proof that $E_{\lambda + i 0}$ is a fundamental solution from \cite[\S 4.3.2]{dwz_internal_waves}. Pick $u \in C_c^\infty(M)$. Define a smooth function $v$ on $M_\tau$ by $v(\Xi(\tau,x)) = u(x)$ and another function $w$ such that $w(\Xi(\tau,x)) = - 4i \pi \lambda \sqrt{1 - \lambda^2} E_\lambda^{(\tau)}(P_{\lambda,0}^{(\tau)} u)(x)$. By the change of variable formula (as in \eqref{eq:intrinsic_fundamental_solution}), we have for $x \in M_\tau$:
\begin{equation*}
    w(x) = \int_{M_\tau} \log(A_\lambda(x-y)) P_{\lambda,0} v(y) \,\mathrm{d}y.
\end{equation*}
Let us fix $x = \Xi(\tau,\tilde{x}) \in M_\tau$ with $\tilde{x} \in M$. For $\mu > 0$ small, we let $D_{\mu,\tau}$ denote the image by $\Xi(\tau,\cdot)$ of the disc of center $\tilde{x}$ and radius $\mu$ in $\mathbb{R}^2$. We have then by dominated convergence
\begin{equation*}
    w(x) = \lim_{\mu \to 0} 4 \int_{M \setminus D_{\mu,\tau}} \log(A_\lambda(x-y)) L_\lambda^+ L_\lambda^- v(y) \,\mathrm{d}y.
\end{equation*}
Notice then that we have
\begin{equation*}
    \begin{split}
        & \log(A_\lambda(x-y)) L_\lambda^+ L_\lambda^- v(y) \mathrm{d}y \\
        &  \ \  =  \frac{\lambda \sqrt{1 - \lambda^2}}{2} \log(A_\lambda(x-y)) L_\lambda^+ L_\lambda^- v(y) \mathrm{d}\ell_\lambda^+ \wedge \mathrm{d}\ell_\lambda^- \\
          & \ \  = \mathrm{d}_y \left( \frac{\lambda \sqrt{1 - \lambda^2}}{2} \log(A_\lambda(x-y)) L_\lambda^- v(y) \mathrm{d}\ell_\lambda^- \right)  + \frac{\lambda \sqrt{1 - \lambda^2}}{2} \frac{L_\lambda^- v(y)}{\ell_\lambda^+(x-y)}\mathrm{d}\ell_\lambda^+ \wedge \mathrm{d}\ell_\lambda^- \\
          & \ \  = \mathrm{d}_y \left( \frac{\lambda \sqrt{1 - \lambda^2}}{2} \log(A_\lambda(x-y)) L_\lambda^- v(y) \mathrm{d}\ell_\lambda^- \right)  - \mathrm{d}_y \left( \frac{\lambda \sqrt{1 - \lambda^2}}{2} \frac{v(y)}{\ell_\lambda^+(x-y)}\mathrm{d}\ell_\lambda^+ \right).
    \end{split}
\end{equation*}

Applying Stokes formula we get (the boundary term with a logarithmic factor goes to zero with $\mu$)
\begin{equation*}
\begin{split}
    w(x) & = -2 \lambda \sqrt{1 - \lambda^2} \lim_{\mu \to 0} \int_{\partial D_{\mu,\tau}} \frac{v(y)}{\ell_\lambda^+(x-y)} \,\mathrm{d}\ell_\lambda^+ \\
         & = 2 \lambda \sqrt{1 - \lambda^2} \lim_{\mu \to 0} \int_{\ell_\lambda^+(\partial D_{\mu,\tau})} \frac{v((\ell_\lambda^+)^{-1}(z))}{z - \ell_\lambda^+(x)} \,\mathrm{d}z \\
         & = 4 i \pi \lambda \sqrt{1 - \lambda^2} v(x)
\end{split}
\end{equation*}
where on the second line $(\ell_\lambda^+)^{-1}$ denotes the inverse the diffeomorphism $M_\tau \to \ell_\lambda^+(M_t)$ induced by $\ell_\lambda^+$ (see Lemma \ref{lemma:diffeomorphisms_deformation}). In view of the definitions of $v$ and $w$, we just proved $E_\lambda^{(\tau)}(P_{\lambda,0}^{(\tau)}u)(x) = u(x)$, and the result follows.
\end{proof}

\subsection{Reduction to the boundary}\label{subsection:reduction_boundary}

In this subsection, we explain how we can use the operator $E_{\lambda}^{(\tau)}$ to represent a solution of the equation $P_{\lambda,0}^{(\tau)} u = 0$ (with Dirichlet boundary condition) using a $1$-form on $\partial \Omega$. We will consider distributions supported on the boundary of $\Omega$, and introduce, following \cite{dwz_internal_waves}, the operator\footnote{We use $\mathcal{D}'(\partial \Omega,T^* \partial \Omega \otimes \mathbb{C})$ to denote the space of generalized sections of the vector bundle $T^* \partial \Omega \otimes \mathbb{C}$. Notice that they are just the current of degree $1$ on $\partial \Omega$, with complex coefficients. The space of smooth sections of $T^* \partial \Omega \otimes \mathbb{C}$ is denoted by $\Gamma(T^* \partial \Omega \otimes \mathbb{C}$).} $\mathcal{I} : \mathcal{D}'(\partial \Omega, T^* \partial \Omega \otimes \mathbb{C}) \to \mathcal{E}'(M)$ defined by
\begin{equation*}
    \mathcal{I}(v)(\varphi) = \int_{\partial \Omega} \varphi v \textup{ for } v \in \mathcal{D}'(\partial \Omega, T^* \partial \Omega \otimes \mathbb{C}) \textup{ and } \varphi \in C^\infty(M).
\end{equation*}

We start with a lemma that describes the action of $E_{\lambda}^{(\tau)}$ on distributions of the form $\mathcal{I}(v)$. The operator $S_\lambda^{(\tau)}$ introduced in Lemma \ref{lemma:fundamental_boundary} is related to the single layer potential operator $S_{\lambda + i 0}$ from \cite[\S 4.5]{dwz_internal_waves}.

\begin{lemma}\label{lemma:fundamental_boundary}
There is $\tau_0 > 0$ such that for every $\tau \in (0,\tau_0)$ and every $v \in \Gamma(T^* \partial \Omega \otimes \mathbb{C})$ the distribution $E_\lambda^{(\tau)}(\mathcal{I}(v))_{|\Omega}$ is smooth on $\Omega$ and has a smooth extension $S_\lambda^{(\tau)} v$ to $\overline{\Omega}$. Moreover, $S_\lambda^{(\tau)}$ defines a bounded operator from $\Gamma(T^* \partial \Omega \otimes \mathbb{C})$ to $C^\infty(\overline{\Omega})$.
\end{lemma}

\begin{proof}
Pick $v \in \Gamma(T^* \partial \Omega \otimes \mathbb{C})$. Let $\tilde{v}$ be the $1$-form on $\partial \Omega_\tau$ such that $\det(\Xi(\tau,\cdot)) v = \Xi(\tau,\cdot)^* \tilde{v}$. Let $w$ be the distribution on $\Omega_\tau$ defined by the relation $w(\Xi(\tau,x)) = S_\lambda^{(\tau)}v(x)$ and notice that $w$ is a smooth function (because the kernel of $E_\lambda^{(\tau)}$ is smooth away from the diagonal) given, for $x \in \Omega_\tau$, by
\begin{equation*}
    w(x) = \frac{i}{4 \pi \lambda \sqrt{1 - \lambda^2}} \int_{\partial \Omega_\tau} \log(A_\lambda(x-y)) \tilde{v}(y). 
\end{equation*}
By differentiation under the integral, we have for $x \in \Omega_\tau$
\begin{equation*}
    L_\lambda^\pm w(x) = \frac{i}{4 \pi \lambda \sqrt{1 - \lambda^2}} \int_{\partial \Omega_\tau} \frac{\tilde{v}(y)}{\ell_\lambda^\pm(x-y)}.
\end{equation*}
Recall from Lemma \ref{lemma:diffeomorphisms_deformation} that $\ell_\lambda^\pm$ induces a diffeomorphism between $M_\tau$ and an open subset, say $W$, of $\mathbb{C}$. Let $F : W \to M_\tau$ denote the inverse diffeomorphism. The change of variable formula yields for $x \in F^{-1}(\Omega_\tau)$:
\begin{equation*}
    L_\lambda^\pm w(F(x)) = \frac{i}{4 \pi \lambda \sqrt{1 - \lambda^2}} \int_{\partial \ell_\lambda^\pm(\Omega_\tau)} \frac{F^* \tilde{v}(y)}{x-y}
\end{equation*}
and it follows from \cite[Lemma 4.6]{dwz_internal_waves} that $L_\lambda^\pm w$ has a smooth extension to $\overline{\Omega}_\tau$. It follows that $\mathrm{d}w$, and thus $w$, have smooth extension to $\overline{\Omega}_\tau$ (because $L_\lambda^+$ and $L_\lambda^-$ span $\mathbb{C}^2$). Hence, $v$ has a smooth extension to $\overline{\Omega}$.
\end{proof}

Using the operator $S_\lambda^{(\tau)}$ from Lemma \ref{lemma:fundamental_boundary}, we give a representation formula for solution of $P_{\lambda,0}^{(\tau)} u = 0$ with Dirichlet boundary condition. We could theoretically give an explicit formula for the $1$-form $v$ in the following lemma (in terms of the first order derivative of $u$ on $\partial \Omega$), but we will not need it.

\begin{lemma}\label{lemma:boundary_representation}
There is $\tau_0 > 0$ such that  for every $\tau \in (0,\tau_0)$ and every $u \in C^\infty(\overline{\Omega})$ such that $u_{|\partial \Omega} = 0$ and $P_{\lambda,0}^{(\tau)} u = 0$, there is $v \in \Gamma(T^* \partial \Omega \otimes \mathbb{C})$ such that
\begin{equation*}
    u = S_\lambda^{(\tau)}v \textup{ and } E_{\lambda}^{(\tau)}(\mathcal{I}(v))_{|M \setminus \overline{\Omega}} = 0.
\end{equation*}
\end{lemma}

\begin{proof}
Let $u \in C^\infty(\overline{\Omega})$ be such that $u_{|\partial \Omega} = 0$ and $P_{\lambda,0}^{(\tau)} u = 0$. Let $U$ be the function on $M$ defined by
\begin{equation*}
    U(x) = \begin{cases}
    u(x) & \textup{ if } x \in \overline{\Omega} \\
    0 & \textup{ if } w \in M \setminus \overline{\Omega}.
    \end{cases}
\end{equation*}
Notice that $U$ defines an element of $\mathcal{E}'(M)$, and thus Proposition \ref{proposition:fundamental_solution} implies that $E_{\lambda}^{(\tau)} P_{\lambda,0}^{(\tau)} U = U$. Hence, we only need to prove that there is $v \in \Gamma(T^* \partial \Omega \otimes \mathbb{C})$ such that $P_{\lambda,0}^{(\tau)} U = \mathcal{I}(v)$. This is a consequence of Stokes formula, using that $P_{\lambda,0}^{(\tau)}$ is a differential operator of order $2$ and that $u_{|\partial \Omega} = 0$.
\end{proof}

We end this subsection by a lemma that relates the operator $E_{\lambda}^{(\tau)}$, when $\tau$ goes to $0$, with the operator of convolution with $E_{\lambda + i 0}$.

\begin{lemma}\label{lemma:support_condition}
Let $(\tau_n)_{n \geq 0}$ be a sequence of strictly positive real numbers converging to $0$. For each $n \in \mathbb{N}$, let $v_n \in \Gamma(T^* \partial \Omega \otimes \mathbb{C})$ be such that $E_{\lambda}^{(\tau_n)}(\mathcal{I}(v_n))_{|M \setminus \overline{\Omega}} = 0$. Assume that there is a current $v \in \mathcal{D}'(\partial \Omega, T^* \partial \Omega \otimes \mathbb{C})$ such that $(v_n)_{n \geq 0}$ converges as $n$ goes to $+ \infty$ to $v$ (as a current). Then $(E_{\lambda + i 0} \ast \mathcal{I}(v))_{|M \setminus \overline{\Omega}} = 0$.
\end{lemma}

\begin{proof}
Let $\varphi \in C_c^\infty(M \setminus \overline{\Omega})$. We have 
\begin{equation*}
    \langle E_{\lambda + i 0} \ast \mathcal{I}(v), \varphi \rangle = \langle \mathcal{I}(v -v_n), E_{\lambda + i 0} \ast \varphi \rangle + \langle \mathcal{I}(v_n), E_{\lambda  + i 0} \ast \varphi - (E_{\lambda}^{(\tau_n)})^*(\varphi) \rangle,
\end{equation*}
where $(E_{\lambda}^{(\tau_n)})^*$ denotes the formal adjoint of $E_\lambda^{(\tau_n)}$. From the convergence assumption on $(v_n)_{n \geq 0}$, we know that 
\begin{equation*}
\langle \mathcal{I}(v -v_n), E_{\lambda + i 0} \ast \varphi \rangle \underset{n \to +\infty}{\to} 0.
\end{equation*}
Hence, we only need to prove that
\begin{equation*}
\langle \mathcal{I}(v_n), E_{\lambda  + i 0} \ast \varphi - (E_{\lambda}^{(\tau_n)})^*(\varphi) \rangle \underset{n \to + \infty}{\to} 0.
\end{equation*}
It follows from Banach--Steinhauss Theorem (see e.g. \cite[Theorem 2.1.8]{hormander1}) that $(v_n)_{n \geq 0}$ is bounded in the space of currents of some given order. Hence, we only need to prove that $( (E_{\lambda}^{(\tau_n)})^*(\varphi))_{n \geq 0}$ converges to $E_{\lambda + i 0} \ast \varphi$ as a $C^\infty$ function in a neighbourhood of $\partial \Omega$.

Let us consequently fix a point $x_0 \in \partial \Omega$, we will prove convergence near that point. Moreover, up to introducing a partition of unity, we may assume that the support of $\varphi$ is small, close to a point $x_1 \in M \setminus \overline{\Omega}$. For each $n \geq 0$, let us introduce the smooth function $\psi_n$ defined by
\begin{equation*}
    \psi_n(x) = \frac{i}{4 \pi \lambda \sqrt{1 - \lambda^2}} \int_M \log(A_\lambda(\Xi(\tau_n,x) - \Xi(\tau_n,y))) \varphi(y) \,\mathrm{d}y.
\end{equation*}
Notice that 
\begin{equation*}
    (E_{\lambda}^{(\tau_n)})^* \varphi (x) = \psi_n(x) \det(D_x \Xi(\tau_n,x))
\end{equation*}
and the function $\det(D_x \Xi(\tau_n,\cdot))$ converges to $1$ in $C^\infty(M)$ as $n$ goes to $+ \infty$. Hence, we want to prove the convergence of $(\psi_n)_{n \geq 0}$ to $E_{\lambda + i0} \ast \varphi$. Applying Lemma \ref{lemma:definition_log}, we have
\begin{equation*}
\begin{split}
    \psi_n(x) & = \frac{i}{4 \pi \lambda \sqrt{1 - \lambda^2}} \int_M \log(1 + \tau_n w(x,y)) \varphi(y) \,\mathrm{d}y \\
        & \qquad \qquad + \frac{i}{4 \pi \lambda \sqrt{1 - \lambda^2}} \underbrace{ \int_M \log(A_\lambda(x-y) + i \tau_n g(\tau_n,x,y)) \varphi(y) \,\mathrm{d}y}_{\tilde{\psi}_n(x) \coloneqq}.
\end{split}
\end{equation*}
By differentiation under the integral, we find that the first term in this formula converges to $0$ as a smooth function, and thus we focus on $\tilde{\psi}_n$. Using the assumption that $\varphi$ is supported near $x_1$, we make a change of variable ``$u = \ell_\lambda^+(y-x), v = \ell_\lambda^{-}(y-x)$'' to get, for some arbitrarily small $\rho > 0$ and for $\bar{x} = (\ell_\lambda^+(x_1 - x_0),\ell_\lambda^-(x_1-x_0))$,
\begin{equation*}
\begin{split}
    \tilde{\psi}_n(x) = \int_{B(\bar{x},\rho)} \log(uv + i \tau_n g(\tau_n, x, x & + u L_\lambda^+ + v L_\lambda^-))  \varphi(x + u L_\lambda^+ + v L_\lambda^-) \,\mathrm{d}u \mathrm{d}v.
\end{split}
\end{equation*}
Let $(\alpha,\beta) \in \mathbb{R}^2$ be such that $\langle (\beta,\alpha),\bar{x} \rangle > 0$. We may find $\alpha$ and $\beta$ because $x_1$ is away from $x_0$ and thus $\bar{x}$ is non-zero. By taking $\rho$ small enough, we ensure that for every $(u,v) \in B(\bar{x},\rho)$ we have $\beta u + \alpha v > \rho$. Let then $\tilde{g}_n(x,\cdot,\cdot)$ and $\tilde{\varphi}$ be almost analytic extensions\footnote{They are smooth extensions of the functions $(u,v) \mapsto g(\tau_n,x,x +u L_\lambda^+ + v L_\lambda^-)$ and $\varphi$ defined on a complex neighbourhood respectively of $B(\bar{x},\rho)$ and $x_0$. Moreover, these extensions satisfy the Cauchy--Riemann equations at infinite order on the intersection of $\mathbb{R}^2$ with their domains of definitions. The analytic extensions $\tilde{g}_n$ may be chosen satisfying uniform $C^\infty$ estimates in $n \geq 0$. See e.g. \cite[Theorem 3.6]{zworski_book}.} for $(u,v) \mapsto g(\tau_n,x,x +u L_\lambda^+ + v L_\lambda^-)$ and $\varphi$. Let $H : B(\bar{x},\rho) \times [0,1] \to \mathbb{C}^2$ denotes the homotopy $H(u,v,t) = (u + i t \alpha, v+ i t \beta)$.

From Lemma \ref{lemma:definition_log} and the choice of $\alpha,\beta$, we find that there is $C > 0$ such that for $t \geq 0$ small, $n \geq 0$ large and $x$ close to $x_0$, if $(u,v) \in B(\bar{x},\rho) + i t (\alpha,\beta)$ then $\im (uv + i \tau_n \tilde{g}_n(x,u,v)) \geq C^{-1}(t + \tau_n)$. In particular, $(u,v) \mapsto \log ( uv + i \tau_n \tilde{g}_n(x,u,v))$ is smooth on a neighbourhood of $ B(\bar{x},\rho) + i [0,t_0](\alpha,\beta)$ for some small $t_0 > 0$, where $t_0$ does not depend on $n$.

We can consequently apply Stoke's formula to get
\begin{equation*}
\begin{split}
    \tilde{\psi}_n(x) = & \int_{B(\bar{x},\rho) + i t_0(\alpha,\beta)} \log(uv + i \tau_n \tilde{g}_n(x, u, v)) \tilde{\varphi}(x + u L_\lambda^+ + v L_\lambda^-) \,\mathrm{d}u \mathrm{d}v \\
    & - \int_{B(\bar{x},\rho) \times [0,t_0]} H^*(\bar{\partial} (\log(uv + i \tau_n \tilde{g}_n(x, u, v)) \tilde{\varphi}(x + u L_\lambda^+ + v L_\lambda^-)) \wedge \mathrm{d}u \wedge\mathrm{d}v).
\end{split}
\end{equation*}

Thanks to the estimate $\im (uv + i \tau_n \tilde{g}_n(x,u,v)) \geq C^{-1}(t + \tau_n)$, we find that the integrand of the first integral is uniformly smooth when $n$ goes to $+ \infty$ and the coefficients of the form in the second integral, and their derivatives with respect to $x$, are uniformly $\mathcal{O}(t^\infty)$. Hence, $\tilde{\psi}_n$ converges (as a smooth function near $x_0$) to the function $\tilde{\psi}_\infty$ given by
\begin{equation*}
   \begin{split}
     & \tilde{\psi}_\infty(x) = \int_{B(\bar{x},\rho) + i t_0(\alpha,\beta)} \log(uv ) \tilde{\varphi}(x + u L_\lambda^+ + v L_\lambda^-) \,\mathrm{d}u \mathrm{d}v \\
    & \qquad \qquad -\int_{B(\bar{x},\rho) \times [0,t_0]} H^*(\bar{\partial} (\log(uv) \tilde{\varphi}(x + u L_\lambda^+ + v L_\lambda^-)) \wedge \mathrm{d}u \wedge\mathrm{d}v).
\end{split} 
\end{equation*}
Notice that $\tilde{\psi}_\infty(x)$ does not depend on the choice of $t_0 > 0$ small in its definition above (by Stokes formula, or just by uniqueness of the limit). Letting $t_0$ go to $0$, the second term go to $0$ (the integrand is $\mathcal{O}(t^\infty)$), and thus we have by dominated convergence
\begin{equation*}
    \tilde{\psi}_\infty(x) = \int_{B(\bar{x},\rho)} \log( uv + i 0) \varphi(x + u L_\lambda^+ + v L_\lambda^-) \,\mathrm{d}u \mathrm{d}v.
\end{equation*}
Finally, the change of variable ``$y = x + u L_\lambda^+ + v L_\lambda^-$'' yields $$\tilde{\psi}_\infty = - 4 i \pi \lambda \sqrt{1 - \lambda^2} E_{\lambda+i 0} \ast \varphi.$$
This completes the proof.
\end{proof}

\subsection{Microlocal structure of the restricted layer potential operators}\label{subsection:boundary_layer}

Our goal is to understand the $1$-form $v$ given by Lemma \ref{lemma:boundary_representation} when $u$ solves $P_{\lambda,0}^{(\tau)} u = 0$ with Dirichlet boundary condition. Notice that $v$ must satisfy $(S_\lambda^{(\tau)} v)_{|\partial \Omega} = 0$, which suggest to introduce the operator $\mathcal{C}_{\lambda}^{(\tau)} : \Gamma(T^* \partial \Omega \otimes \mathbb{C}) \to C^\infty(\partial \Omega)$ defined by $\mathcal{C}_{\lambda}^{(\tau)} w = (S_{\lambda}^{(\tau)}w)_{|\partial \Omega}$.

The operator $\mathcal{C}_\lambda^{(\tau)}$ is related to the restricted single layer potential $\mathcal{C}_{\lambda + i 0}$ from \cite[\S 4.6]{dwz_internal_waves}. Indeed, we have:

\begin{lemma}\label{lemma:convergence_kernel}
The Schwartz kernel of $\mathcal{C}_{\lambda}^{(\tau)}$ converges as $\tau$ goes to $0$ to the Schwartz kernel of $\mathcal{C}_{\lambda + i 0}$ as a distribution.
\end{lemma}

\begin{proof}
Applying Lemma \ref{lemma:definition_log}, we find that the Schwartz kernel of $\mathcal{C}_\lambda^{(\tau)}$ is
\begin{equation*}
\begin{split}
    & \frac{i}{4 \pi \lambda \sqrt{1 - \lambda^2}}\log(A_\lambda(x-y) + i \tau g(\tau,x,y)) \det D_x \Xi(\tau,y) \\ & \qquad \qquad \qquad \qquad + \frac{i}{4 \pi \lambda \sqrt{1 - \lambda^2}} \log(1 + \tau w(x,y)) \det D_x \Xi(\tau,y).
\end{split}
\end{equation*}
The second term converges to $0$ as a smooth function on $\partial \Omega \times \partial \Omega$ when $\tau$ goes to $0$, and we can consequently ignore it. The factor $\det D_x \Xi(\tau,y)$ converges to $1$ as a smooth function and can also be ignored. Hence, we need to prove the convergence of 
\begin{equation}\label{eq:converging_log}
    (x,y) \mapsto \log(A_\lambda(x-y) + i \tau g(\tau,x,y))
\end{equation}
to
\begin{equation}\label{eq:limiting_log}
    (x,y) \mapsto \log(A_\lambda(x-y) + i 0)
\end{equation}
as a distribution on $\partial \Omega \times \partial\Omega$. Notice that \eqref{eq:converging_log} is an integrable function and that it converges almost everywhere to \eqref{eq:limiting_log}. Thus, we only need to establish a $\tau$-independent domination for \eqref{eq:converging_log}. To do so we write
\begin{equation*}
\begin{split}
    & \log(A_\lambda(x-y) + i \tau g(\tau,x,y)) \\
    & \qquad \qquad \qquad= \log|A_\lambda(x-y) + i \tau g(\tau,x,y)|  + \log\left(\frac{A_\lambda(x-y) + i \tau g(\tau,x,y)}{ |A_\lambda(x-y) + i \tau g(\tau,x,y)|}\right).
\end{split}
\end{equation*}
The second term is bounded by $3\pi/2$. To bound the first term, we need to establish a lower bound for $|A_\lambda(x-y) + i \tau g(\tau,x,y)| $. We have for some $C > 0$, recalling Lemma~\ref{lemma:definition_log}~\ref{item:localization_g},
\begin{equation*}
\begin{split}
    & |A_\lambda(x-y) + i \tau g(\tau,x,y)| \\ & \qquad \qquad \geq \frac{|\re(A_\lambda(x-y) + i \tau g(\tau,x,y))|  + |\im(A_\lambda(x-y) + i \tau g(\tau,x,y))| }{2} \\ & \qquad \qquad \geq \frac{|A_\lambda(x-y)| - C |x-y|^2 \tau^2 + C^{-1} \tau|x-y|^2 }{2} \geq \frac{|A_\lambda(x-y)|}{2}
\end{split}
\end{equation*}
for $\tau$ small enough. Since $\log|A_\lambda(x-y)| = \log |\ell_\lambda^+(x-y)| + \log|\ell_\lambda^-(x-y)|$ is integrable on $\partial \Omega \times \partial \Omega$, the result follows. The integrability of each term may be checked locally in $(x,y)$: if $\ell_\lambda^\pm(x-y) \neq 0$, there is nothing to say, if $\ell_\lambda^\pm(x-y) = 0$ and $x$ is not a characteristic point, then we may invert $\ell_\lambda^\pm$ locally and use that the logarithm is integrable, finally if $x = y$ is a characteristic point, we may work in Morse coordinates for $\ell_\lambda^\pm$ and use again that the logarithm is integrable (as it is done in the proof of \cite[Lemma 4.10]{dwz_internal_waves}).
\end{proof}

In the rest of this section, we prove more precise estimates on the behavior of $\mathcal{C}_\lambda ^{(\tau)}$ as $\tau$ goes to $0$. We will prove analogue statements to Lemmas 4.11, 4.12, 4.13 and 4.14 in \cite{dwz_internal_waves}, replacing the parameter ``$\epsilon = \im \omega$'' from this reference by our parameter $\tau$. The ultimate goal is to use the results from \cite{dwz_internal_waves} that follow from Lemmas 4.11, 4.12, 4.13 and 4.14 in order to invert $P_{\lambda,0}^{(\tau)}$, see Proposition \ref{proposition:microlocal_lasota_yorke} and the discussion above it.

Let us introduce for $\tau > 0$ small two (complex) $1$-forms on $M$, the pullbacks of $\mathrm{d}(\ell^+_{\lambda{|M_\tau}})$ and $\mathrm{d}(\ell_{\lambda|M_\tau}^-)$ by $\Xi(\tau,\cdot)$. We will denote them by $\mathrm{d}\ell_\lambda^{+,(\tau)}$ and $\mathrm{d}\ell_\lambda^{-,(\tau)}$. We can then decompose for $v \in \Gamma(T^* \partial \Omega \otimes \mathbb{C})$:
\begin{equation*}
    \mathrm{d} \mathcal{C}_\lambda^{(\tau)} v = T_\lambda^{+,(\tau)}v + T_\lambda^{-,(\tau)}v,
\end{equation*}
where 
\begin{equation*}
    T_{\lambda}^{\pm,(\tau)}v = j^*(L_\lambda^{\pm,(\tau)} S_\lambda^{(\tau)} v \mathrm{d}\ell_\lambda^{\pm,(\tau)}).
\end{equation*}
Here $j$ denotes the inclusion $\partial \Omega \to \overline{\Omega}$ and $L_\lambda^{\pm,(\tau)}$ is defined using the procedure from~\S \ref{section:differential_operators}.

Using the diffeomorphism $z : \mathbb{S}^1 \to \partial \Omega$, we identify $T_{\lambda}^{\pm,(\tau)}$ with an operator from $C^\infty(\mathbb{S}^1)$ to itself. We denote its Schwartz kernel by $K_{\pm}^{(\tau)}(\theta,\theta')$. We will also see $\gamma_\lambda^+$ and $\gamma_\lambda^-$ as diffeomorphisms on $\mathbb{S}^1$ using the parametrization $z$, as in the proof of Lemma~\ref{lemma:first_vector_field}.

We start\footnote{Following \cite{dwz_internal_waves}, we state lemmas with symbols ``$\pm$'' and ``$\mp$''. What we mean is that these results hold when all the instances of ``$\pm$'' and ``$\mp$'' are replaced by ``$+$'' and ``$-$'' respectively, and when all the instances of ``$\pm$'' and ``$\mp$'' are replaced by ``$-$'' and ``$+$'' respectively.} with the analogue of \cite[Lemma 4.11]{dwz_internal_waves}, which describes the kernel $K_{\pm}^{(\tau)}$ away from the diagonal of $\mathbb{S}^1 \times \mathbb{S}^1$ and the reflection points (i.e. the graphs of $\gamma_\lambda^+$ and $\gamma_\lambda^-$).

\begin{lemma}\label{lemma:non_characteristic_non_diagonal}
Let $\theta_0,\theta_1 \in \mathbb{S}^1$ be such that $\theta_0 \neq \theta_1$ and $\gamma_{\lambda}^\pm(\theta_0) \neq \theta_1$. Then there is a neighbourhood $U$ of $(\theta_0,\theta_1) \in \mathbb{S}^1 \times \mathbb{S}^1$ such that $K_{\pm}^{(\tau)}$ is smooth on $U$, uniformly as $\tau > 0$ goes to $0$.
\end{lemma}

\begin{proof}
Since $E_\lambda^{(\tau)}$ has a smooth kernel away from the diagonal, we find that $T_\lambda^{\pm,(\tau)}$ has a smooth kernel away from the diagonal, namely for $\theta \neq \theta'$:
\begin{equation*}
    K_{\pm}^{(\tau)}(\theta,\theta') = \frac{i}{4 \pi \lambda \sqrt{1 - \lambda^2}} \frac{\ell_\lambda^{\pm}(D_x \Xi(\tau,z(\theta)) \cdot z'(\theta))}{\ell_\lambda^\pm(\Xi(\tau,z(\theta)) - \Xi(\tau,z(\theta')))} \det(D_x \Xi(\tau,z(\theta')).
\end{equation*}
Notice that Lemma \ref{lemma:diffeomorphisms_deformation} implies that the denominator is non-zero when $\theta \neq \theta'$ and $\tau > 0$. The denominator stays away from zero as $\tau$ goes to $0$ if $(\theta,\theta')$ is close to $(\theta_0,\theta_1)$, and the result follows.
\end{proof}

We prove now the analogue of \cite[Lemma 4.12]{dwz_internal_waves} and describe $K_{\pm}^{(\tau)}$ near the points of the diagonal of $\mathbb{S}^1 \times \mathbb{S}^1$ that correspond to non-characteristic points of $\partial \Omega$.

\begin{lemma}\label{lemma:non_characteristic_diagonal}
Let $\theta_0 \in \mathbb{S}^1$ be such that $\gamma_\lambda^{\pm}(\theta_0) \neq \theta_0$. There are $\tau_0 > 0$, a neighbourhood $U$ of $\theta_0$ in $\mathbb{S}^1$ and a smooth function $R : [0,\tau_0) \times U \times U \to \mathbb{C}$ such that 
\begin{equation*}
    K_\pm^{(\tau)}(\theta,\theta') = \frac{i}{4 \pi \lambda \sqrt{1 - \lambda^2}}(\theta - \theta' \pm i 0)^{-1} \det(D_x \Xi(\tau,z(\theta'))) + R(\tau,\theta,\theta')
\end{equation*}
for $\tau \in (0,\tau_0)$ and $\theta,\theta' \in U$.
\end{lemma}

\begin{proof}
We have
\begin{equation}\label{eq:definition_kernel_boundary}
\begin{split}
    &K_\pm^{(\tau)}(\theta,\theta') \\ & \quad = \frac{i}{4 \pi \lambda \sqrt{1 - \lambda^2}} \lim\limits_{\delta \to 0+} \frac{\ell_\lambda^{\pm}(D_x \Xi (\tau,z(\theta)) \cdot z'(\theta))}{\ell_\lambda^\pm(\Xi(\tau,z(\theta) + \delta \mathbf{n}(\theta)) - \Xi(\tau,z(\theta')))} \det(D_x \Xi(\tau,z(\theta')),
\end{split}
\end{equation}
where $\mathbf{n}(\theta)$ denotes a non-zero inward-pointing vector at $z(\theta)$. To prove \eqref{eq:definition_kernel_boundary}, we just use the fact that the kernel of $E_{\lambda}^{(\tau)}$ is smooth away from the diagonal and the definition of $\mathcal{C}_\lambda^{(\tau)}$. We can then write
\begin{equation}\label{eq:normal_expansion}
    \ell_\lambda^\pm(\Xi(\tau,z(\theta) + \delta \mathbf{n}(\theta)) - \Xi(\tau,z(\theta'))) = \ell_\lambda^\pm(\Xi(\tau,z(\theta)) - \Xi(\tau,z(\theta'))) + \delta k(\delta,\tau,\theta)
\end{equation}
where $k$ is a smooth function such that $k(0,\tau,\theta) = \ell_\lambda^\pm(D_x\Xi(\tau,z(\theta)) \cdot \mathbf{n}(\theta))$. We also have
\begin{equation*}
    \ell_\lambda^\pm(\Xi(\tau,z(\theta)) - \Xi(\tau,z(\theta'))) = G(\tau,\theta,\theta')(\theta - \theta')
\end{equation*}
where $G$ is a smooth function such that $G(\tau,\theta,\theta) = \ell_\lambda^\pm(D_x \Xi(\tau,z(\theta)) \cdot z'(\theta))$. Consequently, we have
\begin{equation*}
    \frac{1}{\ell_\lambda^\pm(\Xi(\tau,z(\theta) + \delta \mathbf{n}(\theta)) - \Xi(\tau,z(\theta')))} = \frac{1}{G(\tau,\theta,\theta')} \frac{1}{\theta - \theta' + \delta \frac{k(\delta,\tau,\theta)}{G(\tau,\theta,\theta')}}.
\end{equation*}
Notice then that 
\begin{equation*}
\begin{split}
    \pm \im \frac{k(0,\tau,\theta)}{G(\tau,\theta,\theta)} & = \pm \im \frac{\ell_\lambda^\pm(D_x\Xi(\tau,z(\theta)) \cdot \mathbf{n}(\theta))}{\ell_\lambda^\pm(D_x \Xi(\tau,z(\theta)) \cdot z'(\theta))} \\ & = \pm \frac{\im(\ell_\lambda^\pm(D_x\Xi(\tau,z(\theta)) \cdot \mathbf{n}(\theta))\overline{\ell_\lambda^\pm(D_x \Xi(\tau,z(\theta)) \cdot z'(\theta))})}{|\ell_\lambda^\pm(D_x \Xi(\tau,z(\theta)) \cdot z'(\theta))|^2} > 0.
\end{split}
\end{equation*}
Here, we used the fact that $(z'(\theta),\mathbf{n}(\theta))$ is a direct basis of $\mathbb{R}^2$ and Lemma \ref{lemma:diffeomorphisms_deformation}. By continuity, we get that $\pm \im k(\delta,\tau,\theta)/ G(\tau,\theta,\theta')$ is still positive when $\delta$ is close to $0$ and $\theta$ and $\theta'$ are close to $\theta_0$.

It follows then from \cite[Lemma 3.7]{dwz_internal_waves} (see also the Remark on p.32) that, for $\theta,\theta'$ near $\theta_0$
\begin{equation*}
\begin{split}
    & K_\pm^{(\tau)}(\theta,\theta') =  \frac{i}{4 \pi \lambda \sqrt{1 - \lambda^2}} \frac{\ell_\lambda^+(D_x \Xi(\tau,z(\theta)) \cdot z'(\theta))}{G(\tau,\theta,\theta')} \det(D_x \Xi(\tau,z(\theta'))) \frac{1}{\theta - \theta' \pm i 0}.
\end{split}
\end{equation*}
Let $R(\tau,\theta,\theta')$ be the smooth function such that 
\begin{equation*}
\begin{split}
    & \frac{i}{4 \pi \lambda \sqrt{1 - \lambda^2}} \frac{\ell_\lambda^+(D_x \Xi(\tau,z(\theta)) \cdot z'(\theta))}{G(\tau,\theta,\theta')} \det(D_x \Xi(\tau,z(\theta')) \\ & \qquad \qquad = \frac{i}{4 \pi \lambda \sqrt{1 - \lambda^2}} \det(D_x \Xi(\tau,z(\theta'))) + (\theta - \theta') R(\tau,\theta,\theta')
\end{split}
\end{equation*}
and the result follows.
\end{proof}

We describe now $K_{\pm}^{(\tau)}$ near the graph of $\gamma_\lambda^\pm$, except at the intersection with the diagonal, and prove the analogue of \cite[Lemma 4.13]{dwz_internal_waves}.

\begin{lemma}\label{lemma:non_characteristic_reflection}
Let $\theta_0 \in \mathbb{S}^1$ be such that $\gamma_\lambda^\pm (\theta_0) \neq \theta_0$. There are $\tau_0 > 0$, a neighbourhood $U$ of $\theta_0$ in $\mathbb{S}^1$ and smooth functions $a,\psi : [0,\tau_0) \times U \to \mathbb{C}$ and $R : [0,\tau_0) \times U \times \gamma_\lambda^\pm(U) \to \mathbb{C}$ such that for $(\tau,\theta,\theta') \in [0,\tau_0) \times U \times \gamma_\lambda^\pm(U)$ we have
\begin{equation*}
    K_\pm^{(\tau)}(\theta,\theta') = a(\tau,\theta')(\gamma_\lambda^\pm(\theta) - \theta' \pm i \psi(\tau,\theta'))^{-1} + R(\tau,\theta,\theta').
\end{equation*}
Moreover, $\re \psi(\tau,\theta') > 0$ for $\tau$ small and $a(0,\theta') = \frac{i}{4 \pi \lambda \sqrt{1 - \lambda^2}} \frac{1}{(\gamma_\lambda^\pm)'(\theta')}$.
\end{lemma}

\begin{proof}
For $\tau > 0$ small and $\theta,\theta'$ respectively close to $\theta_0$ and $\gamma_\lambda^\pm (\theta_0)$, we have applying Lemma \ref{lemma:linear_forms_deformed}
\begin{equation}\label{eq:linear_forms_boundary}
\begin{split}
    & \ell_\lambda^\pm(\Xi(\tau,z(\theta)) - \Xi(\tau,z(\theta'))) \\ & \qquad \qquad  = (1 + \tau w_\pm(z(\theta),z(\theta')))\ell_\lambda^\pm(z(\theta) - z(\theta')) \\ &  \qquad \qquad \qquad \qquad + i \tau v_\pm(z(\theta),z(\theta')) \ell_\lambda^\mp(z(\theta) - z(\theta')) + \tau^2 H(\tau,z(\theta),z(\theta')).
\end{split}
\end{equation}
Let us also notice that
\begin{equation*}
    \ell_\lambda^\pm(z(\theta) - z(\theta')) = \ell_\lambda^\pm(z(\gamma_\lambda^+(\theta)) - z(\theta')) = G(\theta,\theta')(\gamma_\lambda^\pm(\theta) - \theta'),
\end{equation*}
where $G$ is a smooth function near $(\theta_0,\gamma_\lambda^\pm(\theta_0))$ with $G(\theta,\gamma_\lambda^\pm(\theta)) = \ell_\lambda^\pm(z'(\gamma_\lambda^\pm(\theta)))$. We will also set $q_\pm(\tau,\theta,\theta') = 1 + \tau w_\pm(z(\theta),z(\theta'))$ and notice that $q_\pm(\tau,\cdot,\cdot)$ converges to $1$ as a smooth function when $\tau$ goes to $0$. For $\tau$ small enough, we have
\begin{equation*}
\begin{split}
    & K_\pm^{(\tau)}(\theta ,\theta') =  \frac{i}{4\pi \lambda \sqrt{1 - \lambda^2}} \frac{\ell_\lambda^{\pm}(D_x \Xi(\tau,z(\theta)) \cdot z'(\theta))}{q_\pm(\tau,\theta,\theta')G(\theta,\theta')} \det(D_x \Xi(\tau,z(\theta'))) \\ & \quad \times \left(\gamma_\lambda^+(\theta) - \theta' + i\tau \frac{v_\pm(z(\theta),z(\theta')) \ell_\lambda^\mp(z(\theta) - z(\theta')) - i \tau H(\tau,z(\theta),z(\theta'))}{q_\pm(\tau,\theta,\theta')G(\theta,\theta')}\right)^{-1}.
\end{split}
\end{equation*}

It follows from Lemma \ref{lemma:sign_stuff} that $\pm \ell_\lambda^{\pm}(z(\theta) - z(\theta'))$ and $\ell_\lambda^\pm(z'(\gamma_\lambda^\pm(\theta)))$ have the same sign if $(\theta,\theta')$ is close to $(\theta_0,\gamma_\lambda^\pm(\theta_0))$. Hence, we find that for $\tau$ small enough, we have 
\begin{equation*}
   \pm \re \frac{v_\pm(z(\theta),z(\theta')) \ell_\lambda^\mp(z(\theta) - z(\theta')) - i\tau H(\tau,z(\theta),z(\theta'))}{q_\pm(\tau,\theta,\theta')G(\theta,\theta')} >0.
\end{equation*}
The result follows using \cite[Lemma 3.6]{dwz_internal_waves} (see also the remark on p.32 of this reference).
\end{proof}

We are left with the description of $K_\pm^{(\tau)}$ near the intersection of the diagonal with the graph of $\gamma_\lambda^\pm$, i.e. the analogue of \cite[Lemma 4.14]{dwz_internal_waves}.

\begin{lemma}\label{lemma:characteristic_diagonal}
Let $\theta_0 \in \mathbb{S}^1$ be such that $\gamma_\lambda^\pm(\theta_0) = \theta_0$. There are $\tau_0 > 0$, a neighbourhood $U$ of $\theta_0$ in $\mathbb{S}^1$ and smooth functions $a, \psi : [0,\tau_0) \times U \to \mathbb{C}$ and $R : [0,\tau_0) \times U \times U \to \mathbb{C}$ such that for $\theta,\theta' \in U$ and $\tau \in (0,\tau_0)$ we have
\begin{equation*}\begin{split}
    K_{\pm}^{(\tau)}(\theta,\theta') = & \frac{i}{4 \pi \lambda \sqrt{1 - \lambda^2}} \det(D_x \Xi(\tau,z(\theta')))(\theta - \theta' \pm i 0)^{-1} \\
    & \qquad \qquad \qquad + a(\tau,\theta')(\gamma_\lambda^\pm(\theta) - \theta' \pm i \psi(\tau,\theta'))^{-1} + R(\tau,\theta,\theta').
\end{split}\end{equation*}
Moreover, we have $\re \psi(\tau,\theta') > 0$ for $\tau$ small and $a(0,\theta') = \frac{i}{4 \pi \lambda \sqrt{1 - \lambda^2}} \frac{1}{(\gamma_\lambda^\pm)'(\theta')} $.
\end{lemma}

\begin{proof}
As in the proof of Lemma \ref{lemma:non_characteristic_reflection}, we apply Lemma \ref{lemma:linear_forms_deformed} to see that \eqref{eq:linear_forms_boundary} holds for $\tau > 0$ small and $\theta,\theta'$ near $\theta_0$. Using Morse coordinates\footnote{If $\eta$ denotes Morse coordinates that turns $\theta \mapsto \ell_\lambda^\pm(z(\theta))$ into $\eta \mapsto \ell_\lambda^\pm(z(\theta_0)) + \frac{\ell_\lambda^\pm(z''(\theta_0))}{2} \eta^2$, then $\gamma_\lambda^\pm$ is given in these coordinates by $\eta \mapsto - \eta$. The quantity $\ell_\lambda^\pm(z(\theta) - z(\theta'))$ is given in these coordinates by $\frac{\ell_\lambda^\pm(z''(\theta_0))}{2}(\eta - \eta')(\eta + \eta')$.} for $\theta \mapsto \ell_\lambda^\pm(z(\theta))$ near $\theta_0$, we find that there is a smooth function $G_0$, defined on a neighbourhood of $(\theta_0,\theta_0)$ in $\mathbb{S}^1 \times \mathbb{S}^1$ such that
\begin{equation*}
    \ell_\lambda^\pm(z(\theta) - z(\theta')) = G_0(\theta,\theta')(\gamma_\lambda^\pm(\theta) - \theta')(\theta - \theta')
\end{equation*}
for $(\theta,\theta')$ in the domain of $G_0$. In addition, $G_0$ is real-valued and $G_0(\theta_0,\theta_0) = - \ell_\lambda^\pm(z''(\theta_0))/2$. Applying Taylor's formula, we find that there are smooth functions $G_1, G_2$ such that (for $\tau \geq 0$ small and $\theta,\theta'$ near $\theta_0$)
\begin{equation*}
    v_\pm(z(\theta),z(\theta'))\ell_\lambda^\mp(z(\theta) - z(\theta')) = G_1(\theta,\theta')(\theta - \theta') 
\end{equation*}
and $H(\tau,z(\theta),z(\theta')) = G_2(\tau,\theta,\theta')(\theta - \theta')$. Moreover, $G_1$ is real-valued and $G_1(\theta_0,\theta_0) = v_\pm(z(\theta_0),z(\theta_0)) \ell_\lambda^\mp(z'(\theta_0))$ is non-zero. As above, we will use the notation $q_\pm(\tau,\theta,\theta') = 1 + \tau w_\pm(z(\theta),z(\theta'))$.

Using the auxiliary functions that we just introduced, \eqref{eq:linear_forms_boundary} become
\begin{equation}\label{eq:double_expansion}
\begin{split}
    & \ell_\lambda^{\pm}(\Xi(\tau,z(\theta)) - \Xi(\tau,z(\theta')) \\ & \qquad \qquad = q_\pm(\tau,\theta,\theta') G_0(\theta,\theta') (\gamma_\lambda^\pm(\theta) - \theta') (\theta-\theta') \\ & \qquad \qquad \qquad \qquad \qquad + i\tau G_1(\theta,\theta')(\theta - \theta') + \tau^2 G_2(\tau,\theta,\theta')(\theta-\theta').
\end{split}
\end{equation}

We claim then that, as in Lemma \ref{lemma:non_characteristic_diagonal}, we have
\begin{equation}\label{eq:first_step_difficult_case}
\begin{split}
    & K_{\pm}^{(\tau)}(\theta,\theta') \\ & \quad = \frac{i}{4 \pi \lambda \sqrt{1 - \lambda^2}} \det(D_x\Xi(\tau,z(\theta')))(\theta - \theta' \pm i 0)^{-1} \\
    & \qquad \qquad \times \frac{\ell_\lambda^\pm(D_x \Xi(\tau,z(\theta)) \cdot z'(\theta))}{q_\pm(\tau,\theta,\theta') G_0(\theta,\theta')(\gamma_\lambda^\pm(\theta) - \theta') + i \tau G_1(\theta,\theta') + \tau^2 G_2(\tau,\theta,\theta')}
\end{split}
\end{equation}
Here, the denominator is non-zero for $\tau$ small enough (and $\theta,\theta'$ close to $\theta_0$) because $G_1(\theta_0,\theta_0) \neq 0$. To get this formula, we need to justify, as in the proof of  Lemma \ref{lemma:non_characteristic_diagonal}, that
\begin{equation*}
    \pm \im \frac{k(0,\tau,\theta)}{q_\pm(\tau,\theta,\theta') G_0(\theta,\theta')(\gamma_\lambda^\pm(\theta) - \theta') + i \tau G_1(\theta,\theta') + \tau^2 G_2(\tau,\theta,\theta')}
\end{equation*}
is strictly positive (for $\tau$ small and with $k$ as in \eqref{eq:normal_expansion}). Equivalently, we can prove that
\begin{equation*}
    \mp \im \frac{q_\pm(\tau,\theta,\theta')G_0(\theta,\theta')(\gamma_\lambda^\pm(\theta) - \theta') + i \tau G_1(\theta,\theta') + \tau^2 G_2(\tau,\theta,\theta')}{k(0,\tau,\theta)}
\end{equation*}
is strictly positive. This imaginary part is zero at $\tau = 0$. Hence, we only need to prove that its derivative in $\tau$ has the right sign at $\tau = 0$. By continuity, we only need to do it at $\theta = \theta' = \theta_0$, in which case the derivative in $\tau$ is
\begin{equation}\label{eq:t_derivative}
    \mp \frac{G_1(\theta_0,\theta_0)}{k(0,0,\theta_0)} = \mp v_\pm(z(\theta_0),z(\theta_0)) \frac{\ell_\lambda^\mp(z'(\theta_0))}{\ell_\lambda^\pm(\mathbf{n}(\theta_0))}.
\end{equation}
Notice that the endomorphism of $\mathbb{R}^2$ given by $x \mapsto (\pm \ell_\lambda^\pm(x),\ell_\lambda^\mp(x))$ is orientation preserving (its determinant is $2/(\lambda \sqrt{1 - \lambda^2})$). If we use $(z'(\theta_0),\mathbf{n}(\theta_0))$ as a basis for the starting space and the canonical basis of $\mathbb{R}^2$ for the image, the matrix of this linear map is
\begin{equation*}
    \begin{bmatrix}
    0 & \pm \ell_\lambda^\pm(\mathbf{n}(\theta_0)) \\
    \ell_\lambda^\mp(z'(\theta_0)) & \ell_\lambda^\mp(\mathbf{n}(\theta_0)
    \end{bmatrix}.
\end{equation*}
Since $(z'(\theta_0), \mathbf{n}(\theta_0))$ is positively oriented, it follows that $\mp \ell_\lambda^\mp(z'(\theta_0)) \ell_\lambda^\pm(\mathbf{n}(\theta_0))$ is positive, and thus that \eqref{eq:t_derivative} is positive. Hence, we have \eqref{eq:first_step_difficult_case}.

Differentiating \eqref{eq:double_expansion} in $\theta$, we find that
\begin{equation*}
\begin{split}
    & \ell_\lambda^\pm(D_x \Xi(\tau,z(\theta)) \cdot z'(\theta)) \\ & \qquad  = q_\pm(\tau,\theta,\theta') G_0(\theta,\theta')(\gamma_\lambda^\pm(\theta) - \theta') + i \tau G_1(\theta,\theta') + \tau^2 G_2(\tau,\theta,\theta') \\& \qquad \qquad \qquad + (\theta - \theta') \partial_\theta \Big(q_\pm(\tau,\theta,\theta') G_0(\theta,\theta') (\gamma_\lambda^\pm(\theta) - \theta') \\ & \qquad \qquad \qquad \qquad \qquad \qquad \qquad \qquad \qquad \qquad + i \tau G_1(\theta,\theta') + \tau^2 G_2(\tau,\theta,\theta')\Big),
\end{split}
\end{equation*}
so that 
\begin{equation*}
\begin{split}
    K_{\pm}^{(\tau)}(\theta,\theta') & = \frac{i}{4 \pi \lambda \sqrt{1 - \lambda^2}} \det(D_x\Xi(\tau,z(\theta')))(\theta - \theta' \pm i 0)^{-1} \\ & \qquad \qquad \qquad + \frac{i}{4 \pi \lambda \sqrt{1 - \lambda^2}} \det(D_x\Xi(\tau,z(\theta'))) V(\tau,\theta,\theta'),
\end{split}
\end{equation*}
where
\begin{equation*}
\begin{split}
     V(\tau,\theta,\theta') & = \frac{\partial_\theta(q_\pm(\tau,\theta,\theta')G_0(\theta,\theta')(\gamma_\lambda^\pm(\theta) - \theta') + i \tau G_1(\theta,\theta') + \tau^2 G_2(\tau,\theta,\theta'))}{q_\pm(\tau,\theta,\theta') G_0(\theta,\theta')(\gamma_\lambda^\pm(\theta) - \theta') + i \tau G_1(\theta,\theta') + \tau^2 G_2(\tau,\theta,\theta')} \\
    & = \frac{\partial_\theta(q_\pm(\tau,\theta,\theta') G_0(\theta,\theta')(\gamma_\lambda^\pm(\theta) - \theta') + i \tau G_1(\theta,\theta') + \tau^2 G_2(\tau,\theta,\theta'))}{q_\pm(\tau,\theta,\theta') G_0(\theta,\theta')} \\ & \qquad \qquad \quad \qquad \times \left(\gamma_\lambda^\pm(\theta) - \theta' + i \tau \frac{G_1(\theta,\theta') -i  \tau G_2(\tau,\theta,\theta')}{q_\pm(\tau,\theta,\theta')G_0(\theta,\theta')}\right)^{-1}.
\end{split}
\end{equation*}
In the last expression for $V(\tau,\theta,\theta')$, the first factor is smooth in $\tau,\theta,\theta'$ (because $q_\pm(0,\theta,\theta') = 1$ and $G_0(\theta_0,\theta_0)$ is non-zero). For $\tau$ small and $\theta,\theta'$ close to $\theta_0$ the real part of $\pm \frac{G_1(\theta,\theta') - i \tau G_2(\tau,\theta,\theta')}{q_\pm(\tau,\theta,\theta') G_0(\theta,\theta')}$ has the sign of 
\begin{equation}\label{eq:last_sign}
 \pm \frac{G_1(\theta_0,\theta_0)}{G_0(\theta_0,\theta_0)} = \mp 2 v_\pm (z(\theta_0),z(\theta_0)) \frac{\ell_\lambda^\mp(z'(\theta_0))}{\ell_\lambda^\pm(z''(\theta_0))}.    
\end{equation}
Since $\mu^\pm(z(\theta_0)) = 0$, it follows from \cite[Lemma 2.2]{dwz_internal_waves} that for every $x \in \Omega$ we have $\mp \mu^\mp(z(\theta_0)) \ell_\lambda^\pm(x - z(\theta_0)) > 0$. Since $z(\theta_0)$ is a global extremum for $\ell_\lambda^\pm$ on $\overline{\Omega}$, for any $x \in \Omega$ the signs of $\ell_\lambda^\pm(x - z(\theta_0))$ and $\ell_\lambda^\pm(z''(\theta_0))$ are the same. It follows that $\mp \ell_\lambda^\mp(z'(\theta_0)) \ell_\lambda^\pm(z''(\theta_0)) > 0$, and thus \eqref{eq:last_sign} is positive. Hence, the result follows from \cite[Lemma 3.6]{dwz_internal_waves} (see also the remark on p.32 of this reference).
\end{proof}

\subsection{Proof of Proposition \ref{proposition:inviscid_invertibility}}\label{subsection:proof_invertibility}

Lemmas \ref{lemma:non_characteristic_non_diagonal}, \ref{lemma:non_characteristic_diagonal}, \ref{lemma:non_characteristic_reflection} and \ref{lemma:characteristic_diagonal} are the analogue of Lemmas 4.11, 4.12, 4.13 and 4.14 in \cite{dwz_internal_waves}, with the parameter ``$\epsilon = \im \omega$'' from this reference replaced by our deformation parameter $\tau$. We could consequently prove all the results from \cite{dwz_internal_waves} that are consequences of these lemmas. In particular, we get a microlocal description of $\mathcal{C}_\lambda^{(\tau)}$ as $\tau$ goes to $0$, see \cite[Proposition 4.15]{dwz_internal_waves}. With Lemma \ref{lemma:convergence_kernel}, it implies the following convergence result (this is an analogue of \cite[Lemma 4.16]{dwz_internal_waves} with the parameter $k$ from this reference set to $0$).

\begin{lemma}\label{lemma:boundedness_restricted_layer}
For every $r,s \in \mathbb{R}$ with $r + 1 > s$ the operator $\mathcal{C}_\lambda^{(\tau)}$ converges as $\tau$ goes to $0$ to $\mathcal{C}_{\lambda + i 0}$ for the operator norm\footnote{The norms for Sobolev spaces of currents on $\partial \Omega$ are defined for instance by identifying the currents on $\partial \Omega$ with distributions using any smooth non-vanishing reference $1$-form.\label{note:norm_current}} $H^r(\partial \Omega, T^* \partial \Omega \otimes \mathbb{C}) \to H^s(\partial \Omega)$.
\end{lemma}

The analysis from \cite[\S 5]{dwz_internal_waves} also applies in our case. In particular, we can state ``microlocal Lasota--Yorke inequalities'' that are proven as \cite[Proposition 5.3 and 5.4]{dwz_internal_waves}.

\begin{proposition}\label{proposition:microlocal_lasota_yorke}
There is $\tau_0 > 0$ such that the following holds. Let $\beta > 0$ and $N \in \mathbb{R}$. There is $N_0 \in \mathbb{R}$ and a constant $C > 0$ such that for every $v \in \Gamma(T^* \partial \Omega \otimes \mathbb{C})$ and $\tau \in (0,\tau_0)$ we have
\begin{equation*}
    \n{v}_{H^{-\frac{1}{2}- \beta}} \leq C \n{\mathcal{C}_{\lambda}^{(\tau)} v}_{H^{N_0}} + C \n{v}_{H^{-N}}.
\end{equation*}
Moreover, if $p$ is one of the seminorms from \cite[(3-20)]{dwz_internal_waves}, there is $N_0 \in \mathbb{R}$ and a constant $C > 0$ such that for every $v \in \Gamma(T^* \partial \Omega \otimes \mathbb{C})$ and $\tau \in (0,\tau_0)$ we have
\begin{equation*}
    p(v) \leq C \n{\mathcal{C}_{\lambda}^{(\tau)} v}_{H^{N_0}} + C \n{v}_{H^{-N}}.
\end{equation*}
\end{proposition}

\begin{remark}
Notice that in Lemmas \ref{lemma:non_characteristic_non_diagonal}, \ref{lemma:non_characteristic_diagonal}, \ref{lemma:non_characteristic_reflection} and \ref{lemma:characteristic_diagonal}, there is a factor $D_x \Xi(\tau,\theta')$ that does not appear in the analogue statements in \cite{dwz_internal_waves}. This factor does not impact the proof of the microlocal Lasota--Yorke inequalities. Indeed, if we let $m_\tau$ be the operator of multiplication by $D_x \Xi(\tau,\cdot)$, then we can factorize $\mathcal{C}_\lambda^{(\tau)}$ as $\mathcal{C}_\lambda^{(\tau)} = \widetilde{\mathcal{C}}_\lambda^{(\tau)} m_\tau$, where $\widetilde{\mathcal{C}}_\lambda^{(\tau)}$ satisfies Lemmas \ref{lemma:non_characteristic_non_diagonal}, \ref{lemma:non_characteristic_diagonal}, \ref{lemma:non_characteristic_reflection} and \ref{lemma:characteristic_diagonal} without the factor $D_x \Xi(\tau,\theta')$. The analysis from \cite[\S 5]{dwz_internal_waves} allows to prove the microlocal Lasota--Yorke inequalities for $\widetilde{\mathcal{C}}_\lambda^{(\tau)}$. Since $D_x \Xi(\tau,\cdot)$ converges to $1$ as a smooth function when $\tau$ goes to $0$, the operator $m_\tau$ is an isomorphism from $H^s$ to itself, uniformly as $\tau$ goes to $0$ and for any $s \in \mathbb{R}$. Hence $\mathcal{C}_\lambda^{(\tau)}$ also satisfies the microlocal Lasota--Yorke inequalities.
\end{remark}

We are now ready to prove Proposition \ref{proposition:inviscid_invertibility}.

\begin{proof}[Proof of Proposition \ref{proposition:inviscid_invertibility}]
Let us prove first that there is $\tau_0 > 0$ such that for every $\tau \in (0,\tau_0)$ the operator $P_{\lambda,0}^{(\tau)} : H^2(\Omega) \cap H_0^1(\Omega) \to L^2(\Omega)$ is invertible. We prove it by contradiction and assume that there is a sequence $(\tau_n)_{n \geq 1}$ going to $0$ such that for every $n \geq 1$ the operator $P_{\lambda,0}^{(\tau_n)} : H^2(\Omega) \cap H_0^1(\Omega) \to L^2(\Omega)$ is not invertible. According to Lemma \ref{lemma:inviscid_index}, for each $n \geq 1$ there is a non-zero element $u_n$ of $H^2(\Omega) \cap H_0^1(\Omega)$ such that $P_{\lambda,0}^{(\tau_n)} u_n = 0$. Since $P_{\lambda,0}^{(\tau_n)}$ is properly elliptic (see the proof of Lemma \ref{lemma:inviscid_index}), the function $u_n$ is smooth on $\overline{\Omega}$, see for instance \cite[Chapitre 8, Théorème 1.2]{lions_magenes_3}. Hence, Lemma~\ref{lemma:boundary_representation} implies that for every $n \geq 1$ there is $v_n \in \Gamma(T^* \partial \Omega)$ such that $u_n = S_\lambda^{(\tau_n)} v_n$ and $E_\lambda^{(\tau_n)}(\mathcal{I}(v_n))_{|M \setminus \overline{\Omega}} = 0$. 

From $u_n = S_\lambda^{(\tau_n)} v_n $, we deduce that $v_n$ is non-zero. We can consequently normalize the $u_n$'s so that $\n{v_n}_{H^{-1}} = 1$. Notice also that $\mathcal{C}_\lambda^{(\tau_n)} v_n = u_{n|\partial \Omega} = 0$. Hence, applying Proposition \ref{proposition:microlocal_lasota_yorke}, we find that the sequence $(v_n)_{n \geq 1}$ is bounded in $H^{-\frac{1}{2} - \beta}$ for every $\beta > 0$. Up to extracting, we may assume that $(v_n)_{n \geq 1}$ converges in  $H^{-\frac{1}{2} - \beta}$ for every $\beta > 0$ to some limit $v$ such that $\n{v}_{H^{-1}} = 1$. Moreover, it follows from the second part of Proposition \ref{proposition:microlocal_lasota_yorke} that $v$ belongs to the space of conormal distributions $I^{\frac{1}{4}+}(\partial \Omega, N_+^* \Sigma_- \cup N_-^* \Sigma_+)$ from \cite{dwz_internal_waves}.

It follows from Lemma \ref{lemma:support_condition} that $(E_{\lambda + i 0} \ast \mathcal{I}(v))_{|M \setminus \overline{\Omega}} = 0$ and from Lemma \ref{lemma:boundedness_restricted_layer} that $\mathcal{C}_{\lambda + i 0} v = 0$. These two conditions with $v \in I^{\frac{1}{4}+}(\partial \Omega, N_+^* \Sigma_- \cup N_-^* \Sigma_+)$ allow us to apply \cite[Lemma 7.2]{dwz_internal_waves} in order to get $v = 0$, a contradiction. Notice that the support condition that we proved ($(E_{\lambda + i 0} \ast \mathcal{I}(v))_{|M \setminus \overline{\Omega}} = 0$) is slightly weaker than the one from \cite[Lemma 7.2]{dwz_internal_waves} (which asks for $(E_{\lambda + i 0} \ast \mathcal{I}(v))_{|\mathbb{R}^2 \setminus \overline{\Omega}} = 0$). However, a quick inspection of the proof reveals that the condition that we proved is sufficient. 

We proved that there is $\tau_0 > 0$ such that for every $\tau \in (0,\tau_0)$ the operator $P_{\lambda,0}^{(\tau)} : H^2(\Omega) \cap H_0^1(\Omega) \to L^2(\Omega)$ is invertible. Up to making $\tau_0$ smaller, it follows from Lemma~\ref{lemma:inviscid_index} that there is $\delta > 0$ such that for every $\omega \in (\lambda - \delta,\lambda + \delta)$ and $\tau \in (0,\tau_0)$ the operator $P_{\omega,0}^{(\tau)}$ properly elliptic and Fredholm of index zero as an operator $H^2(\Omega) \cap H_0^1(\Omega) \to L^2(\Omega)$. Pick $\tau_1 \in (0,\tau_0)$. Since $P_{\lambda,0}^{(\tau_1)}$ is invertible we may assume, up to taking $\delta > 0$ smaller, that $P_{\omega,0}^{(\tau_1)}$ is invertible for every $\omega \in (\lambda - \delta,\lambda + \delta)$. Then, working as in the proof of Lemma \ref{lemma:positive_imaginary_part_invertibility}, we can prove that the dimension of the kernel of $P_{\omega,0}^{(\tau)}$ for $\omega \in (\lambda- \delta,\lambda + \delta)$ is a constant function of $\tau \in (0,\tau_0)$. Since $P_{\omega,0}^{(\tau_1)}$ is invertible, we find that $P_{\omega,0}^{(\tau)}$ also is, which concludes the proof of Proposition \ref{proposition:inviscid_invertibility}, recalling Lemma~\ref{lemma:positive_imaginary_part_invertibility}.
\end{proof}

We end this section by gathering together all the pieces of the proof of Theorem \ref{theorem:existence_deformation}.

\begin{proof}[Proof of Theorem \ref{theorem:existence_deformation}]
Recall that $\overline{\Omega}_\tau$ is defined for $\tau$ small at the beginning of \S \ref{section:differential_operators}, based on the construction in \S\S \ref{subsection:escape_function}--\ref{subsection:new_domain}. We will set $\widetilde{\Omega} = \overline{\Omega}_\tau$ for some small $\tau > 0$. From Remark \ref{remark:control_boundary}, we know that $\partial \widetilde{\Omega} \subseteq (\partial \Omega)_{\mathbb{C}}$.

According to Proposition \ref{proposition:inviscid_invertibility}, there is $\delta > 0$ such that, provided $\tau > 0$ is small enough, the operator $P_{\omega,0} : H^2(\widetilde{\Omega}) \cap H_0^1(\widetilde{\Omega}) \to L^2(\widetilde{\Omega})$ is invertible for every $\omega \in (\lambda - \delta,\lambda + \delta) + i [0,+\infty)$. Once $\tau$ is fixed with this property, we may take $\delta > 0$ smaller to have $P_{\omega,0} : H^2(\widetilde{\Omega}) \cap H_0^1(\widetilde{\Omega}) \to L^2(\widetilde{\Omega})$ invertible for every $\omega \in (\lambda - \delta,\lambda + \delta) + i (-\delta,+\infty)$, just because the set of invertible operators is open.
\end{proof}

\section{Vanishing viscosity limit}\label{section:viscosity_limit}

This section is dedicated to the proof of Theorem \ref{theorem:main}. We start by proving a priori estimates uniform in $\nu > 0$ small for the operator $P_{\omega,\nu}^{(\tau)}$ using the work of Frank \cite{frank_coercive_singular} in \S \ref{subsection:a_priori_estimates}. With all that precede, we can then prove Theorem \ref{theorem:main} in \S \ref{subsection:proof_main}.

\subsection{A priori estimates on the deformed domain}\label{subsection:a_priori_estimates}

We want to apply results from \cite{frank_coercive_singular} to the operator $P_{\omega,\nu}$ as $\nu$ goes to $0$. To do so, we introduce for $r,s \in \mathbb{R}$ and $\nu > 0$ a new norm\footnote{The Japanese bracket is defined as usual for $\xi \in \mathbb{R}^2$ by $\brac{\xi} = \sqrt{1 + |\xi|^2}$.} on $H^{r +s}(\mathbb{R}^2)$:
\begin{equation}\label{eq:norm}
    \n{u}_{r,s,\nu}^2 = \int_{\mathbb{R}^2} \brac{\xi}^{2r} \brac{\sqrt{\nu} \xi}^{2s} |\hat{u}(\xi)|^2 \,\mathrm{d}\xi, \ u \in H^{r +s}(\mathbb{R}^2).
\end{equation}
We will mostly work with a similar norm on $\Omega$, namely
\begin{equation*}
    \n{u}_{r,s,\nu} = \inf_{\substack{v \in H^{r+s}(\mathbb{R}^2) \\ v_{|\Omega} = u}} \n{v}_{r,s,\nu}, \ u \in H^{r+ s}(\Omega).
\end{equation*}
Beware that the parameters $r,s$ and $\nu$ play very different roles in the definition of the norm $\n{\cdot}_{r,s,\nu}$. The parameters $r$ and $s$ measure regularity (as the usual norms on Sobolev spaces), while $\nu$ is the viscosity parameter that appears in the operator $P_{\omega,\nu}$. We use a $\nu$-dependent norm in order to incorporate the loss of ellipticity of $P_{\omega,\nu}$ as $\nu$ goes to $0$ in the following estimate:

\begin{proposition}\label{proposition:a_priori_estimate}
There is $\tau_0 > 0$ such that for every $\tau \in (0,\tau_0)$ and every $r,s, N \in \mathbb{R}$ such that $\frac{1}{2} < r < \frac{3}{2}$ and $r + s > \frac{3}{2}$ there are $C,\nu_0 > 0$ such that for every $\nu \in (0,\nu_0), \omega \in \mathbb{D}(\lambda,C^{-1})$ and $u \in H^{r+s}(\mathbb{R}^2)$ such that $u_{|\partial \Omega} = 0$ and $\mathrm{d}u_{|\partial \Omega} = 0$ we have
\begin{equation}\label{eq:a_priori_estimate}
    \n{u}_{r,s,\nu} \leq C \n{P_{\omega,\nu}^{(\tau)}u}_{r-2,s-2,\nu} + C \n{u}_{H^{-N}}.
\end{equation}
\end{proposition}

Proposition \ref{proposition:a_priori_estimate} will be obtained first in the case $\omega = \lambda$ by applying \cite[Theorem 3.5.2]{frank_coercive_singular}. In order to apply this theorem, we need to check that the boundary value problem (of unknown $u$)
\begin{equation}\label{eq:singular_problem}
    \begin{cases}
    P_{\lambda,\nu}^{(\tau)} u = f \\
    u_{|\partial \Omega} = \varphi_1 \\
    \partial_\mathbf{n}u_{|\partial \Omega} = \varphi_2
    \end{cases}
\end{equation}
is elliptic and coercive as defined in \cite[\S 3.5]{frank_coercive_singular}. Here, $\partial_\mathbf{n}$ denote differentiation in the direction of the inward pointing normal of $\partial \Omega$. It will be convenient to set for $\tau> 0$ small, $x \in \overline{\Omega}, \nu \geq 0$ and $\xi \in \mathbb{R}^2$
\begin{equation*}
    A_\tau(\nu,x,\xi) = p_{\lambda,0}^{(\tau)}(x,\xi) + i \lambda \nu q^{(\tau)}(x,\xi),
\end{equation*}
where $q^{(\tau)}$ denote the symbol of the operator $\Delta^2$ acting on $\overline{\Omega}_\tau$ in the coordinates given by $\Xi(\tau,\cdot)$, see \S \ref{section:differential_operators}.

\begin{remark}\label{remark:different_notations}
Our notations differ a little from those of Frank in \cite{frank_coercive_singular}, we highlight here the differences to help the reader understand how we apply the results from this reference.

The parameter going to $0$ in \cite{frank_coercive_singular} is called $\epsilon$. Beware that the analogue of $\epsilon$ in our setting is not $\nu$ but $\sqrt{\nu}$, as can be seen in \eqref{eq:norm}. Frank does not use the norm \eqref{eq:norm} exactly, instead he works with the norm
\begin{equation*}
    \n{u}_{(s_1,s_2,s_3)}^2 = \epsilon^{-2 s_1} \int_{\mathbb{R}^n} \brac{\xi}^{2s_2} \brac{\epsilon \xi}^{2s_3} |\hat{u}(\xi)|^2 \,\mathrm{d}\xi,
\end{equation*}
see \cite[(1.1.1)]{frank_coercive_singular}. Notice that the small parameter $\epsilon$ does not appear in the notation for the norm, but there are still three parameters $s_1,s_2,s_3$. The numbers $s_2$ and $s_3$ are the analogues of $r$ and $s$ respectively in \eqref{eq:norm}. We decided to remove the parameter $s_1$, which only contributes through the extra factor $\epsilon^{-s_2}$ in the norm. This extra parameter can be used to remove the assumption $\mathrm{d}u_{|\partial \Omega} = 0$ in \eqref{eq:a_priori_estimate} and instead include a norm of the Neumann boundary data in the right hand side, see \cite[Theorem 3.5.2]{frank_coercive_singular}.
\end{remark}

We start our study of \eqref{eq:singular_problem} with an ellipticity estimate.

\begin{lemma}\label{lemma:singular_ellpiticity}
There is $\tau_0 > 0$ such that for every $\tau \in (0,\tau_0)$, the problem \eqref{eq:singular_problem} is elliptic in the sense of \cite{frank_coercive_singular}. That is, there is a constant $C > 0$ such that for every $x \in \overline{\Omega}, \xi \in \mathbb{R}^2$ and $\nu \geq 0$ we have
\begin{equation*}
    \left| A_\tau(\nu,x,\xi) \right| \geq C^{-1}|\xi|^2\brac{\sqrt{\nu}\xi}^2.
\end{equation*}
\end{lemma}

\begin{proof}
According to Lemma \ref{lemma:inviscid_ellipticity} there is a constant $C > 0$ such that for $x \in \overline{\Omega}, \xi \in \mathbb{R}^2$ and $\tau > 0$ small enough we have
\begin{equation*}
    \im p_{\lambda,0}^{(\tau)}(x,\xi) + C \tau|\re p_{\lambda,0}^{(\tau)}(x,\xi)| \geq C^{-1}\tau|\xi|^2.
\end{equation*}
Since the symbol of $\Delta^2$ on $\mathbb{R}^2$ is positive and elliptic, we find that, up to making $C$ larger, we have for $x \in \overline{\Omega}, \xi \in \mathbb{R}^2$ and $\tau> 0$ small enough
\begin{equation*}
    \im (i\lambda q^{(\tau)}(x,\xi)) \geq C^{-1}|\xi|^4 \textup{ and } |\re(i \lambda q^{(\tau)}(x,\xi))| \leq C \tau |\xi|^4.
\end{equation*}
Summing all these estimates, we get
\begin{equation*}
    \im A_\tau(\nu,x,\xi) + C \tau |\re A_\tau(\nu,x,\xi)| \geq C^{-1}\tau |\xi|^2 + (C^{-1} - C^2 \tau^2)\nu|\xi|^4.
\end{equation*}
By taking $\tau$ small enough, and maybe changing the value of $C$, we find that
\begin{equation*}
    \im A_\tau(\nu,x,\xi,) + C \tau |\re A_\tau(\nu,x,\xi)| \geq C^{-1} |\xi|^2 (\tau + \nu|\xi|^2),
\end{equation*}
and the result follows.
\end{proof}

We prove now the coercivity condition from \cite[p.105]{frank_coercive_singular}. Notice that since we are working with Dirichlet and Neumann boundary condition, the statement of the coercivity condition becomes significantly simpler.

\begin{lemma}\label{lemma:singular_coercivity}
There is $\tau_0 > 0$ such that for every $\tau \in (0,\tau_0)$ the problem \eqref{eq:singular_problem} is coercive in the sense of \cite{frank_coercive_singular}. That is, if $x \in \partial \Omega, \mathbf{n} \in \mathbb{R}^2$ is the interior conormal to $\Omega$ at $x$ and $\xi \in \mathbb{R}^2$ is not proportional to $\mathbf{n}$ then the following holds:
\begin{enumerate}[label=(\roman*)]
    \item the polynomial $z \mapsto q^{(\tau)}(x,\xi + z \mathbf{n})$ has two roots on each side of the real axis; \label{item:coercivity_high_frequency}
    \item the polynomial $z \mapsto A_\tau(0,x ,\xi + z \mathbf{n})$ has one root on each side of the real axis;\label{item:coercivity_low_frequency}
    \item for every $\nu \in \mathbb{R}_+^*$, the polynomial $z \mapsto A_\tau(\nu,x,\xi + z \mathbf{n})$ has two roots on each side of the real axis; \label{item:coercivity_with_parameter}
    \item the polynomial $z \mapsto z^{-2}A_\tau(1,x,z \mathbf{n})$ has one root on each side of the real axis.\label{item:coercivity_boundary_layer}
\end{enumerate}
\end{lemma}

\begin{proof}
We start by proving \ref{item:coercivity_high_frequency}. The polynomial $q^{(\tau)}(x,\cdot)$ is elliptic for $\tau$ small enough (even at $\tau = 0$). Hence, $z \mapsto q^{(\tau)}(x,\xi + z \mathbf{n})$ has no root on the real axis. Since its roots depend continuously on the parameters $\tau$, the number of roots on each side of the real axis does not depend on $\tau$. Hence, we only need to prove that $z \mapsto q^{(0)}(x, \xi + z \mathbf{n})$ has two roots on each side of the real axis. Since this polynomial has real coefficients, it has the same number of roots on each side of the real axis.

The condition \ref{item:coercivity_low_frequency} follows from the proper ellipticity of $P_{\lambda,0}^{(\tau)}$, see Lemma \ref{lemma:inviscid_index}.

To prove \ref{item:coercivity_with_parameter}, notice that it follows from Lemma \ref{lemma:singular_ellpiticity} that for every $\nu \in (0,+ \infty)$ the polynomial $z \mapsto A_\tau(\nu,x,\xi + z \mathbf{n})$ has no root on the real axis. Hence, the number of root on each side of the real axis of this polynomial does not depend on $\nu$. The coefficients of the polynomial $z \mapsto \nu^{-1} A_\tau(\nu, x,\xi + z \mathbf{n})$ converge to the coefficients of $z \mapsto q^{(\tau)}(x,\xi + z \mathbf{n})$ has $\nu$ goes to $+ \infty$. Hence, it follows from \ref{item:coercivity_high_frequency} that for $\nu$ large enough the polynomial $z \mapsto A_\tau(\nu,x,\xi + z \mathbf{n})$ has two roots on each side of the real axis.

It remains to prove \ref{item:coercivity_boundary_layer}. Notice that the coefficients of the polynomial $z \mapsto A_\tau(\nu,x,\xi + z \mathbf{n})$ converge to those of $z \mapsto A_\tau(0,x, \xi + z \mathbf{n})$ as $\nu$ goes to $0$. Hence, it follows from Rouché's Theorem that there is a neighbourhood $U$ of the roots of $z \mapsto A_\tau(0,x, \xi + z \mathbf{n})$, that we may choose bounded, such that for $\nu$ small enough there are exactly two roots for $z \mapsto A_\tau(\nu,x, \xi + z \mathbf{n})$ within $U$. Moreover, the roots of $z \mapsto A_\tau(\nu,x, \xi + z \mathbf{n})$ within $U$ converge to the roots of $z \mapsto A_\tau(0,x, \xi + z \mathbf{n})$. Hence, \ref{item:coercivity_low_frequency} implies that $z \mapsto A_\tau(\nu,x, \xi + z \mathbf{n})$ has exactly one root on each side of the real axis within $U$. We use then the homogeneity property $A_\tau(\nu,x,\xi + z \mathbf{n}) = \nu^{-1} A_\tau(1,x,\sqrt{\nu} \xi + \sqrt{\nu} z \mathbf{n})$ to find that there is a neighbourhood $V$ of the non-zero roots of $z \mapsto A_\tau(1,x,z \mathbf{n})$, that we may choose at positive distance from zero, such that for $\nu$ small enough the polynomial $z \mapsto A_\tau(\nu,x, \xi + z \mathbf{n})$ has exactly two roots within $\sqrt{\nu}^{-1} V$ (it follows from Lemma \ref{lemma:singular_ellpiticity} that $z \mapsto A_\tau(z, x,z \mathbf{n})$ has exactly two non-zero roots, and that they are not on the real axis). For $\nu$ small enough, $U \cap \sqrt{\nu}^{-1} V = \emptyset$ and it follows from \ref{item:coercivity_with_parameter} that $z \mapsto A_\tau(\nu,x,\xi + z \mathbf{n})$ has exactly two roots on each side of the real axis within $\sqrt{\nu}^{1} V$. After multiplication by $\sqrt{\nu}$, these roots converge to the non-zero roots of $z \mapsto A_\tau(1,x,z\mathbf{n})$, which prove that this polynomial has exactly one non-zero root on each side of the real axis, proving \ref{item:coercivity_boundary_layer}.
\end{proof}

We are now ready to prove Proposition \ref{proposition:a_priori_estimate}. 

\begin{proof}[Proof of Proposition \ref{proposition:a_priori_estimate}]
Lemmas \ref{lemma:singular_ellpiticity} and \ref{lemma:singular_coercivity} allow us to apply \cite[Theorem 3.5.2]{frank_coercive_singular}, which proves that \eqref{eq:a_priori_estimate} holds for $\omega = \lambda$, see Remark \ref{remark:different_notations}. By a standard perturbation argument, to prove that the estimate \eqref{eq:a_priori_estimate} still holds when $\omega$ is close to $\lambda$, one only needs to notice that there is a constant $C > 0$ such that for $\omega$ close to $\lambda$ and $u \in H^{r+s}(\Omega)$ we have
\begin{equation}\label{eq:perturbation_estimate}
    \n{\left(P_{\omega,\nu}^{(\tau)} - P_{\lambda,\nu}^{(\tau)} \right)u}_{r-2,s-2,\nu} \leq C|\omega - \lambda|\n{u}_{r,s,\nu}.
\end{equation}
To prove such an estimate, just notice that
\begin{equation*}
    P_{\omega,\nu}^{(\tau)} - P_{\lambda,\nu}^{(\tau)} = (\lambda - \omega)(\lambda + \omega) \Delta^{(\tau)} + i (\omega - \lambda)\nu(\Delta^{(\tau)})^2.
\end{equation*}
The estimate \eqref{eq:perturbation_estimate} follows, as one easily checks that $\Delta^{(\tau)}$ and $\nu (\Delta^{(\tau)})^2$ are bounded, uniformly in $\nu > 0$, as operators from $H^{r+s}(\Omega)$ to $H^{r+s-4}(\Omega)$, endowed respectively with the norms $\n{\cdot}_{r,s,\nu}$ and $\n{\cdot}_{r-2,s-2,\nu}$ (see \cite[Proposition 2.2.1]{frank_coercive_singular}).
\end{proof}

\subsection{Proof of Theorem \ref{theorem:main}}\label{subsection:proof_main}

Let us prepare the proof of Theorem \ref{theorem:main} with a technical statement.

\begin{lemma}\label{lemma:invariance_invertibility}
There is $\tau_0 > 0$ such that for every $\omega \in (0,1) + i \mathbb{R}, \nu > 0$ and $\tau \in (-\tau_0,\tau_0)$ the operator $P_{\omega,\nu}^{(\tau)} : H^4(\Omega) \cap H_0^2(\Omega) \to L^2(\Omega)$ is Fredholm of index zero. Moreover, if $\tau' \in (-\tau_0,\tau_0)$ then
\begin{equation*}
    \dim \ker P_{\lambda,\nu}^{(\tau)} = \dim \ker P_{\lambda,\nu}^{(\tau')}.
\end{equation*}
\end{lemma}

\begin{proof}
Notice that $(\Delta^{(\tau)})^2$ is a perturbation of $\Delta^2$. Hence by taking $\tau_0$ small enough we ensure that for every $\tau \in (-\tau_0,\tau_0)$ the operator $(\Delta^{(\tau)})^2$ is strongly elliptic \cite[Chapitre 2, Définition 1.2]{lions_magenes_1}. As $i \nu \omega (\Delta^{(\tau)})^2$ is the highest order term in $P_{\omega,\nu}^{(\tau)}$, we find that $P_{\omega,0}^{(\tau)}$ is strongly elliptic and thus properly elliptic, for every $\omega \in (0,1) + i \mathbb{R}, \nu > 0$ and $\tau \in (-\tau_0,\tau_0)$, and it follows that $P_{\omega,\nu}^{(\tau)} : H^4(\Omega) \cap H_0^2(\Omega) \to L^2(\Omega)$ is Fredholm of index zero, see \cite[Chapitre 2, (8.14)]{lions_magenes_1}.

We can then prove that if $\omega \in (0,1) + i \mathbb{R}$ and $\nu > 0$, the function $\tau \mapsto \dim \ker P_{\lambda,\nu}^{(\tau)}$ is constant on $(- \tau_0,\tau_0)$, as we did in the proof of Lemma \ref{lemma:positive_imaginary_part_invertibility}. As in this proof here, we use the fact that since $P_{\omega,\nu}^{(\tau)}$ is properly elliptic, Dirichlet and Neumann boundary conditions cover $P_{\omega,\nu}^{(\tau)}$ on $\partial \Omega$. Notice that we use here the fact that for $\tau \in (-\tau_0,\tau_0)$ the boundary of $\overline{\Omega}_{\tau_0}$ is contained in $(\partial \Omega)_{\mathbb{C}}$.
\end{proof}

We are now ready to prove Theorem \ref{theorem:main}.

\begin{proof}[Proof of Theorem \ref{theorem:main}]
Let us prove first that there is $\delta > 0$ such that for every $\nu > 0$ and $\omega \in (\lambda- \delta,\lambda+\delta) + i [0,+ \infty)$ the operator $P_{\omega,\nu} : H^4(\Omega) \cap H_0^2(\Omega) \to L^2(\Omega)$ is invertible. Due to Lemma \ref{lemma:invariance_invertibility}, we only need to prove that this operator is injective. Hence, let us assume that $u \in H^4(\Omega) \cap H_0^2(\Omega)$ is such that $P_{\omega,\nu} u= 0$. Integration by parts yields:
\begin{equation*}
    0 = \im \langle P_{\omega,\nu}u, u \rangle_{L^2} = \re \omega( \nu \n{\Delta u}_{L^2}^2 + 2 \im \omega \n{\nabla u}_{L^2}^2). 
\end{equation*}
By taking $\delta$ small enough, we ensure that $\re \omega \neq 0$. Since $\nu > 0$ and $\im \omega \geq 0$, we find that $\Delta u = 0$, and since $u$ satisfies Dirichlet boundary condition, it follows that $u = 0$.

We proved the invertibility of $P_{\omega,\nu}$ when $\omega \in (\lambda- \delta,\lambda+\delta) + i [0,+ \infty)$ and $\nu > 0$. In order to end the proof of Theorem \ref{theorem:main}, we only need to prove that there is $\epsilon > 0$ such that for every $\omega \in \mathbb{D}(\lambda,\epsilon)$ and $\nu \in (0,\epsilon)$, the operator $P_{\omega,\nu} : H^4(\Omega) \cap H_0^2(\Omega) \to L^2(\Omega)$ is injective. We will prove it by contradiction, so assume that there are sequences $(\omega_n)_{n \geq 0}$ and $(\nu_n)_{n \geq 0}$ respectively of complex and positive real numbers such that $\omega_n \underset{n \to + \infty}{\to} \lambda$ and $\nu_n \underset{n \to + \infty}{\to} 0$, and for every $n \geq 0$ the kernel of $P_{\omega_n,\nu_n} : H^4(\Omega) \cap H_0^2(\Omega) \to L^2(\Omega)$ is non-trivial.

Let $\tau_0 > 0$ be small enough so that Propositions~\ref{proposition:inviscid_invertibility} and~\ref{proposition:a_priori_estimate} and Lemmas~\ref{lemma:inviscid_index} and~\ref{lemma:invariance_invertibility} hold. Choose then $\tau \in (0,\tau_0)$. By construction of the sequences $(\omega_n)_{n \geq 0}$ and $(\nu_n)$, it follows from Lemma \ref{lemma:invariance_invertibility} that for every $n \geq 0$ there is $u_n \in H^4(\Omega) \cap H_0^2(\Omega)$ non-zero such that $P_{\omega_n,\nu_n}^{(\tau)} u_n = 0$. Let us normalize the $u_n$'s in $L^2(\Omega)$. Applying Proposition \ref{proposition:a_priori_estimate} with $r = 1,s = 2$ and $N = 0$, we find that the sequence $(\n{u_n}_{1,3,\nu_n})_{n \geq 0}$ is uniformly bounded, and thus that $(u_n)_{n \geq 0}$ is uniformly bounded in $H^1_0(\Omega)$. Up to extracting, we may assume that $(u_n)_{n \geq 0}$ converges weakly in $H_0^1(\Omega)$ to some function $u \in H_0^1(\Omega)$. Since weak convergence in $H_0^1(\Omega)$ implies convergence in $L^2(\Omega)$, we find that $\n{u}_{L^2} = 1$, in particular $u$ is non-zero. Moreover, the convergence in distribution implies that $P_{\lambda,0}^{(\tau)} u = 0$. It follows then from Lemmas \ref{lemma:inviscid_index} and \ref{lemma:regularity_first_step} that $u \in H^2(\Omega) \cap H_0^1(\Omega)$, and Proposition \ref{proposition:inviscid_invertibility} implies that $u = 0$, a contradiction.
\end{proof}

\begin{remark}
With $\tau$ as in the proof of Theorem \ref{theorem:main}, we could prove using the a priori estimates from Proposition \ref{proposition:a_priori_estimate} that the inverse $(P_{\lambda,\nu}^{(\tau)})^{-1}$ converges as $\nu$ goes to $0$ to $(P_{\lambda,0}^{(\tau)})^{-1}$ in the strong operator topology of operators from $H^s(\Omega)$ to $H^{s+2}(\Omega)$ for every $s \in (-\frac{3}{2}, - \frac{1}{2})$.

It is very likely that one could establish an asymptotic expansion for the solution $u_\nu$ of the equation $P_{\lambda,\nu}^{(\tau)} u_\nu = f$ as $\nu$ goes to $0$ (where $f$ is a smooth function on $\Omega$). This expansion would involve smooth terms and boundary layers, see e.g. \cite{vishik_ljusternik,volevich_small_parameter}.
\end{remark}

\begin{remark}\label{remark:sixth_order}
The proof of Theorem \ref{theorem:main} could be adapted to deal with the sixth order equation \eqref{eq:sixth_order}. In that case, we consider the family of sixth order operators
\begin{equation*}
    Q_{\omega,\nu} = \partial_{x_2}^2 - \omega^2 \Delta + 2i \omega \nu \Delta^2 + \nu^2 \Delta^3, \quad \omega \in \mathbb{C}, \quad \nu \in \mathbb{R}_+.
\end{equation*}
Here, we replaced $E$ by $\nu$ to be consistent with the rest of the paper. Notice that $Q_{\omega,\nu}$ and $P_{\omega,\nu}$ differ only by the term $\nu^2 \Delta^3$.

Under the assumptions of Theorem \ref{theorem:main}, we can prove that there is $\delta > 0$ such that for every $\omega \in (\lambda- \delta,\lambda + \delta) + i(-\delta, + \infty)$ and $\nu \in (0,\delta)$ the operator $Q_{\omega,\nu} : H^6(\Omega) \cap H_0^3(\Omega) \to L^2(\Omega)$ is invertible. The proof follows the same steps as the proof of Theorem \ref{theorem:main}, the main point being to get, for $\tau > 0$ small, an a priori estimate for the operator $Q_{\omega,\nu}^{(\tau)}$. The estimate for $Q_{\omega,\nu}^{(\tau)}$ is similar to \eqref{eq:a_priori_estimate} except that we have now the norm $\n{\cdot}_{r-2,s-4}$ in the right hand side, and that the conditions on $r$ and $s$ become $\frac{1}{2} < r < \frac{3}{2}$ and $r+ s > \frac{5}{2}$. The a priori estimate for $Q_{\omega,\nu}^{(\tau)}$ may also be deduced from \cite[Theorem 3.5.2]{frank_coercive_singular} if we prove the analogues of Lemmas \ref{lemma:singular_ellpiticity} and \ref{lemma:singular_coercivity}.

To prove the analogue of Lemma \ref{lemma:singular_ellpiticity}, we must now consider for $\tau > 0$ small, $x \in \overline{\Omega},\xi \in \mathbb{R}^2$ and $\nu > 0$ the quantity
\begin{equation*}
    \widetilde{A}_\tau(\nu,x,\xi) = p_{\lambda,0}^{(\tau)}(x,\xi) + i \lambda \nu q^{(\tau)}(x,\xi) + \nu^2 r^{(\tau)}(x,\xi)
\end{equation*}
instead of $A_\tau(\nu,x,\xi)$. Here, $r$ denotes the symbol of $\Delta^3$. One needs to prove that for $\tau > 0$ small enough there is a constant $C > 0$ such that for every $x \in \overline{\Omega}$ and $\xi \in \mathbb{R}^2$ we have
\begin{equation*}
    |\widetilde{A}_\tau(\nu,x,\xi)| \geq C^{-1}|\xi|^2 \brac{\sqrt{\nu}\xi}^4.
\end{equation*}
An efficient way to get this estimate is to use \cite[Remark 2.4.2]{frank_coercive_singular}, which implies that one just needs to prove that $A_\tau(\nu,x,\xi) \neq 0$ for $\nu \in \mathbb{R}_+^*, x \in \overline{\Omega}, \xi \in \mathbb{R}^2 \setminus \set{0}$ and $\tau > 0$ small. It can be achieved by working with the estimates from the proof of Lemmas~\ref{lemma:inviscid_ellipticity} and~\ref{lemma:singular_ellpiticity}. Once we have the analogue of Lemma \ref{lemma:singular_ellpiticity}, the analogue of Lemma \ref{lemma:singular_coercivity} is obtained with a very similar proof.
\end{remark}

\appendix

\section{Elliptic regularity for boundary value problems}\label{appendix:elliptic_regularity}

The goal of this appendix is to prove an elliptic regularity result for elliptic boundary value problems (Lemma \ref{lemma:regularity_first_step}) that is probably standard, but for which we were unable to locate a reference in the literature. The point in the following lemma is that we only assume a priori that the function $u$ belongs to $H_0^1(\Omega)$.

\begin{lemma}\label{lemma:regularity_first_step}
Let $\Omega$ be a bounded domain in $\mathbb{R}^2$ with smooth ($C^\infty$) boundary. Let $P$ be a differential operator of order $2$ on $\overline{\Omega}$ with smooth coefficients. Assume that $P$ is properly elliptic on $\overline{\Omega}$. Let $u \in H_0^1(\Omega)$ be such that $Pu \in L^2(\Omega)$. Then $u \in H^2(\Omega)$.
\end{lemma}

Lemma \ref{lemma:regularity_first_step} is a consequence of the following approximation lemma. This kind of result is standard, going back at least to the work of Friedrichs \cite{friedrichs_44}. See also the proof of \cite[Theorem 17.3.2]{hormander3} for the use of such a result in the proof of a regularity result similar to Lemma \ref{lemma:regularity_first_step}. See also \cite[Lemma E.45]{dyatlov_zworski_book} for the proof of an analogous approximation lemme in the boundaryless case.

\begin{lemma}\label{lemma:regularity_approximation}
Under the assumptions of Lemma \ref{lemma:regularity_first_step} there is a sequence $(u_n)_{n \geq 0}$ of elements of $H^2(\Omega) \cap H_0^1(\Omega)$ such that $(u_n)_{n \geq 0}$ converges to $u$ in $\mathcal{D}'(\Omega)$, the sequence $(u_n)_{n \geq 0}$ is bounded in $H^1(\Omega)$ and the sequence $(Pu_n)_{n \geq 0}$ is bounded in $L^2(\Omega)$.
\end{lemma}

\begin{proof}
Notice that if $\chi \in C^\infty(\overline{\Omega})$ then $P(\chi u) = \chi Pu + [P,\chi] u$ belongs to $L^2(\Omega)$, since $[P,\chi]$ is a differential operator of order $1$. Hence, using a partition of unity we reduce to the case in which the support of $u$ is arbitrarily close to a given point in $\overline{\Omega}$. 

If the support of $u$ is away from $\partial \Omega$, then one may just approximate $u$ by convolution (the point is that the commutator of $P$ with the convolution operator will be of order $1$ uniformly in $n$). See e.g. \cite[Lemma E.45]{dyatlov_zworski_book}.

Hence, we may assume that $u$ is supported near a point of the boundary of $\partial \Omega$. By a change of coordinates, we reduce to the case in which $\overline{\Omega} \cap B(0,1) = \mathbb{R} \times \mathbb{R}_+ \cap B(0,1)$ and $u$ is supported in $\Omega \cap B(0,1/2)$. Let then $\phi : \mathbb{R} \to \mathbb{R}_+$ be a $C^\infty$ function supported in $[-1/3,1/3]$ such that $\int_\mathbb{R} \phi(x) \mathrm{d}x = 1$. For $\epsilon \in (0,1]$, let $u_\epsilon$ be the function supported in $B(0,1) \cap \Omega$ and defined by 
\begin{equation*}
    u_\epsilon(x,y) = \int_\mathbb{R} \epsilon^{-1} \phi(\epsilon^{-1} x') u(x-x',y) \,\mathrm{d}x'
\end{equation*}
for $(x,y) \in B(0,1) \cap \Omega$. Let us write $P$ explicitly
\begin{equation*}
    P = \sum_{\substack{j,k \in \mathbb{N}\\ j + k \leq 2}} a_{j,k} \partial_x^j \partial_y^k.
\end{equation*}
Since $P$ is elliptic the coefficient $a_{0,2}$ does not vanish on $\overline{\Omega}$.  For $\epsilon \in (0,1]$, we let $f_\epsilon$ be the function supported in $B(0,1) \cap \Omega$ and defined by
\begin{equation*}
\begin{split}
    f_\epsilon(x,y) & = a_{0,2}(x,y) \int_\mathbb{R} \epsilon^{-1} \phi(\epsilon^{-1} x') \frac{Pu(x-x',y)}{a_{0,2}(x-x',y)} \,\mathrm{d}x' \\
                    & = \sum_{\substack{j,k \in \mathbb{N}\\ j + k \leq 2}} \underbrace{a_{0,2}(x,y) \int_\mathbb{R} \epsilon^{-1} \phi(\epsilon^{-1} x') \frac{a_{j,k}(x-x',y)}{a_{0,2}(x-x',y)} \partial_x^j \partial_y^k u(x-x',y) \,\mathrm{d}x'}_{ g_{j,k,\epsilon}(x,y) \coloneqq}.
\end{split}
\end{equation*}
for $(x,y) \in B(0,1) \cap \Omega$. We claim that for every $j,k \in \mathbb{N}$ with $j + k \leq 2$, the distribution
\begin{equation}\label{eq:approximation_parts}
    a_{j,k} \partial_x^j \partial_y^k u_\epsilon - g_{j,k,\epsilon}
\end{equation}
belongs to $L^2$, with its norm uniformly bounded as $\epsilon$ goes to $0$. If $j + k \leq 1$, this is just a consequence of Minkowski's inequality, since $u$ is in $H^1$ (in this case both terms are bounded in $L^2$). If $j +k = 2$ but $k \neq 2$, then we may apply Friedrichs' lemma \cite[Lemma 17.1.3]{hormander3} to $\partial_x u \in L^2$ and Fubini's theorem to find that \eqref{eq:approximation_parts} is uniformly bounded in $L^2$. Finally, if $j = 0$ and $k = 2$, the quantity \eqref{eq:approximation_parts} is just zero. Summing \eqref{eq:approximation_parts} over $j$ and $k$, we find that $Pu_\epsilon - f_\epsilon$ belongs to $L^2$, uniformly as $\epsilon$ goes to $0$. By Minkowski's inequality, since $Pu \in L^2(\Omega)$, the function $f_\epsilon$ belongs to $L^2$ uniformly as $\epsilon$ goes to $0$. Hence, $(Pu_\epsilon)_{\epsilon \in (0,1]}$ is bounded in $L^2(\Omega)$.

It follows from the dominated convergence theorem that $u_\epsilon$ converges to $u$ as a distribution when $\epsilon$ goes to $0$ (it is even standard that the convergence holds in $H^1$). Notice that the values of $u_\epsilon$ on $B(0,1) \cap \partial \Omega$ are obtained by averaging $u$ over $B(0,1) \cap \partial \Omega$. Hence, from $u \in H_0^1(\Omega)$, we deduce that $u_\epsilon \in H_0^1(\Omega)$. Using Minkowski's inequality again and differentiation under the integral, we find that $(u_\epsilon)_{\epsilon \in (0,1]}$ is bounded in $H^1(\Omega)$.

It remains to prove that $u_\epsilon$ belongs to $H^2(\Omega)$. Hence, we need to prove that if $j,k \in \mathbb{N}$ are such that $j + k  \leq 2$, then $\partial_x^j  \partial_y^k u_\epsilon \in L^2(\Omega)$. If $j + k \leq 1$, then we just use that $u$ is $H^1$. We can integrate by parts in the formula for $\partial_x u_\epsilon$ to let $\partial_x$ acts on $\phi$ in the integral. Since $u$ is $H^1$, it follows that $\partial_x^j  \partial_y^k u_\epsilon \in L^2(\Omega)$ when $k \neq 2$. Hence,
\begin{equation*}
    \partial_y^2 u_\epsilon =a_{0,2}^{-1} \left( P u_\epsilon - \sum_{\substack{j,k \in \mathbb{N}, \  j + k \leq 2, k \neq 2}} a_{j,k} \partial_x^j \partial_y^k u_\epsilon \right) \in L^2(\Omega),
\end{equation*}
and thus $u_\epsilon \in H^2(\Omega)$.
\end{proof}

\begin{proof}[Proof of Lemma \ref{lemma:regularity_first_step}]
From our assumption on $P$ and $\partial \Omega$, there is a constant $C > 0$ such that for every $v \in H^2(\Omega) \cap H_0^1(\Omega)$ we have
\begin{equation}\label{eq:apriori_H2}
    \n{v}_{H^2} \leq C \n{P v}_{L^2} + C \n{v}_{H^1},
\end{equation}
see for instance \cite[Chapitre 2, Théorème 5.1]{lions_magenes_1}. Notice that Dirichlet boundary condition always covers $\partial \Omega$ for a properly elliptic operator. It follows from \eqref{eq:apriori_H2} that $(u_n)_{n \geq 0}$ is bounded in $H^2(\Omega)$. Since $(u_n)_{n \geq 0}$ converges to $u$ in $\mathcal{D}'(\Omega)$, it follows that $u \in H^2(\Omega)$.
\end{proof}

\bibliographystyle{alpha}
\bibliography{biblio.bib}

\end{document}